\documentclass[a4paper,12pt,oneside,reqno]{amsart}
\usepackage{amssymb,amsfonts,amsmath,amsthm,amscd, bbm}
\usepackage{graphicx, color}
\usepackage{mathrsfs}
\usepackage[shortlabels]{enumitem}
\usepackage{subcaption}
\usepackage{amsthm}
\newtheorem{definition}{Definition}[section]

\numberwithin{equation}{section}
\newcommand{\beq}{\begin{equation}}
\newcommand{\eeq}{\end{equation}}
\newcommand{\bea}{\begin{aligned}}
\newcommand{\eea}{\end{aligned}}
\newcommand{\bdm}{\begin{displaymath}}
\newcommand{\edm}{\end{displaymath}}
\newcommand{\barr}{\begin{array}}
\newcommand{\earr}{\end{array}}
\newcommand{\ben}{\begin{enumerate}}
\newcommand{\een}{\end{enumerate}}
\newcommand{\bde}{\begin{description}}
\newcommand{\ede}{\end{description}}

\newtheorem{teor}{Theorem}
\newtheorem{prop}[teor]{Proposition}
\newtheorem{lem}[teor]{Lemma}

\newtheorem{rem}[teor]{Remark}
\newtheorem{insight}[teor]{Insight}
\newcommand{\R}{\mathbb{R}}

\newcommand{\N}{\mathbb{N}}

\newcommand{\PP}{\mathbb{P}}
\newcommand{\E}{{\mathbb{E}}}

\newcommand{\defi}{\equiv} 

\newcommand{\de}{\delta}

\newcommand{\e}{\epsilon}
\newcommand{\la}{\lambda}

\newcommand{\s}{\sigma}

\newcommand{\lsim}{\lesssim}

\DeclareMathOperator{\arcsinh}{arcsinh}

\newcommand{\ma}{\mathsf a}

\newtheorem{thm}{Theorem}
\newenvironment{thmbis}[1]
  {\renewcommand{\thethm}{\ref{#1}$^\prime$}%
   \addtocounter{thm}{-1}%
   \begin{thm}}
  {\end{thm}}

\begin{document}

\title[Undirected polymers in random environment: mean field limit] {Undirected Polymers in Random Environment: \\ path properties in the mean field limit.} 

\author[N. Kistler]{Nicola Kistler}
\address{Nicola Kistler \\ J.W. Goethe-Universit\"at Frankfurt, Germany.}
\email{kistler@math.uni-frankfurt.de}

\author[A. Schertzer]{Adrien Schertzer}            
\address{adrien schertzer \\ J.W. Goethe-Universit\"at Frankfurt, Germany.}
\email{schertzer@math.uni-frankfurt.de}

\thanks{We are indebted to Lisa Hartung and Marius A. Schmidt for helpful conversations. This work has been partially supported by a DFG research grant, contract number 2337/1-1.}

\date{\today}
\begin{abstract} 
We consider the problem of undirected polymers (tied at the endpoints) in random environment, also known as the unoriented first passage percolation on the hypercube, in the limit of large dimensions. By means of the multiscale refinement of the second moment method we obtain a fairly precise geometrical description of optimal paths, i.e. of polymers with minimal energy. The picture which emerges can be loosely summarized as follows. The energy of the polymer is, to first approximation, uniformly spread along the strand. The polymer's bonds carry however a lower energy than in the directed setting, and are reached through the following  geometrical evolution. Close to the origin, the polymer proceeds in oriented fashion --  it is thus as stretched as possible. The tension of the strand decreases however gradually, with the polymer allowing for more and more backsteps as it enters the core of the hypercube. Backsteps, although increasing the length of the strand, allow the polymer to connect reservoirs of energetically favorable edges which are otherwise unattainable in a fully directed regime. These reservoirs lie at mesoscopic distance apart, but in virtue of the high dimensional nature of the ambient space, 
the polymer manages to connect them through approximate geodesics with respect to the Hamming metric: this is the key strategy which leads to an optimal energy/entropy balance. 
Around halfway, the mirror picture sets in: the polymer tension gradually builds up again, until full orientedness close to the endpoint. The approach yields, as a corollary, a constructive proof of the result by Martinsson [{\it Ann. Appl. Prob.} {\bf 26} (2016), {\it Ann. Prob.} {\bf 46} (2018)] concerning the leading order of the ground state. 
\end{abstract}





\maketitle

\hfill{\it In memory of Dima Ioffe.} \\

\tableofcontents

\section{Introduction}

We denote by $G_n = (V_n, E_n)$ the $n$-dimensional hypercube. $V_n = \{0,1\}^n$ is thus the set of vertices, and $E_n$ the set of edges connecting nearest neighbours. We write $\boldsymbol{0}=(0,0, ... ,0)$ and ${\boldsymbol 1}=(1,1, ... ,1)$ for diametrically opposite vertices. For $l\in \N$ we let
\[
\widetilde \Pi_{n,l} \defi \text{the set of polymers, i.e. paths from \textbf {0} to \textbf {1} 
of length}\; l\,,
\]
as well as 
\[
\widetilde \Pi_{n}\defi \bigcup\limits_{l=1}^{\infty}\widetilde \Pi_{n,l}.
\] 
For $\pi \in \tilde \Pi_{n}$ a polymer going through two vertices $\boldsymbol v, \boldsymbol w$ of the hypercube,  we denote by $l_\pi(\boldsymbol v, \boldsymbol w)$ the length of the connecting substrand, also shortening $l_\pi \defi l_\pi(\boldsymbol 0, \boldsymbol 1)$. \\

\noindent Every edge of the $n$-hypercube is parallel to some unit vector $e_j\in \R^n$, where $e_j$ connects 
\[
(0, \dots, 0) \; \text{and}\; (0, \dots, 0, \underbrace{1}_{j^{th}-\text{coordinate}}, 0, \dots, 0)\,.
\] 
We write $e_{-j}\defi -e_j$. The quantity $\pi_j \in \{1,n\}\cup \{-1,-n\}$ then specifies the direction of a $\pi$-path at step $j$. A {\it forward step} occurs if $\pi_j \in \{1,n\}$; if  $\pi_j \in \{-1,-n\}$ we refer to this as a {\it backstep}. \\

\noindent Remark that the endpoint of the (sub)path $\pi_1 \pi_2\dots \pi_i$ coincides with the vertex given by $\sum_{j\leq i} {e_{\pi_j}}$. The edge traversed in the $j$-th step by the $\pi$-path will be denoted $[\pi]_j$. \\

\noindent To each edge we attach independent, standard (mean one) exponentials random variables $\xi$, the random environment, and assign to a polymer $\pi \in \widetilde \Pi_{n,l}$ its {\it weight/energy} according to 
\[
X_\pi\defi\sum_{j=1}^{l} \xi_{[\pi]_j}.
\]

\noindent The question we wish to address concerns the ground state of undirected polymers in random environment\footnote{This problem also appears in the literature under the name of unoriented first passage percolation, FPP for short.  In mathematical biology it bears relevance to the issue of fitness landscapes. in which case it is dubbed accessibility percolation, see  \cite{Beresticky, Beresticky2, Hegarty, Martinsson1, Martinsson2, Martinsson3, Krug, Krug2} and references therein.  We adopt here the polymer terminology since it is arguably more suitable to convey the type of results we obtain.}, to wit:
\beq \label{one}
m_n \defi \min_{{\pi}\in \widetilde \Pi_{n}}X_{\pi},
\eeq 
in the mean field limit $n\uparrow \infty$, and the statistical/geometrical properties of optimal paths. \\

\noindent A first remark is in place: since polymers with loops cannot achieve the ground state (their energy can always be reduced by removing the loops), we will henceforth focus on the set of {\it loopless} paths of length $l\in \N$,  denoted $\Pi_{n,l}$, and shortening, in full analogy, 
\[
\Pi_n \defi \bigcup\limits_{l=1}^{\infty} \Pi_{n,l}, 
\]
for the set of all loopless paths. 

Looplessness will be very useful: it guarantees, in particular, that the energy of a polymer of length, say, $l$, is indeed given by the sum of $l$ independent standard exponentials. On the other hand, loopless paths are not necessarily directed, see Figure \ref{polymers} below for a graphical rendition.   \\ 

\begin{figure}[!h]
    \centering
\includegraphics[scale=0.3]{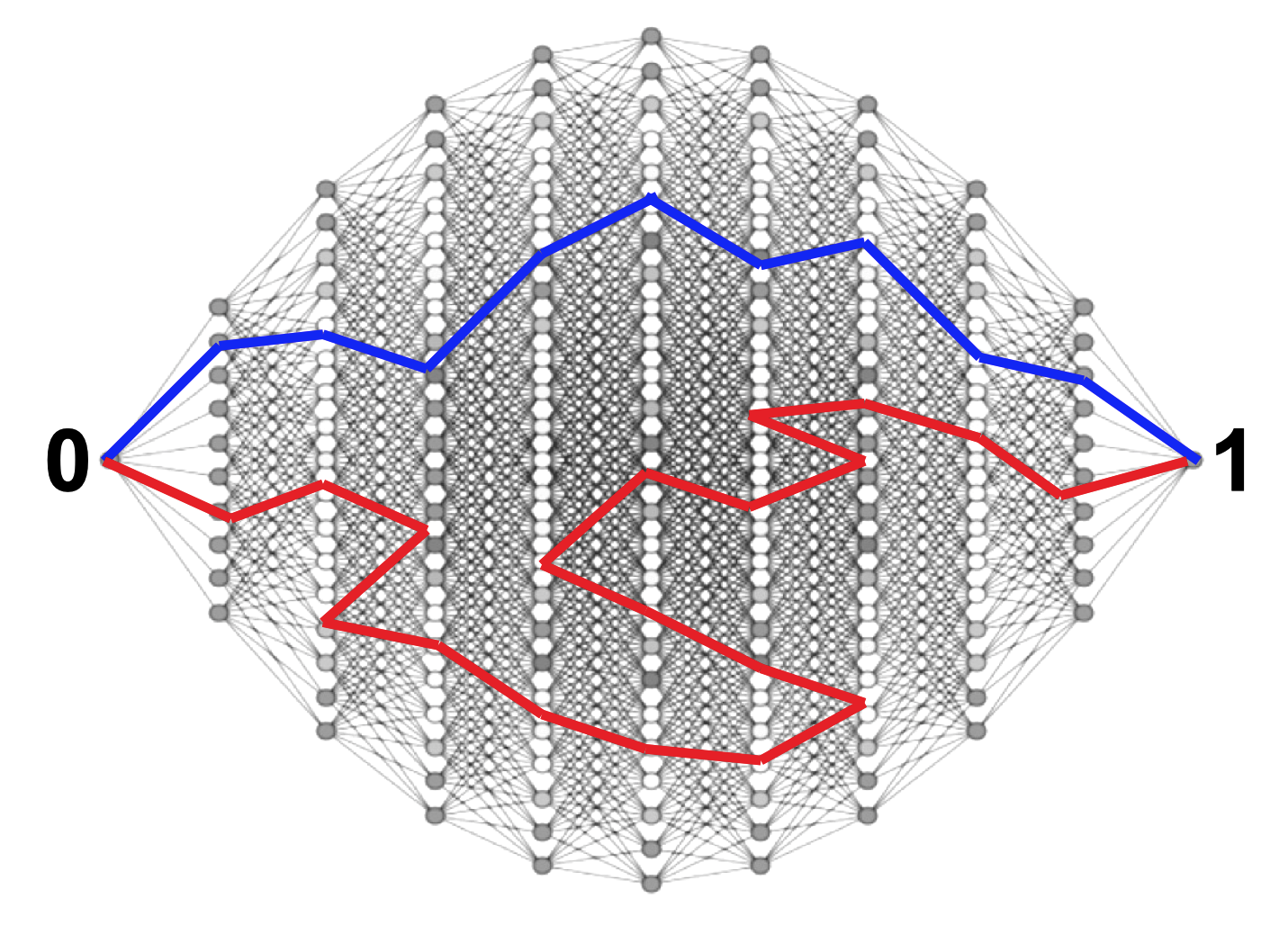}
    \caption{The 10-dim hypercube with two polymers. The blue polymer is {\it directed}: its length coincides with the dimension ($l=n=10$), and it is thus as stretched as possible. The red polymer is {\it undirected}: it performs backsteps, which account for a lower "tension",  and for the long excursions ($l=20$).}
\label{polymers}
\end{figure}

\noindent It is clear that a major issue here will be that of {\it path counting}. For the hypercube, the following beautiful formula is available. We denote by $M_{n,l,d}$ the number of polymers of length $l$ between two points at Hamming distance $d$, i.e. points thus disagree in exactly $d$ coordinates. It then holds :
\beq\bea\label{e4}
M_{n,l,d}=\frac{1}{2^n}\sum_{i=0}^{n}\sum_{j=0}^{d} \binom{d}{j}\binom{n-d}{i-j}{(-1)}^{j}{(n-2i)}^{l}\mathbbm{1}_{j\leq i}.
\eea\eeq 
(This formula concerns all paths of given length: loops, in particular, are also allowed). A proof of this formula, which relies on the classical approach via adjancency matrices, can be found in the monograph by Stanley \cite{stanley}. Since we were not able to identify its first discoverer, we will refer to \eqref{e4} as {\it Stanley's formula}.  

No less remarkable is the following {\it Stanley's identity}, relating $M_{n,l,d}$ to hyperbolic functions. For $x \in \R$, it holds:
\beq\bea\label{e5}
\sum_{l=0}^{\infty}M_{n,l,d}\frac{x^l}{l!} = {\sinh(x)}^{d}{\cosh(x)}^{n-d}\,.
\eea\eeq

Assuming the validity of \eqref{e4}, the proof of \eqref{e5} only requires the binomial theorem and elementary Taylor expansions: it will be given in the Appendix for completeness.  Lightening notations further by setting $M_{n,l} \defi M_{n,l,n}$ for the number of polymers of length $l$ between two opposite vertices on the hypercube, it thus follows from \eqref{e5} that
\beq \label{stanley2}
\sum_{l=0}^{\infty}M_{n,l}\frac{x^l}{l!} = {\sinh(x)}^{n}\,.
\eeq
This relation will allow for precise asymptotical analysis. Before seeing a first, key application, we shall recall yet another technical input concerning {\it tail estimates} for the distribution of the sum of independent standard exponentials as appearing in the problem at hand: denoting by  $\{\xi_i\}_{i \in \N}$ a family of such random variables and with $X_l \defi \sum_{i\leq l} \xi_i$, it then holds:
\beq \label{law_i}
\PP\left( X_l\leq x\right) = \left(1+K(x,l) \right)\frac{e^{- x} x^l}{l!}, 
\eeq
for $x>0$,  and with $0\leq K(x,l)\leq e^{ x} x/(l+1).$ (The proof is truly elementary, but see e.g. \cite[Lemma 5]{kss} for details). \\

Some notational convention: for $a_n, b_n \geq 0$ we write $a_n \lsim b_n$ if $a_n \leq C b_n$ for some numerical constant $C>0$ and $a_n \propto b_n$ if $a_n \lsim b_n$ and $b_n \lsim a_n$ . \\

Armed with Stanley's formula and the tail estimates, we are now ready to make the aforementioned key observation concerning the ground state of undirected polymers: denoting by 
$N_{n,l,x} \defi \#\{\pi \in \Pi_{n,l}, X_{\pi}\leq x\}$ the number of polymers of length $l$ and energies at most $x$, by union bounds and Markov inequality we have
\beq\bea \label{e1}
\PP\left(m_n\leq x\right)=\PP\left(\cup_{l=0}^{\infty} \{N_{n,l,x} \geq 1\}\right)  \leq \sum_{l=0}^{\infty} \E(N_{n,l,x}).
\eea\eeq 
Remark that we are considering polymers with no loops, in which case the energies are indeed sums of $l$ independent random variables. Furthermore, it clearly holds that $\# \Pi_{n,l} \leq M_{n,l}$, since allowing loops can only increase the cardinality\footnote{Here and henceforth we use Stanley's formula although we will be mostly considering loopless polymers: in hindsight, the error/overshooting will turn out to be negligible. This is course due to the high dimensionality of the problem at hand.}. All in all, we have

\beq \bea \label{e11}
\E(N_{n,l,x}) & \leq M_{n,l}\PP\left(X_{l}\leq x\right) \lsim M_{n,l}\frac{x^l}{l!} \,,
\eea \eeq
the second inequality by the tail estimates. 

Performing now the sum over all polymer-lengths in \eqref{e1} and then using \eqref{e5}, we thus obtain 
\beq \label{e12}
\PP\left(m_n\leq x\right) \lsim \sinh(x)^n \,.
\eeq
The $\sinh$-function is increasing, therefore, denoting by 
\beq \label{gs}
\mathsf E \defi \arcsinh(1) = \log(1+\sqrt{2}), 
\eeq
we deduce from \eqref{e12}, and the Borel-Cantelli lemma, a {\it lower bound}  to the ground state, to wit:
\beq
\PP\left( \lim_{n \to \infty} m_n \geq \mathsf E \right) = 1. 
\eeq
As it turns out, this bound is tight.\\

\noindent{\bf Martinsson's Theorem \cite{Martinsson1, Martinsson2}.} {\it For undirected polymers on the hypercube, it holds
\beq  \label{fpp} 
\lim_{n \to \infty} m_n = \mathsf E,
\eeq
in probability.} \\

In other words, a "mean field trivialization" occurs in the limit of large dimensions, and the model of unoriented polymers in random environment thus falls in the so-called {\it REM class} \cite{kistler}. Given the simple derivation of the lower bound, which eventually relies on the Markov inequality only, one is perhaps tempted to tackle the missing upper bound via the Second Moment Method. This is however not the route taken by Martinsson who, in fact, has found {\it two} rather distinct proofs. 

The historically first proof has appeared in \cite{Martinsson1}. In that paper, Martinsson builds upon ideas of Durrett \cite{Durrett} and work by Fill and Pemantle \cite{Fill_Pemantle}, and settles the issue of the upper bound through a delicate comparison with the so-called Branching Translation Process, BTP for short. The BTP is a hierarchical model amenable to an explicit analysis and which, crucially, stochastically dominates the model of unoriented polymers. 

In the second proof of the above theorem, Martinsson proceeds through some ingenious use of the FKG inequality, and (related) subadditivity/monotonicity properties of paths with optimal energies, see \cite{Martinsson2}  for details. 
  
Both proofs naturally come with their own strengths and weaknesses: the first one not only provides a solution of the problem at hand, but also insights into the structure of the BTP which are interesting in their own right, whereas the second proof settles the FPP on Cartesian power graphs, and thus applies in vast generality. 

It seems however fair to say that, by their own nature, both approaches shed little light on the physical phenomena which eventually lead to the mean field trivialization. It is the purpose of this article to fill this gap by providing yet a third proof of the upper bound for the ground state, and hence of Martinsson's Theorem. 

To this end, we will implement the {\it multiscale refinement of the second moment method} \cite{kistler}, a tool which forces us to identify the mechanisms allowing polymers to reach minimal energies. (As will become clear in the treatment, the choice of an exponentially distributed random environment presents no loss of generality).  Unfortunately,  the formulation of our main result, Theorem \ref{fpp1} below, requires not a little infrastructure: this will be provided in the next Section \ref{gathering_insights}. In order the justify (and de-mystify) some otherwise odd looking formulas, concepts, {\it etc.} we will proceed gradually,  increasing the amount of details concerning the geometry of optimal paths through simple observations and elementary computations. The upshot of these findings  will be recorded in the form of {\sf Insights}.  A cautionary note is here due. The computations underlying {\sf Insight \ref{first_insight}-\ref{geodesics}} below are rigorous yet {\it per se} not necessarily conclusive: indeed, they all rely on the {\it existence} of paths with the established geometric properties, but this will be, in fact, the content of Theorem \ref{fpp1} itself. \\

Our new approach leads to a proof of Martinsson's theorem which is much longer than those already available. It does however yield a detailed geometrical description of optimal polymers, and this in turn opens a gateway towards the unsettled issue of fluctuations and weak limits. 

\section{Drawing the picture} \label{gathering_insights}
As we have seen, a reasonable candidate for the ground state eventually follows from an application of the Markov inequality. Albeit crucial, the ground state encodes however only some limited information. Another fundamental quantity is of course the {\it length} of an optimal polymer: as it turns out, a simple computation, allows to make an educated guess.

\subsection{A candidate optimal length} \label{candidate_opt_l} Due to the high dimensionality of the problem, in order to identify the optimal length  it seems  natural to analyze the asymptotics of $\E(N_{n,l, x})$, the expected number of polymers with energies at most $x \in \R_+$, and prescribed length $l \in \N$. To this end, we recall Stanley's identity  \eqref{stanley2} which states that 
\beq \label{restan}
\sum_{l=0}^{\infty}M_{n,l}\frac{x^l}{l!} = {\sinh(x)}^{n}.
\eeq
Restricting to $x>0$ implies that
\beq \label{general}
M_{n,l} \frac{x^l}{l!} \leq \sinh(x)^n\,,
\eeq
and therefore, by optimizing, we obtain, 
\beq\bea\label{general_inf}
M_{n, l} \leq \inf_{x>0}\left[  {\sinh(x)}^{n} \frac{l!}{x^l} \right]\,.
\eea\eeq
Consistently with our terminology, we refer to \eqref{general} and \eqref{general_inf} as {\it Stanley's M-bounds}.

Recall that $N_{n,l,E}$ is the number of paths of length $l$ between two opposite vertices,  and energy at most 
$\mathsf E = \log(1+\sqrt{2})$ as given in \eqref{gs}. By the tail estimates, and the above Stanley's M-bound, we thus have
\beq\bea \label{e2}
\E(N_{n,l,\mathsf E}) \lsim M_{n,l}\frac{\mathsf E^l}{l!}\leq \mathsf E^l\inf_{x>0} \frac{\sinh(x)^n}{x^l} = \mathsf E^l \frac{\sinh(x^*)^n}{{x^*}^l},
\eea\eeq 
where $x^* = x^*(l)$ is the minimizer of the r.h.s. above; taking the derivative of the target function, we see that this is the (unique) solution of 
\beq\bea \label{e3}
\frac{x}{\tanh(x)}=\frac{l}{n}.
\eea\eeq 
At this point one is perhaps tempted to revert the line of reasoning: with the natural candidate for the optimal energy in mind, we choose $x^* \defi \mathsf E$, in which case it follows from \eqref{e3} that $l=\sqrt{2}{\mathsf E}n$, as an elementary computation shows. Changing the order of extremization is of course not quite justified\footnote{One can prove that for all $l\in \N$, and $x^*$ satisfying \eqref{e3}, it holds that
\[
\sinh(x^*)^n \frac{\mathsf E^l}{{x^*}^l}\leq 1,  
\]
with the bound being saturated at $x^*=E$. As a matter of fact, we will prove an even stronger statement, namely that the length of optimal polymers indeed strongly concentrates on $\mathsf L n$, asymptotically in $n$. As we will see, this concentration follows from a key property of the power expansion \eqref{restan}, when evaluated at $x=\mathsf E$: in this case, the $(\mathsf Ln)^{th}$ Taylor-term carries virtually the whole "mass" (whence the saturation). Such a result also provides intriguing clues about the issue of fluctuations, but since it is not instrumental for the rest of the discussion, we postpone the precise formulation, see Proposition \ref{stronger_statement}  below. }, but the upshot turns out to be correct: \vspace{0.2cm}
\begin{center}
\fbox{\begin{minipage}{18em}
\begin{insight} \label{first_insight}
On the $n$-dim hypercube, the (candidate) length of optimal polymers is $$\sqrt{2} \mathsf E n.$$
\end{insight}
\end{minipage}}
\end{center}
\vspace{0.3cm}

\noindent Henceforth, we will shorten
\beq
\mathsf L \defi \sqrt{2} \mathsf E\,,
\eeq 
and always assume, without loss of generality, that $\mathsf Ln \in \N$.

\subsection{Uniform distribution of the energy}
Having found natural candidates for the minimal energy and optimal length, a further question naturally arises: 
\begin{center}
{\it how is an $\mathsf E$-energy distributed along the polymer?}
\end{center}
To formalize, let us consider $\alpha \in [0,1]$, and shorten $\underline{\alpha} \defi 1-\alpha$; furthermore let $\la \in [0,1]$ and similarly shorten $\underline{\la} = 1-\la$. We denote by
\beq\bea
N_{n, \mathsf L n}^{\la, \alpha}:=\# \left\{ \pi \in \Pi_{n,\mathsf Ln}: \;  \sum\limits_{i=1}^{\alpha \mathsf Ln}{\xi_{[\pi]_i}}\leq \lambda \mathsf E, \sum\limits_{i=\underline{\alpha}  \mathsf Ln+1}^{\mathsf Ln}{\xi_{[\pi]_i}}\leq \underline{\lambda} \mathsf E \right\}\,.
\eea\eeq
the number of polymers with the property that an $\la$-fraction of the energy $\mathsf E$ is carried by an 
$\alpha$-fraction of the length  (and similarly for the remaining part of the strand).

\begin{figure}[!h]
    \centering
\includegraphics[scale=0.3]{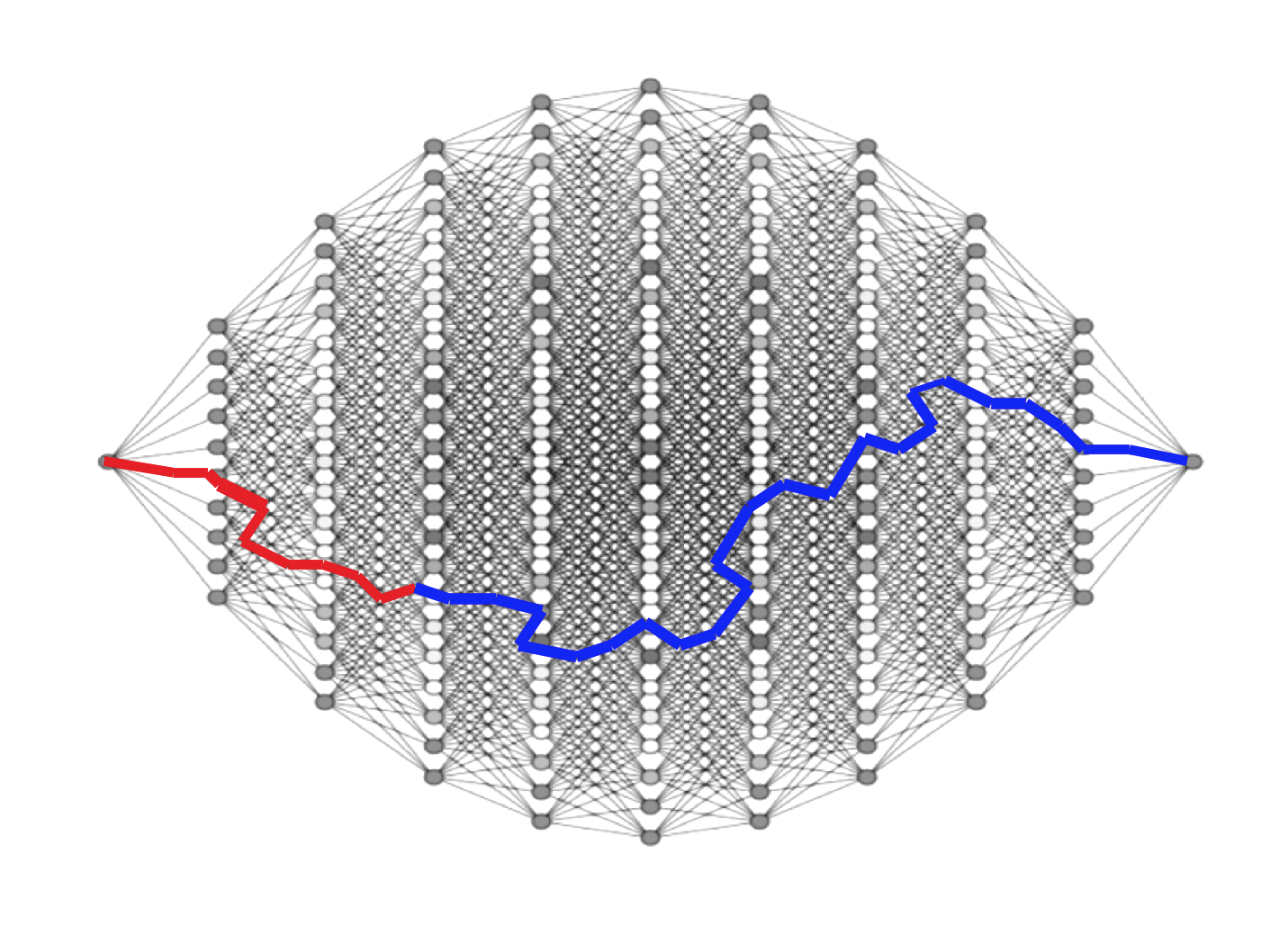}
    \caption{A polymer with $(  \la,  \alpha)$-distribution of the energy $\mathsf E$: the red strand has lenght $\alpha \mathsf Ln$ and carries an energy $\la \mathsf E$, whereas the blue strand has length $\underline{\alpha} Ln$ and carries the remaining energy $\underline{\la} \mathsf E$.}
\label{fig_un}
\end{figure}
Since polymers are loopless, and by independence, we have
\beq\bea \label{first_mom_un}
\E\left( N_{n,\mathsf Ln}^{ \la, \alpha}    \right)&\leq M_{n,\mathsf  Ln}\PP\left (\sum\limits_{i=1}^{\alpha \mathsf  Ln}{\xi_{[\pi]_i}}\leq \lambda \mathsf E, \sum\limits_{i=\underline{\alpha} \mathsf  Ln+1}^{\mathsf  Ln}{\xi_{[\pi]_i}}\leq \underline{\lambda} \mathsf E \right)\\
& = M_{n,\mathsf  Ln}\PP\left(\sum\limits_{i=1}^{\alpha \mathsf  Ln}{\xi_{[\pi]_i}}\leq \lambda \mathsf E \right) \times \PP\left(\sum\limits_{i=\underline{\alpha} \mathsf  Ln+1}^{Ln}{\xi_{[\pi]_i}}\leq\underline{\la} \mathsf E \right)\\
& \lsim M_{n, \mathsf  Ln} \frac{(\lambda \mathsf E)^{\alpha \mathsf  Ln}}{(\alpha \mathsf  Ln)!} \times \frac{ (\underline{\lambda}\mathsf E)^{ \underline{\alpha} \mathsf  Ln} }{(\underline{\alpha} \mathsf  Ln)!} \,,
\eea \eeq
the last inequality by the usual tail estimates. By {\it Stanley's M-bound} \eqref{general}, this time with $x = \mathsf E$, we have
\beq \label{bound_m}
M_{n, \mathsf  Ln} \leq \sinh(\mathsf E)^n \frac{(\mathsf  Ln)!}{\mathsf E^{\mathsf  Ln}} = \frac{(\mathsf  Ln)!}{\mathsf E^{\mathsf  Ln}},
\eeq
the last step since $\sinh(\mathsf E)=1$.  Using this in \eqref{first_mom_un} we thus get
\beq\bea \label{first_mom_un_two}
\E\left( N_{n,\mathsf Ln}^{ \la, \alpha}    \right)
& \lsim \frac{(\mathsf  Ln)!}{\mathsf E^{\mathsf  Ln}}  \frac{(\lambda \mathsf E)^{\alpha \mathsf  Ln}}{(\alpha \mathsf  Ln)!}  \frac{ (\underline{\lambda}\mathsf E)^{ \underline{\alpha} \mathsf  Ln} }{(\underline{\alpha} \mathsf  Ln)!} \\
& = { {\mathsf L} n \choose \alpha  \mathsf L n } (\lambda)^{\alpha \mathsf L n} (\underline{\lambda})^{\underline{\alpha}\mathsf L n},
\eea \eeq
where in the last step we have used that $\mathsf E^{\alpha} \mathsf E^{\underline \alpha} = \mathsf E$, and simplified. By  elementary Stirling approximation  (to first order) of the binomial factor in \eqref{first_mom_un_two},  and again recalling that $\underline{\alpha}=1-\alpha$, and similarly for $\underline{\la}$, we thus arrive at the inequality  
\beq
\E\left( N_{n,\mathsf  Ln}^{ \la,  \alpha}    \right) \lsim {\left\{{\left(\frac{\lambda}{\alpha }\right)}^{\alpha}{\left(\frac{1-\lambda}{1-\alpha}\right)}^{1-\alpha}\right\}}^{\mathsf  Ln}\,.
\eeq
Note that $x  \mapsto  x^y(1-x)^{1-y}$ is strictly concave with a unique critical point at $x=y$. Therefore, $\E N_{n,\mathsf  Ln}^{ \la,  \alpha}$ vanishes exponentially fast as soon as 
$\la \neq \alpha$. Borel-Cantelli then implies the following, loosely formulated summary of the current section: \vspace{0.3cm}

\begin{center}
\fbox{\begin{minipage}{28em}
\begin{insight} \label{uni-energy}
The energy $\mathsf E$ is spread \emph{uniformly} along the polymer.
\end{insight}
\end{minipage}}
\end{center} \vspace{0.2cm}

This insight is of course in complete agreement with the phenomenon of mean field trivialization, see \cite{kistler} for more on this issue.

\subsection{Length vs. distance: the macroscopic picture} We address here the loosely formulated question: 
\begin{center}
{\it at which Hamming distance from the origin \\ do we find a strand of prescribed length?}
\end{center} 

\noindent It is clear that the answer will yield profound insights into the geometry of optimal polymers. To formalize, consider as before $ \alpha \in [0,1]$. (We stick to the convention $\underline{\alpha}=1-\alpha$). For $d \in [0,1]$, let $d_n=\left\lfloor dn\right\rfloor$ and denote by
\beq
H_{d_n}:=\{ \boldsymbol v \in V_n: \; d(\boldsymbol 0, \boldsymbol v)=d_n\}
\eeq 
the {\it hyperplane} consisting of all vertices at Hamming distance $d_n$ from the origin.  (Remark that  
$\sharp H_{d_n} =\binom{n}{d_n}$: indeed, in order to specify a point on the hyperplane we simply need to switch $d_n$ coordinates of $\boldsymbol 0 = (0,0, \dots, 0)$ into 1).

For ${\boldsymbol w} \in H_{d_n}$ we denote by $\Pi_{\alpha \mathsf  Ln}^{d}[\boldsymbol 0 \to {\boldsymbol w}]$ the set of paths connecting $\boldsymbol{0}$ to $\boldsymbol w$ in $\alpha Ln$ steps.  In full analogy, $\Pi_{\underline{\alpha} \mathsf  Ln}^{d}[{\boldsymbol w} \to \boldsymbol 1]$ stands for the set of path connecting $\boldsymbol w$ to $\boldsymbol{1}$ in $\underline{\alpha} \mathsf  Ln$ steps.  Lastly, we denote by $\Pi_{\mathsf  Ln}^{d, \alpha}[\boldsymbol{0} \to \boldsymbol{1}]$ the set of paths of length $\mathsf  Ln$ from  $\boldsymbol{0}$ to  $\boldsymbol{1}$, which are in $H_{d_n}$ after $\alpha\mathsf   L n$ steps. (Note that these paths can cross the hyperplane multiple times, see Figure \ref{pol_unif} below for a graphical rendition).

The goal is now to compute the expected number of these polymers after distributing the energy, in line with the {\sf Insight} from the previous section, {\it  uniformly} along the path. To this end, introduce the cardinalities
\[
N_{n,\mathsf  Ln}^{d,  \alpha}[{\boldsymbol{0} \to \boldsymbol w}]=\#\left\{ \pi\in \Pi_{\alpha \mathsf  L n}^{d}[0 \to \boldsymbol w]:\; \sum\limits_{i=1}^{\alpha \mathsf  L n}{\xi_{[\pi]_i}}\leq \alpha \mathsf E\right\}\,,
\]
\[
N_{n,\mathsf  Ln}^{ d, \underline{ \alpha}}[{\boldsymbol w \to \boldsymbol{1}}]=\#\left\{\pi\in \Pi_{ \underline{\alpha} \mathsf  Ln}^{ d} [\boldsymbol w \to 1], \sum\limits_{i=1}^{\underline{\alpha} \mathsf  Ln}{\xi_{[\pi]_i}}\leq \underline{\alpha}\mathsf  E\right\}\,,
\]
and
\[
N_{n,\mathsf  Ln}^{ d, \alpha}[{\boldsymbol{0}\to \boldsymbol{1}}]=\#\left\{\pi\in\Pi_{\mathsf  Ln}^{ d, \alpha}[\boldsymbol{0} \to \boldsymbol{1}], \sum\limits_{i=1}^{\mathsf  Ln}{\xi_{[\pi]_i}}\leq \mathsf E\right\}.
\]

\begin{figure}[!h]
    \centering
\includegraphics[scale=0.32]{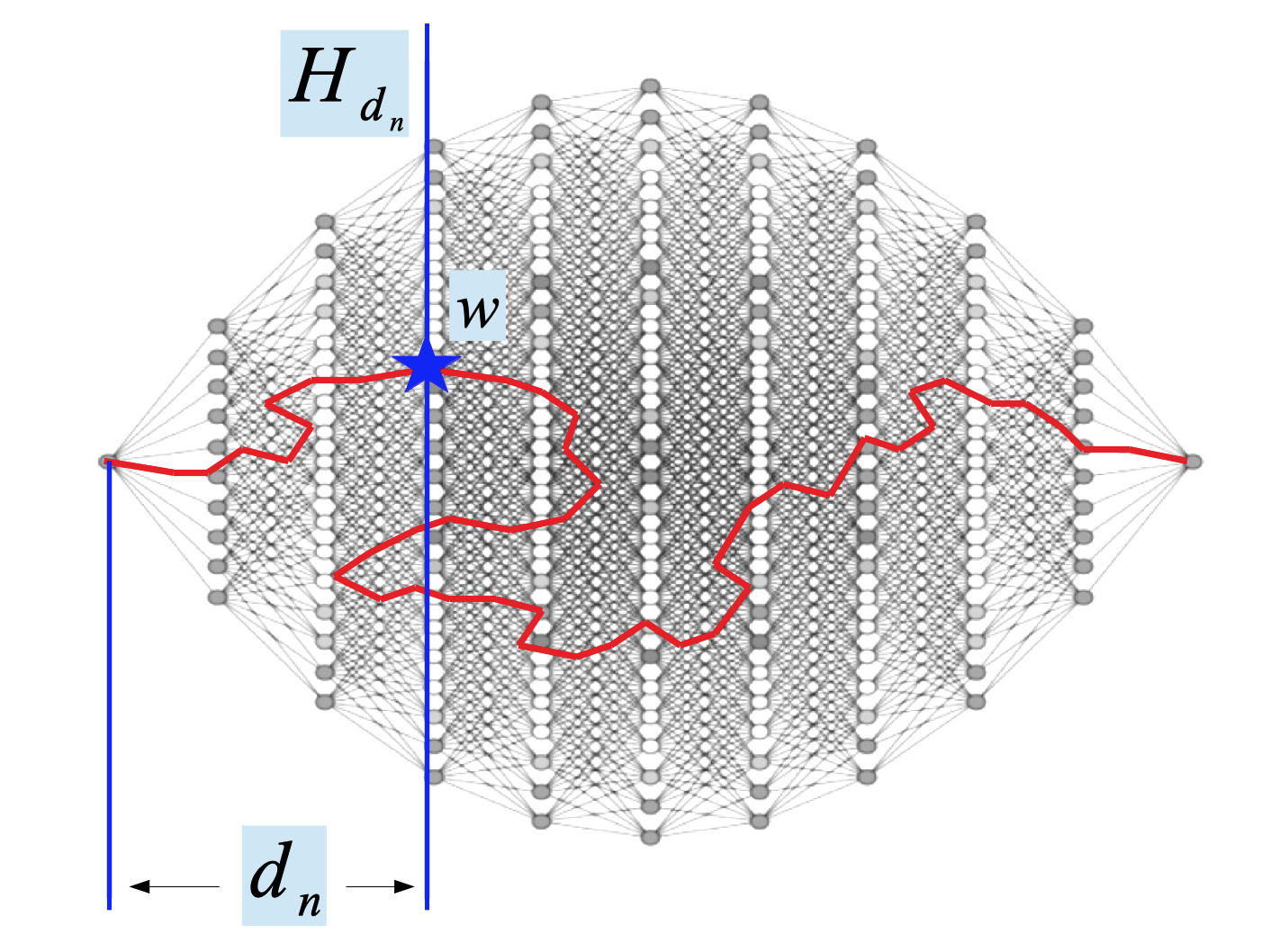}
    \caption{Path-decomposition with an hyperplane $H_{d_n}$ at Hamming distance $d_n$ from $\boldsymbol{0}$. The strand  up to the first crossing of the hyperplane has an $\alpha$-fraction of length, and carries an $\alpha$-fraction of  energy. The rest of the strand has length $\underline{\alpha} \mathsf  Ln$, and carries the remaining $\underline{\alpha}$-fraction of energy.}
\label{pol_unif}
\end{figure}

\noindent Since polymers are loopless, and by independence, it holds\\
\beq\bea\label{fin2}
\E\left( N_{n,\mathsf  Ln}^{ d,  \alpha}[{\boldsymbol{0}\to \boldsymbol{1}}] \right) &= \sum_{\boldsymbol w \in H_{d_n}} \E\left (N_{n,\mathsf  Ln}^{ d,  \alpha}[{\boldsymbol{0} \to \boldsymbol w}] \right) \E\left(N_{n,\mathsf  Ln}^{ d,  \underline{\alpha}}[{\boldsymbol w \to \boldsymbol{1}}]\right)\\
&=\binom{n}{d_n} \E\left (N_{n,\mathsf  Ln}^{ d,  \alpha}[{\boldsymbol{0} \to \boldsymbol w}] \right) \E\left(N_{n,\mathsf  Ln}^{ d,  \underline{\alpha}}[{\boldsymbol w \to \boldsymbol{1}}]\right)\\
&\lsim \binom{n}{d_n} M_{n,\alpha \mathsf  L n,d_n} \frac{{(\alpha \mathsf E)}^{\alpha \mathsf  L n}}{(\alpha \mathsf  L n)!} M_{n,\underline{\alpha} \mathsf  Ln,n-d_n} \frac{({\underline{\alpha} \mathsf E)}^{\underline{\alpha} \mathsf  Ln}}{(\underline{\alpha} \mathsf Ln)!}\,,\\
\eea\eeq
the last inequality by the usual tail estimates. 

In full analogy with \eqref{general_inf}, which is a consequence of Stanley's identity \eqref{stanley2}, the following {\it Stanley's M-bound} is a consequence of Stanley's identity \eqref{e5}: for $x>0$, it holds 
\beq\bea\label{sf}
M_{n, l , d} \leq {\sinh(x)}^{d}{\cosh(x)}^{n-d} \frac{l!}{x^l}\,.
\eea\eeq
Using this for the r.h.s. of \eqref{fin2} we see that for arbitrary $y_1, y_2 >0$, it holds:
\beq\bea\label{fin3}
\E\left( N_{n,\mathsf  Ln}^{ d,  \alpha}[{\boldsymbol{0}\to \boldsymbol{1}}] \right) \lsim \binom{n}{d_n} \frac{{\sinh(y_1)}^{d_n}{\cosh(y_1)}^{n-d_n}}{{\left(\frac{y_1}{\alpha \mathsf E}\right)}^{\alpha L n}}\frac{{\sinh(y_2)}^{n-d_n}{\cosh(y_2)}^{d_n}}{{\left(\frac{y_2}{\underline{\alpha} \mathsf E}\right)}^{\underline{\alpha} \mathsf  Ln}}.
\eea\eeq
Taking $y_1=\alpha \mathsf E$ and $y_2=\underline{\alpha} \mathsf E$, and by elementary Stirling approximation (to first order), 
\beq \label{fin7}
\E\left( N_{n,\mathsf  Ln}^{ d,  \alpha}[{\boldsymbol{0}\to \boldsymbol{1}}] \right)
\lsim {\left(\frac{\cosh(\alpha E)\sinh(\underline{\alpha} E)}{1-\frac{d_n}{n}}\right)}^{n-d_n}{\left(\frac{\sinh(\alpha E)\cosh(\underline{\alpha} E)}{\frac{d_n}{n}}\right)}^{d_n}\,.
\eeq
We will now slightly modify the  form of the r.h.s. above. In order to do so, we recall that
\beq \bea
1 = \sinh(\mathsf E) & = \sinh\left( \alpha \mathsf E + \underline{\alpha}\mathsf E \right) 
\\& = \cosh(\alpha E)\sinh(\underline{\alpha} \mathsf E)+ \sinh(\alpha \mathsf E) \cosh(\underline{\alpha} \mathsf E)\,,
\eea \eeq
the last step by the addition formula for hyperbolic functions, hence 
\beq \bea
\cosh(\alpha \mathsf E)\sinh(\underline{\alpha} \mathsf E) & = 1 - \sinh(\alpha \mathsf E) \cosh(\underline{\alpha} \mathsf E) 
\eea \eeq
This allows to reformulate \eqref{fin7} as 
\beq \bea \label{fin8}
\E\left( N_{n, \mathsf Ln}^{ d,  \alpha}[{\boldsymbol{0}\to \boldsymbol{1}}] \right)  \lsim \left\{{\left(\frac{1-\sinh(\alpha \mathsf E)\cosh(\underline{\alpha} \mathsf E)}{1-\frac{d_n}{n}}\right)}^{1-\frac{d_n}{n}}{\left(\frac{\sinh(\alpha \mathsf E)\cosh(\underline{\alpha} \mathsf E)}{\frac{d_n}{n}}\right)}^{\frac{d_n}{n}}\right\}^n \,.
\eea\eeq
One plainly checks that the function 
\beq
[0,1] \ni \alpha   \mapsto \sinh(\alpha \mathsf E)\cosh(\underline{\alpha} \mathsf E)  
\eeq
is bijective, whereas $x  \mapsto  (1-x)^{1-y}x^y$ is strictly concave with a unique critical point at $x=y$. It thus steadily follows that the r.h.s. of \eqref{fin8} is exponentially small if $\frac{d_n}{n} \neq \sinh(\alpha E)\cosh( \underline{\alpha} E)$. We may thus summarize these findings as follows:  \vspace{0.2cm}

\begin{center}
\fbox{\begin{minipage}{29em}
\begin{insight}\label{insight_alpha_d} After an $\alpha$-fraction of the total length, an optimal polymer  finds itself  
at a typical (normalized) Hamming distance 
\beq \label{quecestbeau}
d = \sinh(\alpha \mathsf E)\cosh((1-\alpha) \mathsf E)
\eeq
from the origin.
\end{insight}
\end{minipage}}
\end{center}  \vspace{0.3cm}

The above {\sf Insight} is both intriguing and delicate. Indeed, a polymer of length greater than the dimension can (must) cross multiple times certain hyperplanes, yet the map $\alpha \mapsto d(\alpha)$ as in \eqref{quecestbeau} is increasing: for consistency, we must therefore deduce that excursions can only happen on mesoscopic (if not microscopic) scales. In other words, and loosely:  \vspace{0.2cm}

\begin{center}
\fbox{\begin{minipage}{27em}
\begin{insight} \label{insight_rare}
Backsteps must be relatively rare, and spread out.
\end{insight}
\end{minipage}}
\end{center} \vspace{0.3cm}

Not surprisingly, this additional {\sf Insight} will play a key role, and guide us through the next steps, but before proceeding any further, a comparison with the directed case is perhaps in place. To better visualize, we re-parametrize in terms of the (normalised) {\it length} of the polymer: with $\alpha E \hookrightarrow l$, and recalling that $\mathsf  L = \sqrt{2} \mathsf E$, we see that the "Hamming depth" $d_{\text{un}}(l)$ reached by the unoriented polymer at length $l$ is then given by 
\beq
l \in [0, \mathsf  L] \mapsto d_{\text{un}}(l) \defi \sinh\left(\frac{l}{\sqrt{2}} \right)\cosh\left( \frac{\mathsf  L-l}{\sqrt{2}} \right)\,.
\eeq
In case of oriented polymers, the Hamming depth as a function of the length is simply 
\beq
l \in [0,1] \mapsto d_{\text{or}}(l) \defi l \,.
\eeq
The two functions are plotted in Figure \ref{or_vs_unor} below, whereas a rendition of the emerging picture at the level of the strands  is given in Figure \ref{strand_evol}.

\begin{figure}[!h]
    \centering
\includegraphics[scale=0.6]{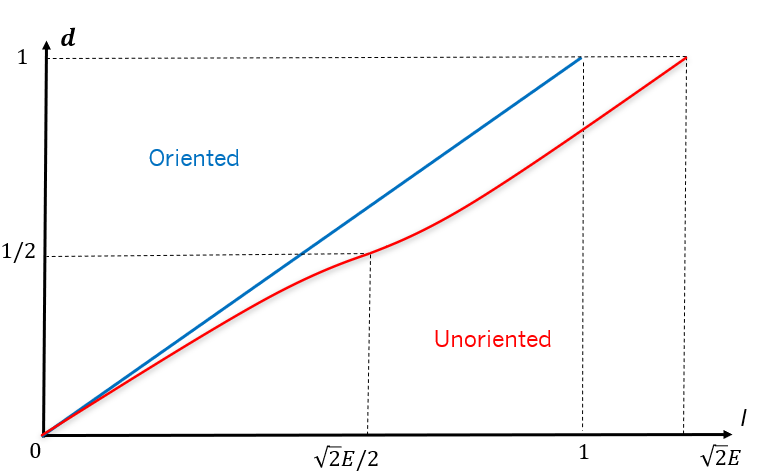}
    \caption{Hamming-depth as a function of the length: directed (blue) vs. undirected (red) polymers. For small lengths, the depths are comparable: close to the origin, the undirected polymer is thus as directed as possible. The  slope of the red curve decreases however gradually as the  polymer approaches the core of the hypercube: the further the polymer goes, the "loser" it becomes. Due to the inherent symmetry of the hypercube,  a mirror picture sets in, of course, at half-length.
}
\label{or_vs_unor}
\end{figure}

\begin{figure}[!h]
    \centering
\includegraphics[scale=0.3]{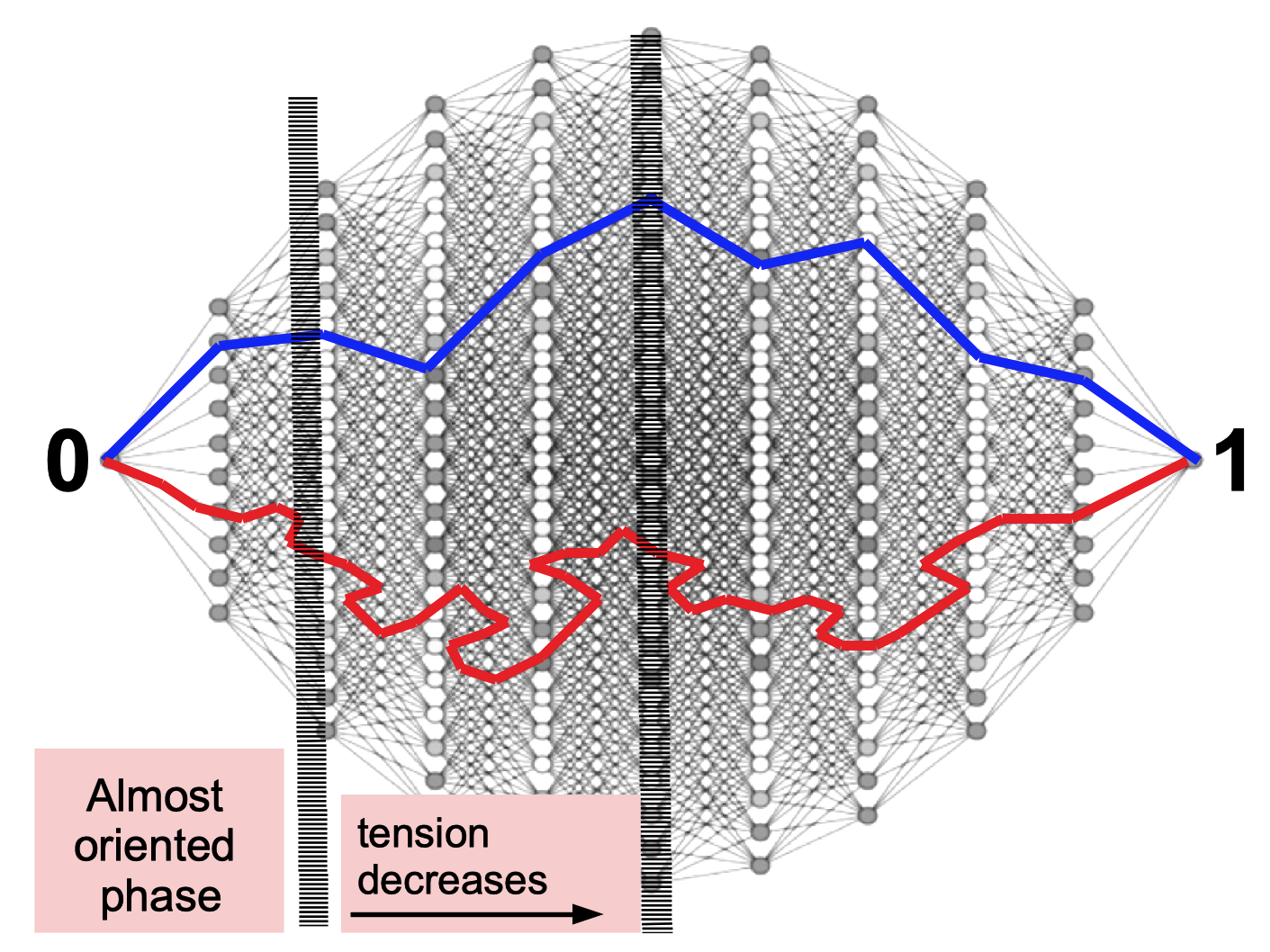}
    \caption{Directed (blue) vs. undirected (red) polymers. The red strand starts off as stretched as possible, but allows for more and more backsteps as it approaches the core of the hypercube. The phenomena are amplified for better visualisation only: in line with {\sf Insight \ref{insight_rare}}, backsteps live on meso/microscopic scale only. In particular, long excursions as in Figure \ref{polymers} above are, in fact, ruled out.
}
\label{strand_evol}
\end{figure}

\newpage

\subsection{Length vs. distance: the mesoscopic picture} \label{meso_sec} As mentioned in the introduction, our approach will eventually rest on a multiscale analysis: in this section, inspired by the previous {\sf Insights}, we introduce the necessary {\it coarse graining} \cite{kistler}. To see how this goes, we denote by $K \in \N$ the numbers of "scales", and shorten henceforth $\hat n_K \defi n/K$ (assuming w.l.o.g. that $\hat n_K \in \N$). We then split the hypercube into $K$ "slabs", i.e. hyperplanes equidistributed w.r.t. the Hamming distance: for $ i=1\dots K$ we let
\beq
H_i \defi \left\{v \in V_n, d(0,v)=  i \hat n_K \right\}\,.
\eeq
We will refer to these hyperplanes as {\it $H$-planes}. Accordingly, we split a polymer of length $\mathsf  L n$ into $K$ substrands of length $\alpha_i Ln$, for $i=1\dots K$, with the normalization $\sum_{i\leq K} \alpha_i=1$.  We shorten $\boldsymbol \alpha =(\alpha_1, \alpha_2, ..., \alpha_K) \in [0,1]^K$ for such a vector, $\overline{\boldsymbol \alpha}_i \defi \sum_{j=1}^i \alpha_j$ for the (fraction of) length of the strand when the polymer crosses the $i^{th}$ H-plane, and 
$\underline{\boldsymbol \alpha}_i \defi 1-\sum_{j=1}^i \alpha_j$ for the length of the remaining strand. A graphical rendition is given in Figure \ref{coarse} below.  \\

\begin{figure}[!h]
    \centering
\includegraphics[scale=0.3]{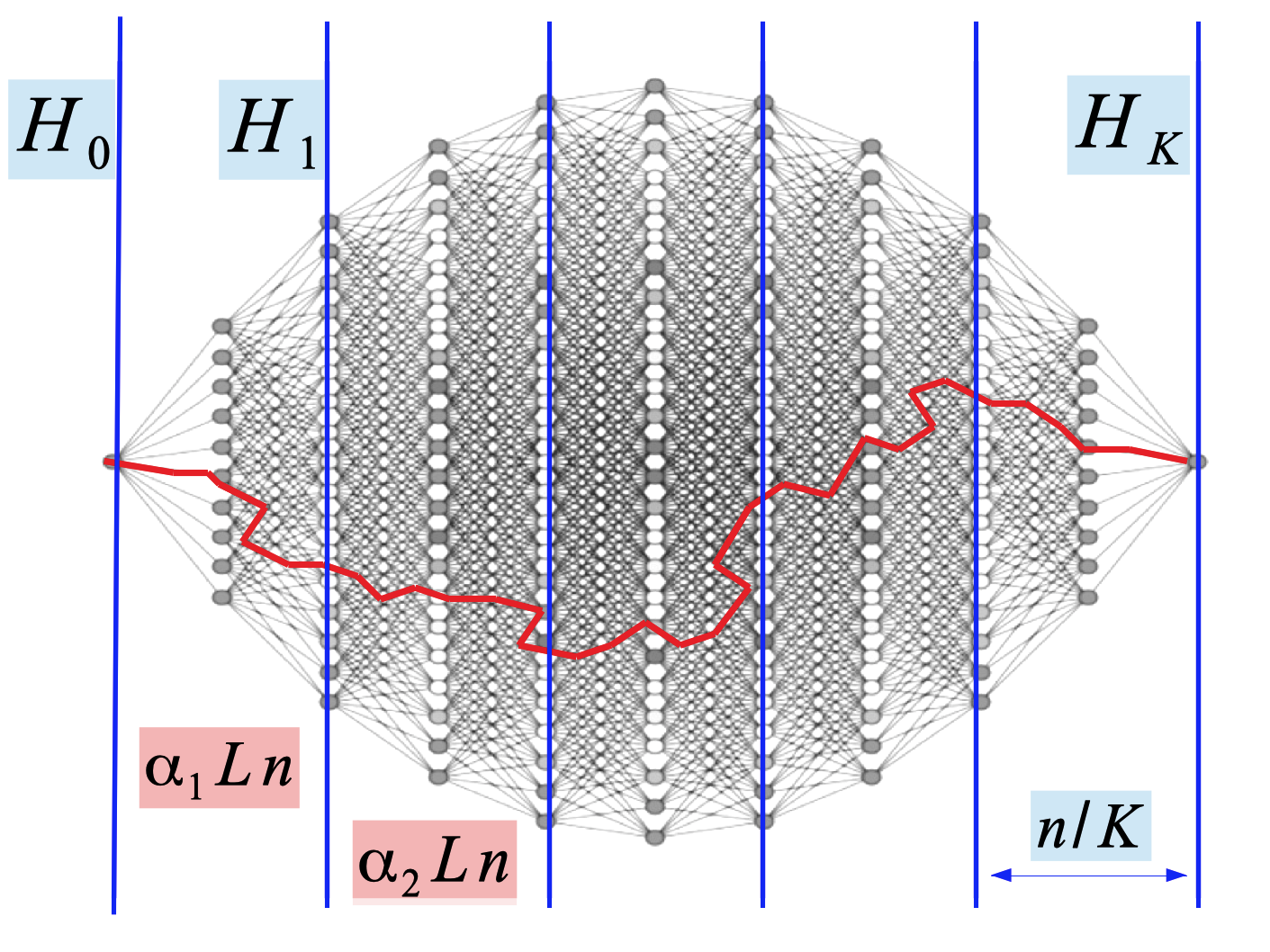}
    \caption{$K$-levels coarse graining: the Hamming distance between any two (successive) hyperplanes is $\hat n_K = n/K$. Remark that by \eqref{if_one}-\eqref{if_two}, the length of the substrand from hyperplane to hyperplane is a function of $\mathsf E$ and $K$ only.}
\label{coarse}
\end{figure}

By the above {\sf Insight \ref{insight_alpha_d}}, length of substrands and Hamming-depth must satisfy the fundamental relation 
\beq\label{fundamental}
\sinh( \overline{\boldsymbol \alpha}_i \mathsf  E)\cosh(  \underline{\boldsymbol \alpha}_i \mathsf E)=\frac{i}{K}\,, \quad i=1\dots K.
\eeq 
The function $x \in [0,1] \mapsto \sinh( x \mathsf E)\cosh( (1-x) \mathsf E)$ is invertible, and one can even construct explicitely the solutions of the above equation: recalling that $\arcsinh(x)=\log(x+\sqrt{1+x^2})$ one plainly checks that  these are given by
\beq \label{if_one} 
\overline{\boldsymbol \alpha}_i=\frac{1}{2}\left\{ 1+\frac{1}{\mathsf E} \arcsinh\left(2\frac{i}{K}-1\right)\right\}\,.
\eeq
This also uniquely identifies the length of the substrands, to wit:
\beq \label{if_two} 
\alpha_i = \overline{\boldsymbol \alpha}_i- \overline{\boldsymbol \alpha}_{i-1}\,,
\eeq
for $i=1\dots K$, see Figure \ref{alpha(i)} below for a plot. 

In particular, it follows from \eqref{if_one} and \eqref{if_two} that 
\beq \label{sym_k}
\alpha_i=\alpha_{K+1-i},
\eeq 
which is in full agreement with the inherent symmetry of the problem at hand, and $\sum_{j\leq K} \alpha_j=1$. Furthermore, since $\arcsinh$ is $1$-Lipschitz we also immediately see that 
\beq \label{control_alpha}
\alpha_i\leq \frac{1}{K\mathsf E}. 
\eeq

In order to emphasize that the $\alpha's$ are no longer arbitrary, we will write  henceforth $\ma = \ma(\mathsf E, K)$ for the solutions of the equations \eqref{if_one}, \eqref{if_two}.

\begin{figure}[!h]
    \centering
\includegraphics[scale=0.6]{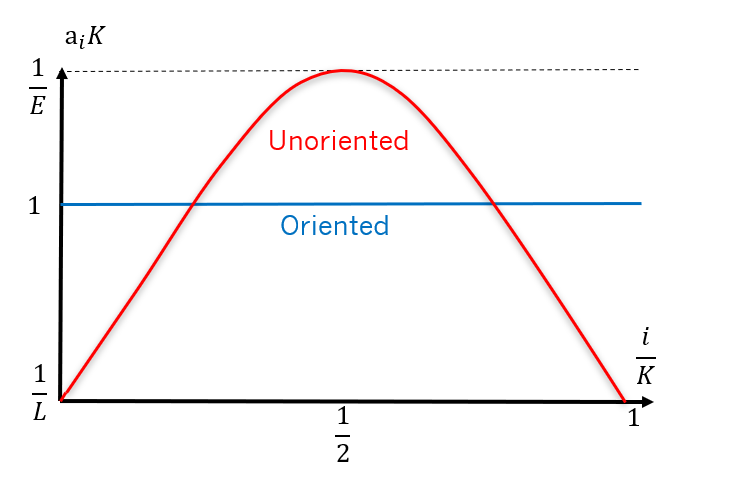}
\caption{Substrand-length as function of the depth, $i \in \{1,\dots K\} \mapsto \ma_i$.
This plot simply restates the key property of optimal polymers: substrands between equidistant hyperplanes become longer as the polymer enters the core of the hypercube. 
}
\label{alpha(i)}
\end{figure}
A straightforward large-$K$ Taylor expansion (with $i/ K = const.$) yields that 
\beq
\ma_{i+1} - \ma_{i} = \frac{2}{K^2 \mathsf E} \left( 1- \frac{2i}{K}\right) + O\left( \frac{1}{K^3}\right)\,,
\eeq
which is manifestly different from the case of directed polymers, where the differential would necessarily vanish. Thus a fundamental question immediately arises: 
\begin{center}
{\it how do substrands of undirected polymers connect \\ the  coarse graining-hyperplanes?} 
\end{center}

\noindent To shed light on this issue we consider $\boldsymbol d =( d_1, d_2, ..., d_K) \in [0,1]^K$ and introduce
\beq \bea \label{looplesspaths_pi_i}
\Pi_{i}^{\boldsymbol d}[\boldsymbol v  \to \boldsymbol w ] \defi\;  & \text{\sf  all loopless paths connecting}\\
& \text{\sf  two vertices} \; \boldsymbol v \in H_{i-1}, \, \boldsymbol w \in H_{i} \\
& \text{\sf  which are at Hamming distance}\; d(\boldsymbol v, \boldsymbol w)=d_i n\,,
\eea \eeq
and
\beq \bea \label{looplesspaths_pi}
 \Pi_{\{1\dots K\}}^{\boldsymbol d}[\boldsymbol 0  \to \boldsymbol 1]  \defi & \; \text{\sf  all loopless paths connecting}\; \boldsymbol 0 \; \text{to}\; \boldsymbol 1\;, \\
& \text{\sf  and that cover a}\; d_i n \text{\sf -Hamming distance}\\
& \text{\sf while connecting the H-hyperplanes,}\; i=1\dots K. \\
\eea \eeq
A graphical rendition is given in Figure \ref{paths_di} below. 
\begin{figure}[!h]
    \centering
\includegraphics[scale=0.35]{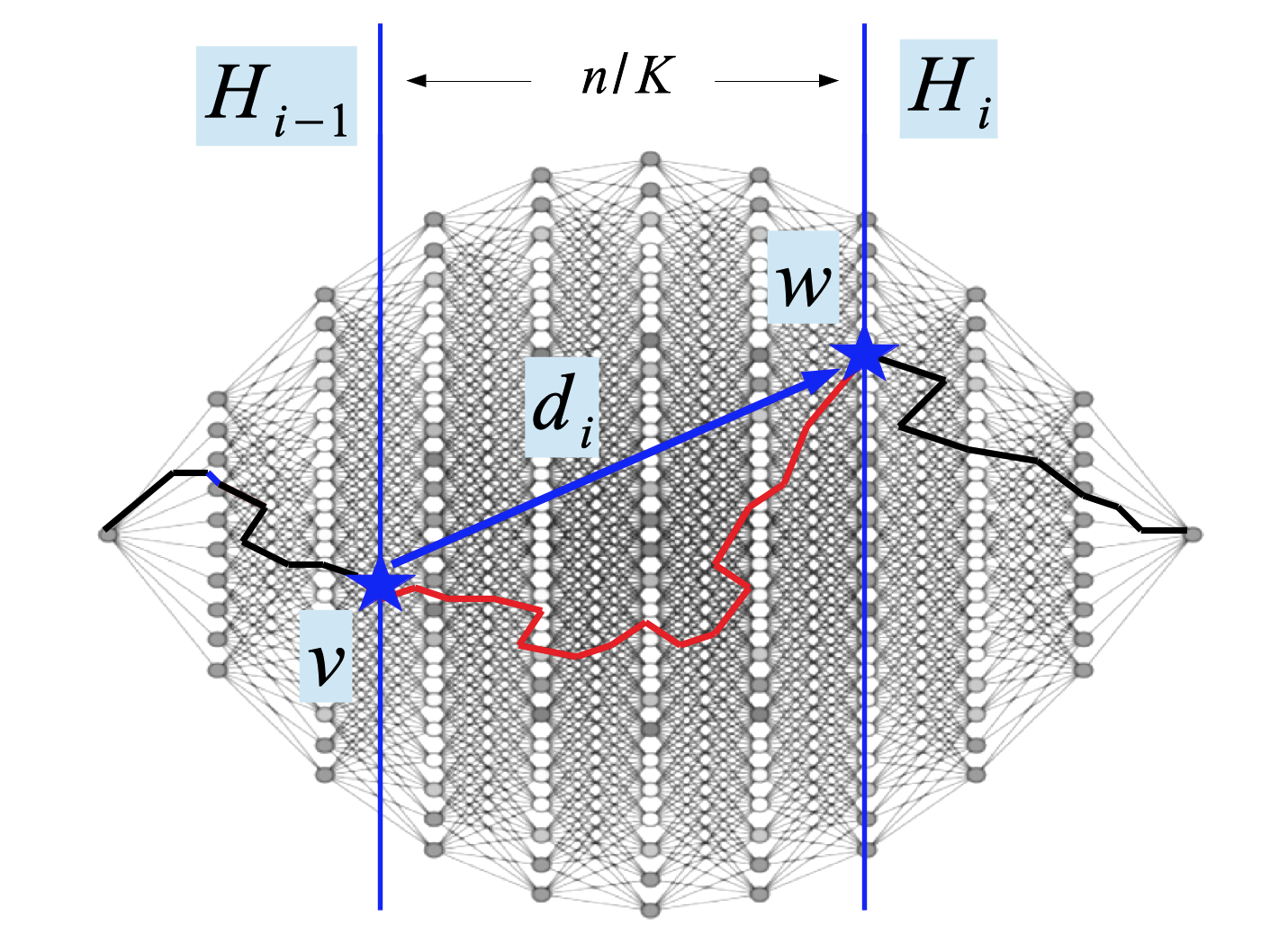}
\caption{A polymer between two hyperplanes: the two vertices $\boldsymbol v$ and $\boldsymbol w$ are at a  Hamming distance $d(\boldsymbol v, \boldsymbol w) = d_i$. Remark that, in particular, $d(H_{i-1}, H_i) = 1/K \leq d_i \leq l_\pi(\boldsymbol v, \boldsymbol w)$.}
\label{paths_di}
\end{figure} \\

\noindent For $\pi \in  \Pi_{\{1\dots K\}}^{\boldsymbol d}[\boldsymbol 0  \to \boldsymbol 1]$, and two vertices 
$\boldsymbol v \in H_{i-1}, \boldsymbol w \in H_i$ (for some $i=1\dots K$), we furthermore shorten 
\beq \bea
X_{\pi}(\boldsymbol v, \boldsymbol w)  \defi\;  & \text{energy of the substrand which connects} \; \boldsymbol v, \, \boldsymbol w\,\,.
\eea \eeq
and denote by 
\beq \label{going_vw}
N_{i}^{\boldsymbol d}[\boldsymbol v \to \boldsymbol w]=\#\left\{ \pi\in \Pi_{i}^{\boldsymbol d}[\boldsymbol v  \to \boldsymbol w ], X_{\pi}(\boldsymbol v, \boldsymbol w) \leq \ma_i \mathsf E\right\}\,,
\eeq
the number of substrands with energies at most $\ma_i E$ connecting such vertices. Finally, let 
\beq \label{going_01}
N_{\{1\dots K\}}^{\boldsymbol d}[{\boldsymbol{0}\to \boldsymbol{1}}]=\#\left\{\pi\in\Pi_{\{1\dots K\}}^{\boldsymbol d}[\boldsymbol 0  \to \boldsymbol 1],\  X_{\pi}(\boldsymbol 0, \boldsymbol 1)\leq \mathsf E\right\}
\eeq
stand for the number of paths with prescribed evolutions\footnote{We shall perhaps emphasize that the above prescription of the evolution involves the Hamming-depths and energies, but {\it not} the length of the connecting substrands. This is because in \eqref{going_vw} we are spreading the energies uniformly along the length of the polymer, very much in line with {\sf Insight \ref{uni-energy}}: energies and optimal lengths are two sides of the same coin.}. The goal is to compute the expectation of this random set, as this will provide fundamental insights into the possible choices of $\boldsymbol d$, which are the only degrees of freedom left. As we will see shortly, there is only one reasonable choice. Before that we need however to introduce some key concepts.

\begin{definition} \label{key_concepts} Let ${\boldsymbol v} \in H_{i-1}$ and $\boldsymbol w \in H_{i}$.
\begin{itemize}
\item The \emph{effective forward steps} are given by 
\[
\mathsf{ef}_i(\boldsymbol v, \boldsymbol w) \defi \frac{1}{n} \# \left\{ 0's\; \text{in}\; \boldsymbol v \; \text{which switch into}\; 1's\; \text{in}\; \boldsymbol w \right\}
\]
\item The \emph{effective backsteps} are given by 
\[
\mathsf{eb}_i(\boldsymbol v, \boldsymbol w) \defi \frac{1}{n} \# \left\{ 1's\; \text{in}\; \boldsymbol v \; \text{which switch into}\; 0's\; \text{in}\; \boldsymbol w \right\}\,.
\]
\item The \emph{detours} are given by
\[
\gamma_{\pi}(\boldsymbol v, \boldsymbol w) \defi \frac{1}{n}\left\{ l_{\pi}(\boldsymbol v, \boldsymbol w)-d(\boldsymbol v, \boldsymbol w)\right\}.
\]
\end{itemize}
\end{definition}
Some comments concerning the above terminology are perhaps in place: we note that the {\it effective forward steps} encode the fraction of steps forward which are not undone by backsteps in the reverse direction; similarly, the {\it effective backsteps} encode the (fraction of)  backsteps which are not undone by steps forward in the reverse direction (or vice versa). Finally, the {\it detours}  capture the amount of forward steps in a path $\pi$ which are cancelled by backsteps in the reverse direction (or vice versa): the smaller $\gamma_\pi$, the higher the "tension" of the substrand. For this reason, we call a substrand  {\it stretched} if the detours vanish. A stretched path is, in fact, a {\it geodesic}. \\

The above quantities are all intertwined. Indeed, it holds:
\beq\bea \label{b-f}
d_i= \mathsf{ef}_i( \boldsymbol v, \boldsymbol w)+\mathsf{eb}_i( \boldsymbol v, \boldsymbol w) \quad \text{and} \quad \frac{1}{K}=\mathsf{ef}_i( \boldsymbol v, \boldsymbol w)-\mathsf{eb}_i( \boldsymbol v, \boldsymbol w)\,.
\eea\eeq 
In particular, it follows from the above relations that 
\beq\bea \label{elementaire}
 \mathsf{ef}_i( \boldsymbol v, \boldsymbol w) =\frac{d_i}{2}+\frac{1}{2K} \quad \text{and} \quad  \mathsf{eb}_i( \boldsymbol v, \boldsymbol w)=\frac{d_i}{2}-\frac{1}{2K}\,.
\eea\eeq 
In other words, effective forward- and backsteps along a substrand depend on the number of scales, and the remaining degrees of freedom $\boldsymbol d$ (which we are going to identify shortly), but {\it not} on the endpoints.  An equally simple line of reasoning shows that detours, {\it as soon as the polymer-length is specified}, do not depend on the specific form of the $\pi$-path, neither: in fact, $\gamma_{i,\pi}n+d_i n=\ma_i  L n$. \\

As mentioned, the goal is to compute the expected number of paths connecting $\boldsymbol 0$ to $\boldsymbol 1$. Since polymers are loopless, and by independence, it holds: 
\beq \label{sum_uno}
\E \left(N_{\{1\dots K\}}^{\boldsymbol d}[{\boldsymbol{0}\to \boldsymbol{1}}] \right)= \sum_{(\star)} \prod_{i=1}^K \E N_i^{\boldsymbol d}\left[ \boldsymbol v^{(i-1)} \to \boldsymbol v^{(i)} \right],
\eeq
where the $(\star)$-sum runs over all possible vertices $\boldsymbol v^{(i)} \in H_i, i=1\dots K$.  But by \eqref{elementaire}, none of the expectations on the r.h.s. depend on the specific $\boldsymbol v$-choice.
The cardinality of $(\star)$ is easily computed: shortening 
\beq \label{phi_function}
[0, \infty) \ni x \mapsto \varphi(x) \defi x^x,
\eeq
one plainly checks that
\beq \bea \label{cardi}
\# (\star)  & =   \prod_{i=1}^{K} \dbinom{\frac{i-1}{K} n}{\mathsf{eb}_i n} \dbinom{\left(1-\frac{i-1}{K}\right) n}{\mathsf{ef}_i n}\\
& \lesssim \prod_{i=1}^{K} \left\{ \frac{\varphi\left( \frac{i-1}{K}\right) \varphi\left( 1-\frac{i-1}{K}\right)}{
\varphi(\mathsf{eb}_i)\varphi\left(\frac{i-1}{K}-\mathsf{eb}_i\right) \varphi\left(\mathsf{ef}_i\right) \varphi\left(1-\frac{i-1}{K}-\mathsf{ef}_i\right)} 
\right\}^n \,,
\eea \eeq
the last step by elementary Stirling-approximation to first order. 

By the tail estimates, and Stanley's M-bound \eqref{sf} with $x = \ma_i \mathsf E$ , it holds  
\beq \bea \label{sameoldstory}
\E N_i^{\boldsymbol d}\left[ \boldsymbol v^{(i-1)} \to \boldsymbol v^{(i)} \right]  \lesssim {\sinh(\ma_i \mathsf E)}^{ d_i n} {\cosh(\ma_i \mathsf E)}^{ (1-d_i)n}\,,
\eea \eeq
for $i=1 \dots K$. 

Plugging \eqref{cardi} and \eqref{sameoldstory} into \eqref{sum_uno}, and rearranging, we thus get the {\it upperbound}
\beq\bea\label{presque}
\E \left(N_{\{1\dots K\}}^{\boldsymbol d}[{\boldsymbol{0}\to \boldsymbol{1}}] \right)
&\lesssim \mathcal F_{\ma, K}(\boldsymbol d)^n,
\eea\eeq
where we have shortened 
\beq
\mathcal F_{\boldsymbol{\ma}, K}(\boldsymbol d) \defi  \prod_{i=1}^{K}\frac{{\sinh(\ma_i E)}^{d_i} {\cosh(\ma_i E)}^{(1-d_i)} \varphi\left(\frac{i-1}{K}\right)\varphi\left(1-\frac{i-1}{K}\right)}{\varphi(\mathsf{eb}_i)\varphi\left(\frac{i-1}{K}-\mathsf{eb}_i\right) \varphi\left(\mathsf{ef}_i\right) \varphi\left(1-\frac{i-1}{K}-\mathsf{ef}_i\right)} \,.
\eeq

Since $\boldsymbol \ma = \boldsymbol \ma(\mathsf E, K)$ are solutions of \eqref{if_one}-\eqref{if_two}, the $\boldsymbol d's$ appearing in the $\mathcal F$-function are the only degrees of freedom left (By \eqref{elementaire}, we recall that $\mathsf{ef}_i$ and $\mathsf{eb}_i$ are function of $d_i$).
The next result shows that even for these, there is  in fact  one reasonable choice only.
\begin{thm}\emph{(Optimal Hamming distance)} \label{choosing_d} Let $\boldsymbol{\mathsf{d}} = (\mathsf d_1, \dots, \mathsf d_K)$, with 
\beq \label{d_gone}
\mathsf d_i \defi \sinh(\ma_i \mathsf E)\cosh((1-\ma_i)\mathsf  E).
\eeq
It then holds:
\beq \label{crucial_one}
\mathcal F_{\boldsymbol{\ma}, K}(\boldsymbol{\mathsf{d}}) = 1, 
\eeq
and
\beq \label{crucial_two}
\mathcal F_{ \boldsymbol{\ma}, K}(\boldsymbol d)< 1, \quad \text{for}\quad \boldsymbol{d} \neq \boldsymbol{\mathsf{d}}.
\eeq
\end{thm}
By \eqref{presque} and \eqref{crucial_two}, the expected number of polymers connecting a sequence of prescribed vertices on the $H$-planes is thus exponentially small, {\it unless} the Hamming distance of the considered vertices satisfies \eqref{d_gone}: of course, the latter will henceforth be the value of our choice. \\

Theorem \ref{choosing_d} is absolutely crucial for our approach. The proof, which requires a fair amount of work, is postponed. For the remaining part of this section we dwell rather informally on some of its far-reaching implications.\\

 We anticipate that we will eventually consider a large (yet finite) number of scales for the coarse graining, in which case an elementary large-$K$ Taylor expansion (together with the fact that  $\mathsf L = \sqrt{2} \mathsf E$) shows
that to first approximation, Hamming distance between two vertices on the H-planes and substrand-legth do, in fact, coincide:
\beq\bea \label{surprise}
\mathsf d_i =&\sinh(\ma_i \mathsf E)\cosh((1-\ma_i) \mathsf E) = \ma_i \mathsf L +O\left(\frac{1}{K^2}\right)\,.
\eea\eeq
A minute's thought suggests that the above may be reformulated as follows: \\
\begin{center}
\fbox{\begin{minipage}{27em}
\begin{insight} \label{geodesics}
Optimal polymers connect the coarse graining H-planes through essentially stretched paths. 
\end{insight}
\end{minipage}}
\end{center} \vspace{0.3cm}

This is a somewhat surprising feature, which at first sight may even appear non-sensical.  The devil is however in the details: by \eqref{if_one}, and large-$K$ Taylor expansions (again with $i/K=const$), one can check that 
\beq
\ma_i = \frac{1}{K\mathsf E\sqrt{1+(\frac{2i}{K}-1)^2}} +O\left(\frac{1}{K^2}\right) \,,
\eeq
which combined with \eqref{surprise}, and recalling $\mathsf L = \sqrt{2}\mathsf  E$, leads to
\beq \label{d_dev}
\mathsf{d}_i = \frac{\sqrt{2}}{K\sqrt{1+\left(\frac{2i}{K}-1\right)^2}} + O\left(\frac{1}{K^2}\right)\,.
\eeq 
From this we may evince that:
\begin{itemize}
\item for small $i$ (say $i=s K$, and $s \ll 1/2$) it holds that 
\beq \label{small_i_a}
\mathsf a_i =  \frac{1}{K \mathsf E \sqrt{2}} + O\left( \frac{1}{K^2}\right) = \frac{1}{K \mathsf L} +  O\left( \frac{1}{K^2}\right)\,, 
\eeq
as well as 
\beq \label{small_i_d}
\mathsf{d}_i = \frac{1}{K} + O\left( \frac{1}{K^2}\right) = d(H_{i-1}, H_i) +O\left( \frac{1}{K^2}\right)\,,
\eeq
the latter confirming that close to the origin, unoriented polymers proceed in almost directed fashion;
\item for large $i$ (say $i= s K$, and $s \uparrow 1/2$) it holds that $\mathsf d_i \approx \sqrt{2}/K \gg 1/K$, which is much larger than the Hamming distance between two successive H-planes. Substrands of optimal polymers close to the core of the hypercube therefore reach, through approximate geodesics, vertices which are otherwise unattainable in a fully directed regime. Although the length of the substrand is increased, this strategy allows undirected polymers to gain access to a reservoir of energetically favorable  edges. A graphical rendition of this feature, which encodes the key strategy of optimal polymers, is given in Figure \ref{cone} below.
\end{itemize}

\begin{figure}[!h]
    \centering
\includegraphics[scale=0.37]{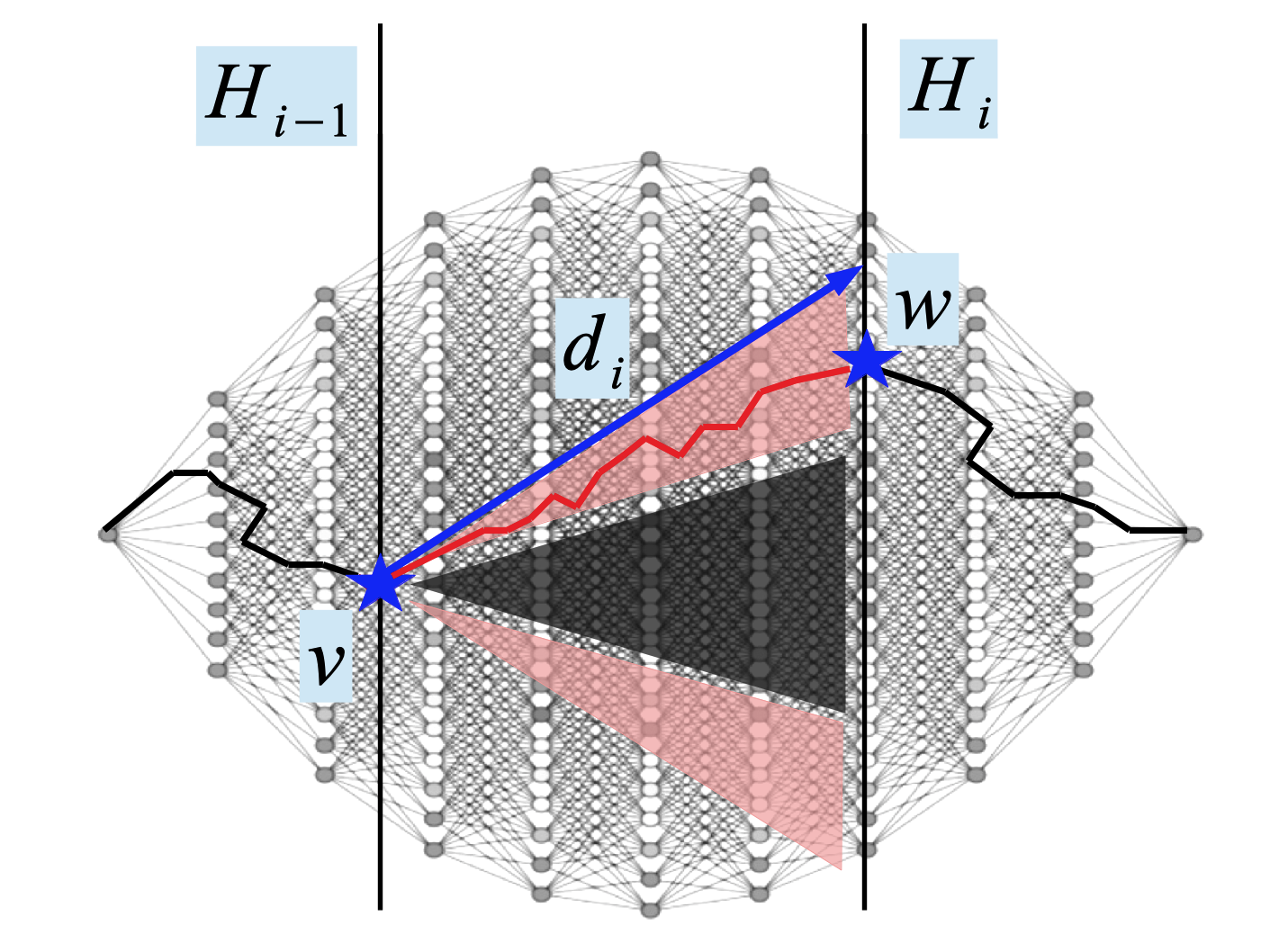}
\caption{The black-shaded cone corresponds to the region where a fully directed polymer would lie. In virtue of Theorem \ref{choosing_d}, the optimal, undirected  polymers evolve however in the red-shaded cones, thereby reaching vertices which are at larger Hamming distance. (Note also that black and red vertical boundaries of these cones are disjunct). For large hyperplane-density, the substrands (in red) of optimal polymers are, in first approximation, geodesics.}
\label{cone}
\end{figure} 

The feature according to which undirected polymers proceed through approximate geodesics 
is absolutely fundamental. On the one hand it neatly explains the deeper mechanisms eventually responsible for the onset of the mean field trivialization. On a more technical level, this  property will lead to a dramatic simplification of some otherwise daunting combinatorial estimates, eventually enabling us to implement the second moment method. In fact, in a (fully) stretched regime, a backstep cannot be cancelled by a forward step (and vice versa). This entails, in particular, a natural representation of paths connecting say $\boldsymbol v \in H_{i-1}$  to $\boldsymbol w \in H_i$ in terms of permutations of the $\boldsymbol v$-coordinates which must be changed in order to obtain $\boldsymbol w$, see in particular Lemma \ref{r} below for a clear manifestation of this  feature.

\subsection{Main result} \label{main_r_sec} We  now specify a subset of polymers with path properties capturing all {\sf Insights} gathered so far: our main result, which is at last formulated in this section, simply states that such a subset is, in fact, non-empty. Towards this goal, some additional observations/notation is needed.

For arbitrary $\boldsymbol d = (d_1, \dots, d_K) \in [0, 1)^K$ (the Hamming-depths) and $\boldsymbol \gamma = (\gamma_1, \dots, \gamma_K) \in [0, \infty)^K$ (the detours),  consider the subset 
\beq \bea
\mathcal P_{n,K} \left\{ \boldsymbol d, \boldsymbol \gamma \right\}  \defi & \; \text{\sf  all paths connecting}\; \boldsymbol 0 \; \text{to}\; \boldsymbol 1\;, \\
& \text{\sf  and that cover a normalized}\; d_i \text{\sf -Hamming distance},\\
& \text{\sf with} \; \gamma_i \; \text{\sf detours}, \\
& \text{\sf while connecting the H-hyperplanes,}\; i=1\dots K\,. \\
\eea \eeq
We now make a specific choice of the free parameters, $\boldsymbol d$ and $\boldsymbol \gamma$, which is naturally justified by the picture canvassed in the above sections. As a matter of fact, we will force polymers to reflect an "extreme" version of the picture. Precisely: 
\begin{itemize}
\item instead of considering polymers which are {\it essentially} directed close to the endpoints (recall in particular Figure \ref{or_vs_unor}) we will consider polymers which are {\it fully} directed in these regimes. We will achieve this by fixing a small $m = 205 \ll K$ (as already mentioned, we will choose $K$ large enough). With $\boldsymbol {\mathsf d} = (\mathsf d_1, \dots, \mathsf d_K)$ the optimal Hamming distance as in \eqref{d_gone} from Theorem \ref{choosing_d} we then set
\beq \bea 
\boldsymbol{ \mathsf d}_{opt}=\left(\underbrace{ 1/K, \dots, 1/K}_{m-\text{times}}, \mathsf d_{m+1}, \mathsf d_{m+2},...,  \mathsf d_{K-m}, \underbrace{ 1/K, \dots, 1/K }_{m-\text{times}}\right)\,,
\eea \eeq
\item instead of considering polymers which are {\it essentially} stretched between the coarse graining H-planes (recall in particular {\sf Insight \ref{geodesics}}), we will consider polymers which proceed through {\it exact geodesics}; this will be achieved by setting 
\beq
{\boldsymbol \gamma}_{opt} \defi (0, \dots, 0)\,.
\eeq
\end{itemize}
Denoting by $\mathsf L_{opt}$ the normalized length of paths in $\mathcal P_{n,K}\left\{ \boldsymbol{ \mathsf d}_{opt} , {\boldsymbol \gamma}_{opt} \right\}$, it holds that
\beq
\mathsf L_{opt}= \| \boldsymbol{ \mathsf d}_{opt}\|_1\,.
\eeq
We then focus on the ensuing subset  $\mathcal P_{n,K}\left\{ \boldsymbol{ \mathsf d}_{opt} , {\boldsymbol \gamma}_{opt} \right\} \subset \widetilde \Pi_{n, \mathsf L_{opt}n}$.  A graphical rendition of these polymers, which are only marginally shorter than $\mathsf L=\sqrt{2}\mathsf E$ (see  \eqref{diff_length} below for more on this), is given in Figure \ref{full_monty}.

\begin{figure}[!h]
    \centering
\includegraphics[scale=0.3]{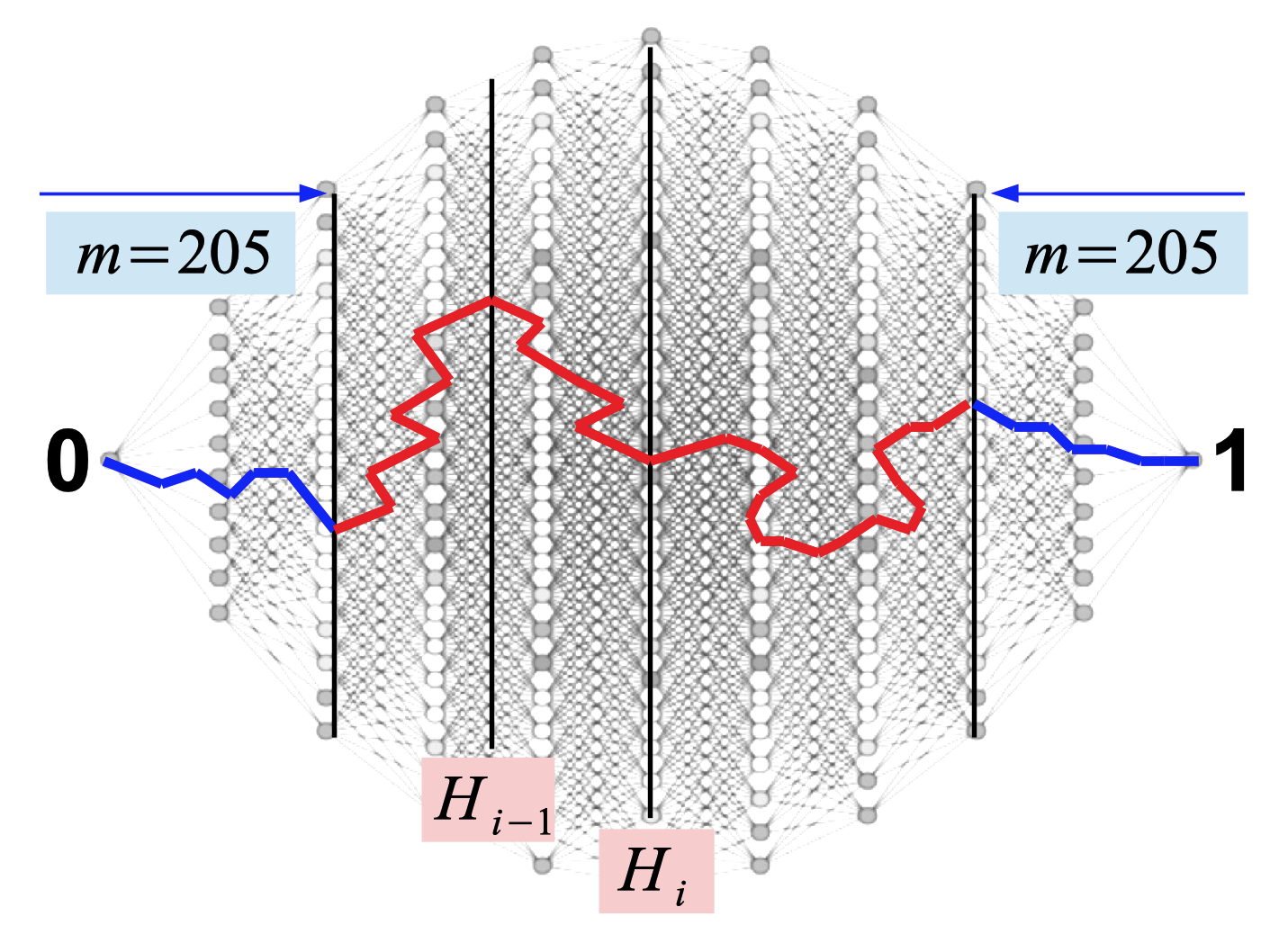}
    \caption{A polymer in $\mathcal P_{n,K}$: the blue substrand is fully directed. The red substrands connect the H-planes of the coarse graining through stretched paths, i.e. geodesics. }
\label{full_monty}
\end{figure}
Since Hamming-depths and detours are specified, we lighten henceforth notation by
\beq
\mathcal P_{n,K}  \defi \mathcal P_{n,K} \left\{ \boldsymbol{ \mathsf d}_{opt} , {\boldsymbol \gamma}_{opt} \right\}.\eeq  
Let now $\e >0$, and consider the  subset of polymers
\beq
\mathcal E_{n, K}^{\e} \defi  \pi\in \mathcal P_{n,K}\;\text{\sf with energies}\;  X_{\pi}\leq \mathsf E+\epsilon\,,
\eeq
namely those paths which {\it i)} are fully directed close to the endpoints, {\it ii)} connect the coarse graining H-planes in the core of the hypercube through geodesics, {\it iii)} and which reach an $\e$-neighborhood of the ground state energy. Our main result states that such polymers do, in fact, exist:

\begin{thm}  \emph{(The geometry of optimal polymers).} \label{geo_optimal_paths}
\label{fpp1} For $\e>0$ there exists  $K = K(\e) \in \N$ such that
\beq \label{Martinsson+} 
\lim_{n\to \infty} {\PP\left( \#\, \mathcal E_{n, K}^{\e} \geq 1 \right)}=1.
\eeq 
\end{thm}

The proof of Theorem \ref{geo_optimal_paths}, which eventually boils down to an application of the  Paley-Zygmund  inequality, is both technically demanding and long, and will be given in the next sections. Before seeing how this goes, some comments are in order. 

First, we remark that the length of the substrands connecting the H-planes (which is related to the $\mathsf a's$) does not appear explicitely in the statement of Theorem \ref{geo_optimal_paths}, and neither do the sub-energies. This is again due to the fact that, in line with {\sf Insight \ref{uni-energy}}, uniformly spread lengths/energies will be hiding behind the optimal Hamming-depths. 

Second, we point out that Theorem \ref{fpp1}, when combined with the simple {\it lower bound} discussed in the Introduction,  yields a constructive proof of Martinsson's Theorem. 

Lastly, and with the unsettled issue of fluctuations in mind, we shall dwell on a conceptually intricate aspect of the theorem, namely the nature of the parameter $K$ encoding the density of hyperplanes for the coarse graining. One perhaps expects that larger constants lead to more accurate pictures, but this is only to some extent correct. In fact, too large hyperplane-density would even lead to inconsistencies: higher and higher densities "unbend" the strands, ultimately to the point of complete directedness, but this, in turn, would starkly contradict the crucial feature of optimal polymers, namely that their length is {\it larger} than the dimension. A delicate balance must therefore be met. As we will see in the course of the second moment implementation, see
\eqref{27}, \eqref{vanish2}, \eqref{finitooooo} and \eqref{choixK} below, for the present purpose of analyzing the ground state to leading order, it indeed suffices to take a  large but {\it finite} $K =  \max \left\{ 2\times10^7, m\e^{-2} \right\}$. How fast (in the dimension $n$) the hyperplane-density can be allowed to grow is  an interesting, and important issue, which unfortunately eludes us. \\

We conclude this section with the aforementioned result concerning the concentration of the length of optimal polymers, as this provides a neat round-off of the picture. To this end, remark that Theorem \ref{fpp1} involves paths of length $\mathsf L_{opt}$; by a more detailed study of taylor's remainder term in \eqref{surprise}, \eqref{small_i_a} and \eqref{small_i_d}, 
and recalling that $\mathsf L = \sqrt{2} \mathsf E$, it can be plainly checked that
\beq \label{diff_length}
0 \leq \mathsf L - \mathsf L_{opt} \leq \frac{m}{K}\,.
\eeq
In other words, for large hyperplane density, the difference between $\mathsf L_{opt}$ and $\mathsf L$ is vanishing. Our second main result states that the length $\mathsf L$ is, in fact, optimal:

\begin{thm} \emph{(Concentration of the polymer's length).} \label{stronger_statement} 
For $\epsilon>0$ and $a>\frac{\mathsf E}{2}+\sqrt2\mathsf E+\frac{1}{\sqrt2}$,  
\beq\bea \label{concentration}
\lim_{n\to \infty} \PP\left(  \#\left\{\pi \in \Pi_{n}:\;  X_{\pi}\leq \mathsf E+\epsilon^2,  \frac{1}{n} | l_\pi(\boldsymbol 0, \boldsymbol 1) - \mathsf L n| \geq a \epsilon \right\} \geq 1  \right) = 0\,.
\eea\eeq 
\end{thm}

\begin{rem}
The proof of the above Theorem, which is given in Section \ref{section_concentration} below, suggests (albeit feebly) that the $(\epsilon^2, \epsilon)$-scaling in \eqref{concentration} is, in fact, optimal, and this in turn suggests that a central limit theorem applies for the optimal length. 
\end{rem}

The rest of the paper is organised as follows. In the next Section \ref{choosing_sec} we will provide a proof of Theorem \ref{choosing_d}. In Section \ref{taming}, and for technical reasons which will become clear in the course of the treatment, some additional restrictions on the candidate optimal polymers will be specified: this will lead to the identification of a subset of $\mathcal P_{n, K}$ on which we will henceforth focus our attention.  Specifying these additional requirements will have an impact on the first moment as controlled in Theorem \ref{choosing_d}, and these modifications will be dealt with in Section \ref{changing_deal}. Section \ref{proof_thm_two} forms the main body of the paper: there we will set up the second moment approach, postponing, however, the highly technical issues concerning the required path-counting to Section \ref{combinatorial_estimates}. Finally, the proof of optimality of the length $\mathsf L$ is given in Section \ref{section_concentration}.

\section{The optimal Hamming distance: proof of Theorem \ref{choosing_d}} \label{choosing_sec}
Recall that $\varphi(x)= x^x$ for $x\geq  0$, with the convention $0^0 = 1$. We shorten 
\beq\bea
g_{j,K}(x)\defi \frac{{\sinh(\ma_j \mathsf E)}^{x} {\cosh(\ma_j \mathsf E)}^{(1-x)} \varphi\left(\frac{j-1}{K}\right)\varphi\left(1-\frac{j-1}{K}\right)}{\varphi\left(\frac{x}{2}-\frac{1}{2K}\right)\varphi\left(\frac{j-1}{K}-(\frac{x}{2}-\frac{1}{2K})\right) \varphi\left(\frac{x}{2}+\frac{1}{2K}\right) \varphi\left(1-\frac{j-1}{K}-(\frac{x}{2}+\frac{1}{2K})\right)} \,,
\eea\eeq
in which case, in virtue of $\eqref{elementaire}$, we may represent the $\mathcal F$-function as
\beq\bea
\mathcal F_{\boldsymbol{\ma}, K}(\boldsymbol d)=\prod_{j=1}^{K}g_{j,K}(d_j)\,.
\eea\eeq
Since the terms in the product on the r.h.s. are non-interacting, we clearly have
\beq\bea
\max_{\boldsymbol d} \{ \mathcal F_{\boldsymbol{\ma}, K}(\boldsymbol d) \}=\prod_{j=1}^{K}\max_{x \geq 0} \{ g_{j,K}(x) \}.
\eea\eeq
We now claim that
\beq\bea\label{claim1}
\prod_{j=1}^{K}\max_{x\geq 0} \{g_{j,K}(x) \} =1,
\eea\eeq
and 
\beq\bea\label{claim2}
\arg\max_{x\geq 0} \ \{g_{j,K}(x)\}= \mathsf d_j, 
\eea\eeq
with $\mathsf d_j$ as in \eqref{d_gone}. \\

\noindent We will prove \eqref{claim2} first. We begin with the cases $j=1, K$ and claim that
\beq \label{claimweirdo}
\arg \max_{x \geq 0} g_{1, K}(x) = \arg \max_{x\geq 0} g_{K, K}(x) = \frac{1}{K}\,,
\eeq
and 
\beq \label{claimweirdo_two}
\frac{1}{K} = \mathsf d_1 = \mathsf d_K.
\eeq
In fact, $g_{1, K}(x)$ involves the terms 
\beq \bea
\varphi\left(\frac{x}{2}-\frac{1}{2K}\right), \\
\varphi\left(\frac{j-1}{K}-\left\{\frac{x}{2}-\frac{1}{2K}\right\} \right)\Big|_{j=1} & = \varphi\left(\frac{1}{2K} - \frac{x}{2}\right),
\eea \eeq 
but for both to be properly defined it must hold 
\beq
\frac{x}{2}-\frac{1}{2K} \geq 0, \quad \text{and} \quad \frac{1}{2K} - \frac{x}{2}\geq 0,
\eeq
implying $x= \frac{1}{K}$. A similar reasoning applies to $g_{K,K}$, and \eqref{claimweirdo} is settled. 
Claim \eqref{claimweirdo_two} follows from \eqref{fundamental} for the $j=1$ case, whereas the 
$j=K$ case follows by symmetry, see in particular \eqref{sym_k}. \\

\noindent Concerning the other indices, we fix $j \in \{2, \dots, K-1\}$ and shorten, for $x \geq 0$, 
\beq\bea
g_{j,K}(x) \defi \frac{N_{j, K}(x)}{D_{j,K}(x)},
\eea\eeq
where
\beq\bea
N_{j, K}(x) \defi {\sinh(\ma_j \mathsf E)}^{x} {\cosh(\ma_j \mathsf  E)}^{(1-x)} \varphi\left(\frac{j-1}{K}\right)\varphi\left(1-\frac{j-1}{K}\right),
\eea\eeq
and
\beq\bea
D_{j,K}(x) \defi \varphi \left(\frac{x}{2}-\frac{1}{2K} \right)\varphi\left(\frac{j}{K}-\frac{x}{2}-\frac{1}{2K}\right) \varphi\left(\frac{x}{2}+\frac{1}{2K}\right) \varphi\left(1-\frac{j}{K}-\frac{x}{2}+\frac{1}{2K}\right).
\eea\eeq
Taking the $x$-derivative, we see that 
\beq\bea\label{derivee}
g_{j,K}(x)'>0 \iff {N_{j, K} (x)}' D_{j,K}(x)>N_{j, K} (x){D_{j,K}(x)}'.
\eea\eeq
An elementary computation then yields
\beq\bea\label{derivee1}
{N_{j, K} (x)}'=N_{j, K} (x) \log(\tanh(\ma_j \mathsf E)),
\eea\eeq
and 
\beq\bea\label{derivee2}
{D_{j,K}(x)}'=\frac{1}{2} D_{j,K}(x) \log\left\{ \frac{(\frac{x}{2}-\frac{1}{2K})(\frac{x}{2}+\frac{1}{2K})}{(\frac{j}{K}-\frac{x}{2}-\frac{1}{2K})(1-\frac{j}{K}-\frac{x}{2}+\frac{1}{2K})} \right\} .
\eea\eeq
Combining \eqref{derivee}, \eqref{derivee1} and \eqref{derivee2}, we therefore get
\beq\bea \label{derivee3}
g_{i,K}(x)'>0 \iff {\tanh(\ma_j \mathsf  E)}^2>\frac{(\frac{x}{2}-\frac{1}{2K})(\frac{x}{2}+\frac{1}{2K})}{(\frac{j}{K}-\frac{x}{2}-\frac{1}{2K})(1-\frac{j}{K}-\frac{x}{2}+\frac{1}{2K})}.
\eea\eeq
Consider now 
\beq\bea \label{fpoint}
{\tanh(\ma_j \mathsf E)}^2=\frac{(\frac{x}{2}-\frac{1}{2K})(\frac{x}{2}+\frac{1}{2K})}{\left(\frac{j}{K}-\frac{x}{2}-\frac{1}{2K}\right)(1-\frac{j}{K}-\frac{x}{2}+\frac{1}{2K})}.
\eea\eeq
This is a quadratic equation (in $x$), whose unique positive solution is given by 
\beq\bea \label{solution}
\hat x \defi -{\sinh(\ma_j \mathsf  E)}^2+\sqrt{{\sinh(\ma_j \mathsf  E)}^4+4{\sinh(\ma_j \mathsf  E)}^2\left\{\frac{2j-1}{2K}-\frac{j(j-1)}{K^2}\right\}+\frac{1}{K^2}}.
\eea\eeq
A straightforward analysis shows that the quotient on the r.h.s. of \eqref{derivee3} is, in fact, increasing in $x$: in other words, the $x$-derivative $g_{i,K}'$ is positive for $x < \hat x$ and negative for $x>\hat x$, implying that $\hat x$ is indeed the extremal point. To finish the proof of \eqref{claim2} it thus remains to show that $\hat x = \mathsf d_j$, i.e. that $\hat x = \sinh(\mathsf a_j \mathsf E) \cosh((1-\mathsf a_j) \mathsf  E)$. In order to do so, we will avoid the use of the explicit formulation \eqref{solution}, but rely rather on the expression \eqref{fpoint} and the following
\begin{lem} \label{allequivalent} Let $d \in \R$ satisfy
\beq \label{e_1}
\frac{d}{2}-\frac{1}{2K}=\sinh(\overline{\ma}_{j-1}  \mathsf E)\sinh(\ma_{j}  \mathsf E)\sinh(\underline{\ma}_{j}  \mathsf E)\,.
\eeq
Then the above, and the following relations are all equivalent: 
\beq \label{e_3}
1-\frac{j}{K}-\frac{d}{2}+\frac{1}{2K}=\cosh(\overline{\ma}_{j-1} \mathsf E)\cosh(\ma_{j} \mathsf E)\sinh(\underline{\ma}_{j} \mathsf E),
\eeq
\beq\bea \label{e_2}
 & \frac{d}{2}+\frac{1}{2K} =\cosh(\overline{\ma}_{j-1}  \mathsf E)\sinh(\ma_{j}  \mathsf E)\cosh(\underline{\ma}_{j}  \mathsf E),
\eea\eeq
\beq \label{e_4}
\frac{j}{K}-\frac{d}{2}-\frac{1}{2K}=\sinh(\overline{\ma}_{j-1}  \mathsf E)\cosh(\ma_{j}  \mathsf E)\cosh(\underline{\ma}_{j} \mathsf E).
\eeq
It follows in particular, that for such $d$ it holds $d = \hat x$, and $d= \mathsf d_j$. 
\end{lem}
\begin{proof}[Proof of Lemma \ref{allequivalent}]
We first prove the equivalence of
\beq \bea \label{allequivalent_two}
\eqref{e_1} \iff \eqref{e_3} \iff \eqref{e_2} \iff \eqref{e_4},
\eea\eeq
Indeed, by \eqref{fundamental} and the fact that
\beq \bea \label{fundie}
\sinh(\overline{\ma_j} \mathsf E)\cosh(\underline{\ma}_j \mathsf E)+\cosh(\overline{\ma}_j \mathsf E)\sinh(\underline{\ma}_j \mathsf E)=1,
\eea\eeq
it holds:
\beq\label{fundamental2}
\sinh(\underline{\ma}_j \mathsf E)\cosh(\overline{\ma}_j \mathsf E)=1-\frac{j}{K}
\eeq
for all  $j =1 \dots  K$. Relation \eqref{e_1} therefore implies that
\beq \bea \label{simpl1}
1-\frac{j}{K}-\frac{d}{2}+\frac{1}{2K} &=1-\frac{j}{K}-\sinh(\overline{\ma}_{j-1} \mathsf E)\sinh(\ma_{j} \mathsf E)\sinh(\underline{\ma}_{j} \mathsf E)\\ 
&=\left\{ \cosh(\overline{\ma}_{j}\mathsf E)-\sinh(\overline{\ma}_{j-1} \mathsf E) \sinh(\ma_{j} \mathsf E)  \right\} \sinh(\underline{\ma}_{j}\mathsf E)\\
&=\cosh(\overline{\ma}_{j-1} \mathsf E)\cosh(\ma_{j} \mathsf E)\sinh(\underline{\ma}_{j} \mathsf E)
\eea\eeq
the second equality with \eqref{fundamental2} and the last  by the addition formula $\cosh(a+b)=\cosh(a)\cosh(b)+\sinh(a)\sinh(b)$. Thus,
\beq \bea
\eqref{e_1} \iff \eqref{e_3}.
\eea\eeq
A similar computation gives that
\beq \bea
\eqref{e_2} \iff  \eqref{e_4}.
\eea\eeq
It remains to prove that 
\beq \bea
\eqref{e_1} \iff \eqref{e_2}.
\eea\eeq
To see this we note that \eqref{e_1} yields
\beq \bea \label{a}
&\frac{d}{2}+\frac{1}{2K}=\sinh(\overline{\ma}_{j-1} \mathsf E)\sinh(\ma_{j} \mathsf E)\sinh(\underline{\ma}_{j} \mathsf E)+\frac{1}{K}
\eea\eeq
but combining the fundamental r.h.s (\ref{fundamental}) and (\ref{fundamental2}) gives that
\beq \bea \label{astuce}
\frac{1}{K}=\sinh(\underline{\ma}_{j-1}  \mathsf E)\cosh(\overline{\ma}_{j-1}  \mathsf E)-\sinh(\underline{\ma}_{j} \mathsf E)\cosh(\overline{\ma}_{j}  \mathsf E)
\eea\eeq
Thus, by \eqref{astuce}, we see that
\beq \bea \label{b}
\eqref{a}&=\sinh(\underline{\ma}_{j}  \mathsf E)\left(\sinh(\overline{\ma}_{j-1}  \mathsf E)\sinh(\ma_{j} E)-\cosh(\overline{\ma}_{j}  \mathsf E)\right)+\sinh(\underline{\ma}_{j-1}  \mathsf E)\cosh(\overline{\ma}_{j-1}  \mathsf E)\\
&=-\sinh(\underline{\ma}_{j}  \mathsf E)\cosh(\ma_{j}  \mathsf E)\cosh(\overline{\ma}_{j-1}  \mathsf E)+\sinh(\underline{\ma}_{j-1}  \mathsf E)\cosh(\overline{\ma}_{j-1}  \mathsf E),
\eea\eeq
the last equality again by the addition formula $\cosh(a+b)=\cosh(a)\cosh(b)+\sinh(a)\sinh(b)$. Hence
\beq \bea
\eqref{b}&=\left(-\sinh(\underline{\ma}_{j}  \mathsf E)\cosh(\ma_{j}  \mathsf E)+\sinh(\underline{\ma}_{j}  \mathsf E+\ma_{j}  \mathsf E)\right)\cosh(\overline{\ma}_{j-1} \mathsf E)\\
&=\cosh(\underline{\ma}_{j}  \mathsf E)\sinh(\ma_{j}  \mathsf E)\cosh(\overline{\ma}_{j-1}  \mathsf E).
\eea\eeq
and \eqref{allequivalent_two} is established. \\

Let now $d$ satisfy any of the equivalent \eqref{e_1}-\eqref{e_4}. It holds:
\beq\bea\label{check2}
& \frac{(\frac{d}{2}-\frac{1}{2K})(\frac{d}{2}+\frac{1}{2K})}{(\frac{j}{K}-\frac{d}{2}-\frac{1}{2K})(1-\frac{j}{K}-\frac{d}{2}+\frac{1}{2K})}=  \\
&\qquad \qquad = \frac{\sinh(\overline{\ma}_{j-1} \mathsf E)\sinh(\ma_{j} \mathsf E)\sinh(\underline{\ma}_j \mathsf E) \times \cosh(\overline{\ma}_{j-1} \mathsf E)\sinh(\ma_{j} \mathsf E)\cosh(\underline{\ma}_j \mathsf E)}{\sinh(\overline{\ma}_{j-1} \mathsf E)\cosh(\ma_{j} \mathsf E)\cosh(\underline{\ma}_j \mathsf E)\times \cosh(\overline{\ma}_{j-1} \mathsf E)\cosh(\ma_{j} \mathsf E)\sinh(\underline{\ma}_j \mathsf E)}\\
& \qquad \qquad =\tanh(\ma_j \mathsf E)^2\,,
\eea\eeq
hence, by uniqueness of the (positive) solution of \eqref{fpoint}, we deduce that $d= \hat x$. \\

\noindent Finally, it holds: 
\beq\bea
d & =\frac{d}{2}+\frac{1}{2K}+\frac{d}{2}-\frac{1}{2K}\\
&=\sinh(\ma_{i} \mathsf E)\cosh(\overline{\ma}_{j-1} \mathsf E)\cosh(\underline{\ma}_j \mathsf E)+ \sinh(\ma_{i} \mathsf E) \sinh(\overline{\ma}_{j-1} \mathsf E)\sinh(\underline{\ma}_j \mathsf E)\,,
\eea \eeq
the last equality by \eqref{e_1} and \eqref{e_2},  hence
\beq \bea
d &=\sinh(\ma_{i}  \mathsf E) \times \left\{\cosh(\overline{\ma}_{j-1} \mathsf E)\cosh(\underline{\ma}_j \mathsf E)+ \sinh(\overline{\ma}_{j-1} \mathsf E)\sinh(\underline{\ma}_j \mathsf E) \right\} \\
& = \sinh(\ma_{j} \mathsf E) \times \cosh((1-\ma_{j}) \mathsf E),
\eea\eeq
by the addition formula for hyperbolic functions (and using that $\overline{\ma}_{j-1} + \underline{\ma}_{j} = 1-\mathsf a_j$, by definition), settling the claim that $d=\mathsf d_j$.

\end{proof}

The remaining Claim \eqref{claim1} is taken care of by the  following Lemma,
which tracks the evolution of the $g$-product while changing the hyperplane-index.

\begin{lem}[Evolution Lemma]\label{f} For any $i=1 \dots K$, it holds:
\beq\bea  \label{f_f}
\prod_{j=1}^{i} g_{j,K}(\mathsf{ \mathsf d_j})={\left[ \frac{\sinh(\overline{\ma}_i  \mathsf E)}{\frac{i}{K}}\right]}^{\frac{i}{K}} {\left[ \frac{\cosh(\overline{\ma}_i  \mathsf E)}{1-\frac{i}{K}}\right]}^{1-\frac{i}{K}}\,.
\eea\eeq
Furthermore,
\beq \label{full_produce}
\prod_{j=1}^{K} g_{j,K}(\mathsf{ \mathsf d_j})= 1\,.
\eeq
\end{lem}
\begin{proof}
We will proceed by induction over $i$. The cases $K=1, 2$ are trivial, so let $K\geq3$. Recalling that  $\mathsf  d_1=\frac{1}{K}$, we therefore have that 
\beq
g_{1,K}(\mathsf d_1)={\left[\frac{\sinh(\ma_1 \mathsf E)}{\frac{1}{K}}\right]}^{\frac{1}{K}}{\left[\frac{\cosh(\ma_1 \mathsf E)}{1-\frac{1}{K}}\right]}^{1-\frac{1}{K}},
\eeq
which settles the base case $i = 1$. We thus assume that \eqref{f_f} holds for an $i \in \{1,K-2\}$, and show 
that this implies the validity of the $(i+1)$-case, namely that
\beq\bea \label{remain1}
{\left[ \frac{\sinh(\overline{\ma}_i \mathsf E)}{\frac{i}{K}}\right]}^{\frac{i}{K}} {\left[\frac{\cosh(\overline{\ma}_i \mathsf E)}{1-\frac{i}{K}}\right]}^{1-\frac{i}{K}} g_{i+1,K}(\mathsf  d_{i+1})={\left[ \frac{\sinh(\overline{\ma}_{i+1}\mathsf E)}{\frac{i+1}{K}}\right] }^{\frac{i+1}{K}} {\left[ \frac{\cosh(\overline{\ma}_{i+1} \mathsf  E)}{1-\frac{i+1}{K}}\right]}^{1-\frac{i+1}{K}}.
\eea\eeq
Remark that by \eqref{fpoint},
\beq \bea
& {\sinh(\ma_{i+1} \mathsf E)}^{\mathsf d_{i+1}} {\cosh(\ma_ {i+1} \mathsf E)}^{1-\mathsf d_{i+1}} = 
{\tanh(\ma_{i+1} \mathsf E)}^{\mathsf d_{i+1}} {\cosh(\ma_ {i+1} \mathsf E)} \\
& \qquad = \left[ \frac{(\frac{d_{i+1}}{2}-\frac{1}{2K})(\frac{d_{i+1}}{2}+\frac{1}{2K})}{\left(\frac{i+1}{K}-\frac{d_{i+1}}{2}-\frac{1}{2K}\right)(1-\frac{i+1}{K}-\frac{d_{i+1}}{2}+\frac{1}{2K})}\right]^{\frac{d_{i+1}}{2}} \cosh(\ma_ {i+1} \mathsf E) \,.
\eea \eeq
By definition of $g_{i+1,K}$, the above, and simple rearrangements, we thus have
\beq\bea
g_{i+1,K}(\mathsf d_{i+1}) & = \frac{{(\frac{{\mathsf  d_{i+1}}}{2}-\frac{1}{2K})}^{\frac{1}{2K}}\cosh(\ma_{i+1} E){\left(\frac{i}{K}\right)}^{\frac{i}{K}}{(1-\frac{i}{K})}^{1-\frac{i}{K}}}{{(\frac{i+1}{K}-\frac{{\mathsf  d_{i+1}}}{2}-\frac{1}{2K})}^{\frac{i+1}{K}-\frac{1}{2K}}{(\frac{\mathsf  d_{i+1}}{2}+\frac{1}{2K})}^{\frac{1}{2K}}{(1-\frac{i+1}{K}-\frac{\mathsf  d_{i+1}}{2}+\frac{1}{2K})}^{1-\frac{i+1}{K}+\frac{1}{2K}}}.
\eea\eeq
Thus (\ref{remain1}) is equivalent to prove that
\beq\bea \label{remain2}
&{\left[\frac{\frac{\mathsf d_{i+1}}{2}+\frac{1}{2K}}{\frac{\mathsf d_{i+1}}{2}-\frac{1}{2K}}\right]}^{\frac{1}{2K}}{\left[\frac{i+1}{K}-\frac{\mathsf d_{i+1}}{2}-\frac{1}{2K}\right]}^{\frac{i+1}{K}-\frac{1}{2K}}{\left[1-\frac{i+1}{K}-\frac{\mathsf d_{i+1}}{2}+\frac{1}{2K}\right]}^{1-\frac{i+1}{K}+\frac{1}{2K}}\\
& \ \ \ \ \ \ \ \ \ \ \ \ \ \ \ \ \ \ \ \ \ \ \ \ =\frac{{\sinh(\overline{\ma}_i\mathsf E)}^{\frac{i}{K}}{\cosh(\overline{\ma}_i \mathsf E)}^{1-\frac{i}{K}}\cosh(\ma_{i+1}\mathsf E)}{{\left[\frac{\sinh(\overline{\ma}_{i+1} \mathsf E)}{\frac{i+1}{K}}\right]}^{\frac{i+1}{K}} {\left[\frac{\cosh(\overline{\ma}_{i+1} \mathsf E)}{1-\frac{i+1}{K}}\right]}^{1-\frac{i+1}{K}}}\,.
\eea\eeq
We now rewrite the term on the l.h.s.  \eqref{remain2} as
\beq\bea\label{left}
&{\left[\frac{(\frac{ \mathsf d_{i+1}}{2}+\frac{1}{2K})(\frac{i+1}{K}-\frac{ \mathsf d_{i+1}}{2}-\frac{1}{2K})}{(\frac{\mathsf d_{i+1}}{2}-\frac{1}{2K})(1-\frac{i+1}{K}-\frac{ \mathsf d_{i+1}}{2}+\frac{1}{2K})}\right]}^{\frac{1}{2K}} \\
& \hspace{3cm} \times {\left[\frac{\frac{i+1}{K}-\frac{ \mathsf d_{i+1}}{2}-\frac{1}{2K}}{1-\frac{i+1}{K}-\frac{ \mathsf d_{i+1}}{2}+\frac{1}{2K}}\right]}^{\frac{i}{K}}\\
&\hspace{6cm} \times \left[ 1-\frac{i+1}{K}-\frac{d_{i+1}}{2}+\frac{1}{2K}\right],\\
\eea\eeq
and the term on the r.h.s. of \eqref{remain2} as
\beq\bea\label{right}
&\left[ \left(\frac{\frac{i+1}{K}\cosh(\overline{\ma}_{i+1} \mathsf E)}{(1-\frac{i+1}{K})\sinh(\overline{\ma}_{i+1} \mathsf E)}\right)^2\right]^{\frac{1}{2K}} \\
& \hspace{3cm} \times {\left[\frac{\frac{i+1}{K}\tanh(\overline{\ma}_i \mathsf E)\cosh(\overline{\ma}_{i+1} \mathsf E)}{(1-\frac{i+1}{K})\sinh(\overline{\ma}_{i+1} \mathsf E)}\right]}^{\frac{i}{K}}\\
& \hspace{6cm} \times \frac{\cosh(\overline{\ma}_i E)\cosh(\ma_{i+1} E)(1-\frac{i+1}{K})}{\cosh(\overline{\ma}_{i+1} E)}\\
& = \left[\left(\frac{\cosh(\underline{\ma}_{i+1} \mathsf E)}{\sinh(\underline{\ma}_{i+1} \mathsf E)}\right)^2\right]^{\frac{1}{2K}} \\
& \hspace{3cm} \times {\left[ \frac{\tanh(\overline{\ma}_i \mathsf E)\cosh(\underline{\ma}_{i+1} \mathsf E)}{\sinh(\underline{\ma}_{i+1} \mathsf E)}\right]}^{\frac{i}{K}}\\
&\hspace{6cm} \times \cosh(\overline{\ma}_i \mathsf E)\cosh(\ma_{i+1} \mathsf E)\sinh(\underline{\ma}_{i+1} \mathsf E)\,,
\eea\eeq 
the last step by \eqref{fundamental} and \eqref{fundamental2}. But by \eqref{e_1}, \eqref{e_3}, \eqref{e_2} and \eqref{e_4}, the terms raised to the same powers in \eqref{left} and the r.h.s. of \eqref{right} coincide,  settling the induction step.  \\

\noindent We now move to \eqref{full_produce}. It holds:
\beq\bea\label{prof}
\prod_{j=1}^{K} g_{j,K}( \mathsf d_j)&=\prod_{j=1}^{K-1} g_{j,K}( \mathsf d_j) \  g_{K,K}( \mathsf d_K)\\
& = {\left[\frac{\sinh(\overline{\ma}_{K-1} \mathsf E)}{1-\frac{1}{K}}\right]}^{1-\frac{1}{K}} {
\left[\frac{\cosh(\overline{\ma}_{K-1} \mathsf E)}{\frac{1}{K}}\right]}^{\frac{1}{K}}{\sinh(\ma_K \mathsf E)}^{\frac{1}{K}}{\cosh(\ma_K \mathsf E)}^{1-\frac{1}{K}}\\
&={\left[\frac{\sinh(\overline{\ma}_{K-1} \mathsf E)\cosh(\ma_K \mathsf E)}{1-\frac{1}{K}}\right]}^{1-\frac{1}{K}} {\left[\frac{\cosh(\overline{\ma}_{K-1} \mathsf E)\sinh(\ma_K \mathsf E)}{\frac{1}{K}}\right]}^{\frac{1}{K}}\,,
\eea\eeq
the second equality by the induction step, and the third by simple rearrangements. 
By the $\ma's$ symmetry \eqref{sym_k},  and the normalization $\sum_{i=1}^K \ma_i=1$, it thus holds 
\beq\bea
\eqref{prof}&={\left[\frac{\sinh(\overline{\ma}_{K-1} \mathsf E)\cosh(\underline{\ma}_{K-1} \mathsf E)}{1-\frac{1}{K}}\right]}^{1-\frac{1}{K}} {\left[ \frac{\cosh((1-\ma_1) \mathsf E)\sinh(\ma_1 \mathsf E)}{\frac{1}{K}}\right]}^{\frac{1}{K}}=1,
\eea\eeq
the last equality by the fundamental \eqref{fundamental}.
\end{proof}

\section{Taming optimal polymers}  \label{taming}
In order to prove our main result Theorem \ref{geo_optimal_paths}, we will show non-emptiness of a {\it subset} of $\mathcal P_{n,K}$, whose paths satisfy additional properties. As a matter of fact, we will introduce two additional restrictions: the first one, which is explained in Section \ref{sprinckle}, concerns the {\it geometry} of paths, i.e. their combinatorial properties. The second restriction, explained in Section \ref{energies_sec},  concerns the way {\it energies} are distributed along the paths.  Both restrictions will be of course inspired by/in line with the above {\sf Insights}. We emphasize that the reason for restricting the candidate polymers further is here chiefly technical: the additional requirements we are about to introduce will in fact lead to a considerable simplification of some otherwise daunting combinatorial estimates. 

\subsection{A sprinckle of microstructure} \label{sprinckle}
We introduce yet another coarse graining: for $ i=0\dots K-1$, we  split the region between two consecutive hyperplanes $H_{i-1}$ and $H_i$ further, into $K'$ additional slabs:
\beq
H_{i,j}' \defi \left\{v \in V_n, d(0,v)=  \left(i+\frac{j}{K'} \right)\hat n_K  \right\}, \quad \  \quad j=0\dots K' \,,
\eeq 
(remark that $H_{i,0}' =H_{i}$ and $H_{i,K'}'=H_{i+1}$), and focus henceforth on the subset
\beq \bea \label{defi_pkkprime}
\mathcal P_{n,K,K'}  \defi & \; \text{\sf all polymers}\; \pi \in \mathcal P_{n,K} \; \text{\sf which cover} \\
& \text{\sf a (normalized) Hamming distance}\; \left( \mathsf{{ef}_i}+\mathsf{{eb}_i}\right) / K' \\
& \text{\sf while connecting the hyperplanes} \; H_{i,j}' \; \text{\sf and} \; H_{i,j+1}', \\
& \text{\sf for}\; j=0\dots K'\; \text{\sf and}\; i=1\dots K-1. \\
\eea \eeq
The subset $\mathcal P_{n,K,K'}$ is of course motivated by {\sf Insight \ref{insight_rare}}: adding an additional level of coarse graining and spreading the backsteps as evenly as possible among the $K'$-slabs, allows to rule out polymers where backsteps tend to accumulate, cfr. Figure \ref{no_cumulation} and \ref{no_cumulation2} below. \\

\begin{figure}
\centering
\begin{minipage}{.56\textwidth}
  \centering
 \includegraphics[scale=0.3]{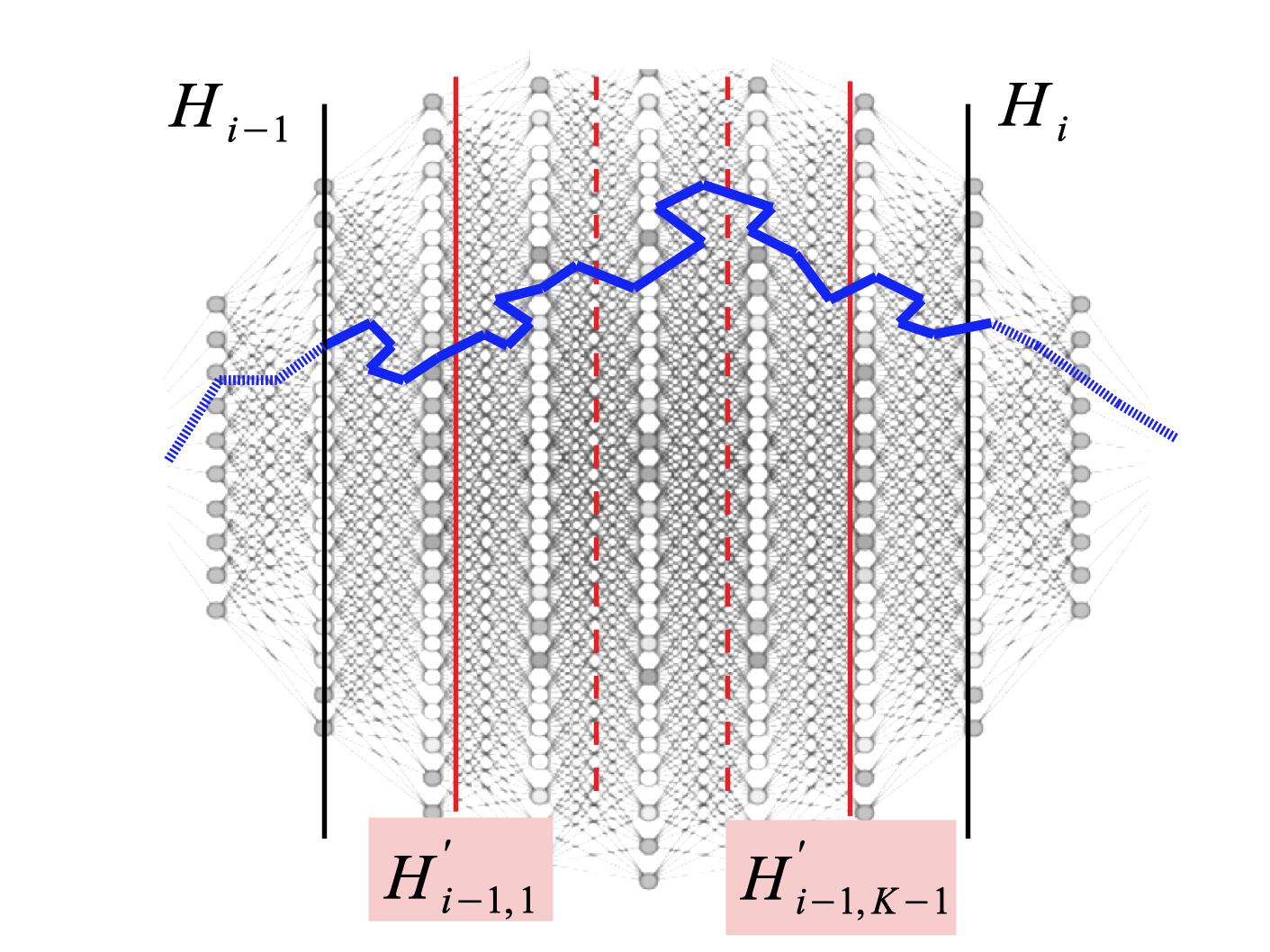}
  \captionof{figure}{The backsteps (five in total) are spread as evenly as possible: one for each sublayer $H'$.}
  \label{no_cumulation}
\end{minipage}%
\begin{minipage}{.56\textwidth}
  \centering
  \includegraphics[scale=0.3]{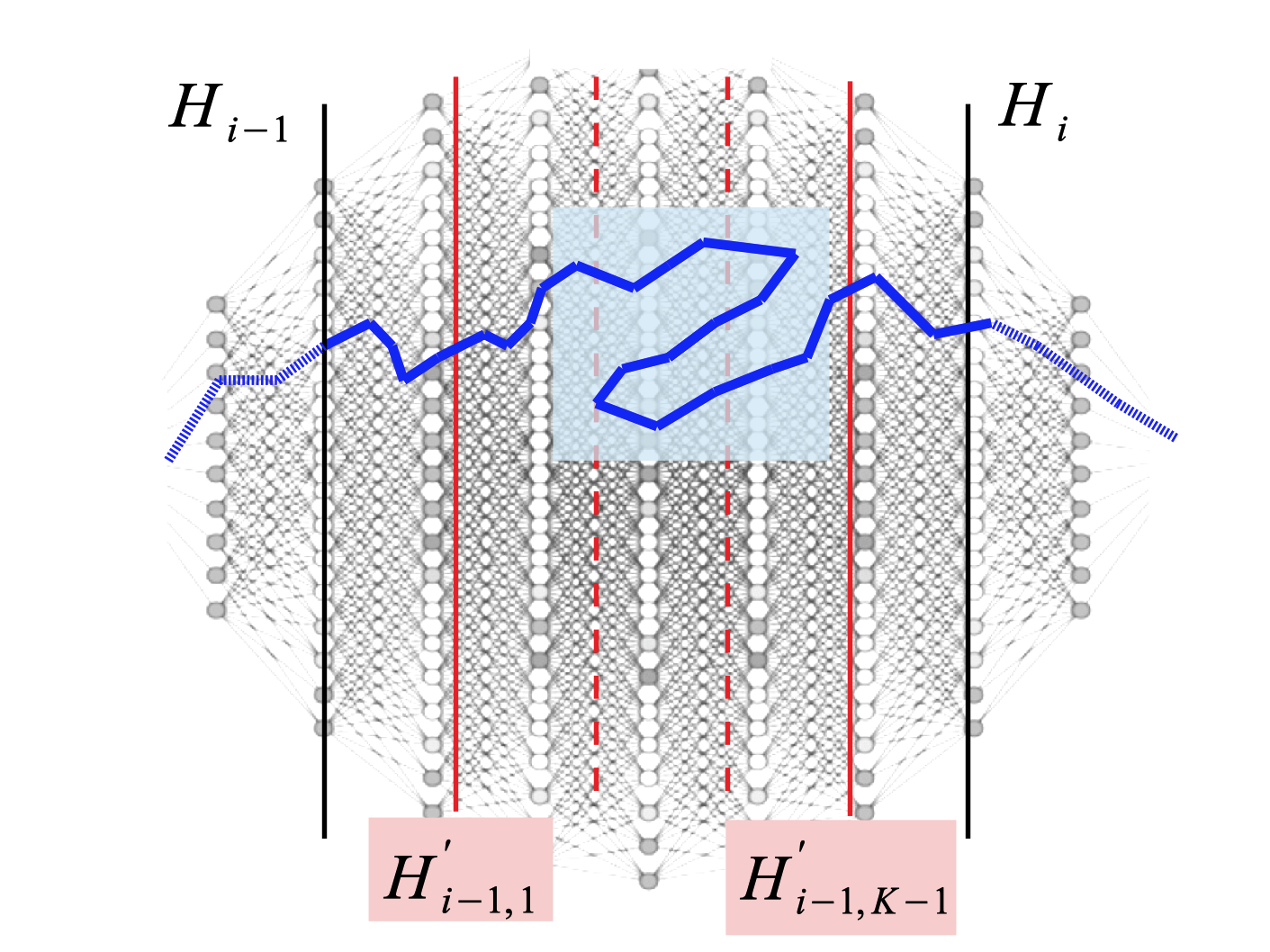}
 \captionof{figure}{The five  backsteps are lumped together: this polymer wouldn't belong to $\mathcal P_{n,K,K'}$.}
   \label{no_cumulation2}
\end{minipage}
\end{figure}

\noindent Finally, we render the $H$-hyperplanes (of the coarser layer) {\it repulsive}, i.e. we force paths to cross them only once. As we will see shortly, see Lemma \ref{rem} below, this can be achieved by considering the following (sub)subset of polymers:
\beq \bea
\mathcal P^{\sf rep}_{n,K,K'}  \defi & \;  \text{\sf all polymers}\; \pi \in \mathcal P_{n,K, K'} \; \text{\sf which connect the hyperplanes}\; H_{i,0}' \; \text{and} \; H_{i,1}' \\
&\text{\sf by  first making} \; (\mathsf{ef}_i \; \hat n_{K'}) \; \text{\sf steps forward} \; \text{\sf and only then} \; (\mathsf{eb}_i  \; \hat n_{K'}) \; \text{\sf backsteps}, \\
& \text{\sf and which connect the hyperplanes}\; H_{i,K'-1}' \; \text{and} \; H_{i,K'}' \\
&\text{\sf by  first making} \; (\mathsf{eb}_i \; \hat n_{K'}) \; \text{\sf backsteps, and only then} \; (\mathsf{ef}_i  \; \hat n_{K'}) \; \text{\sf steps forward}\,, \\
& \text{\sf for}\; i=1\dots K. 
\eea \eeq
Note that $\mathcal P^{\sf rep}_{n,K,K'}$ is still a deterministic set. A graphical rendition is given in Figure \ref{noback} below.

\begin{figure}[!h]
    \centering
        \includegraphics[scale=0.32]{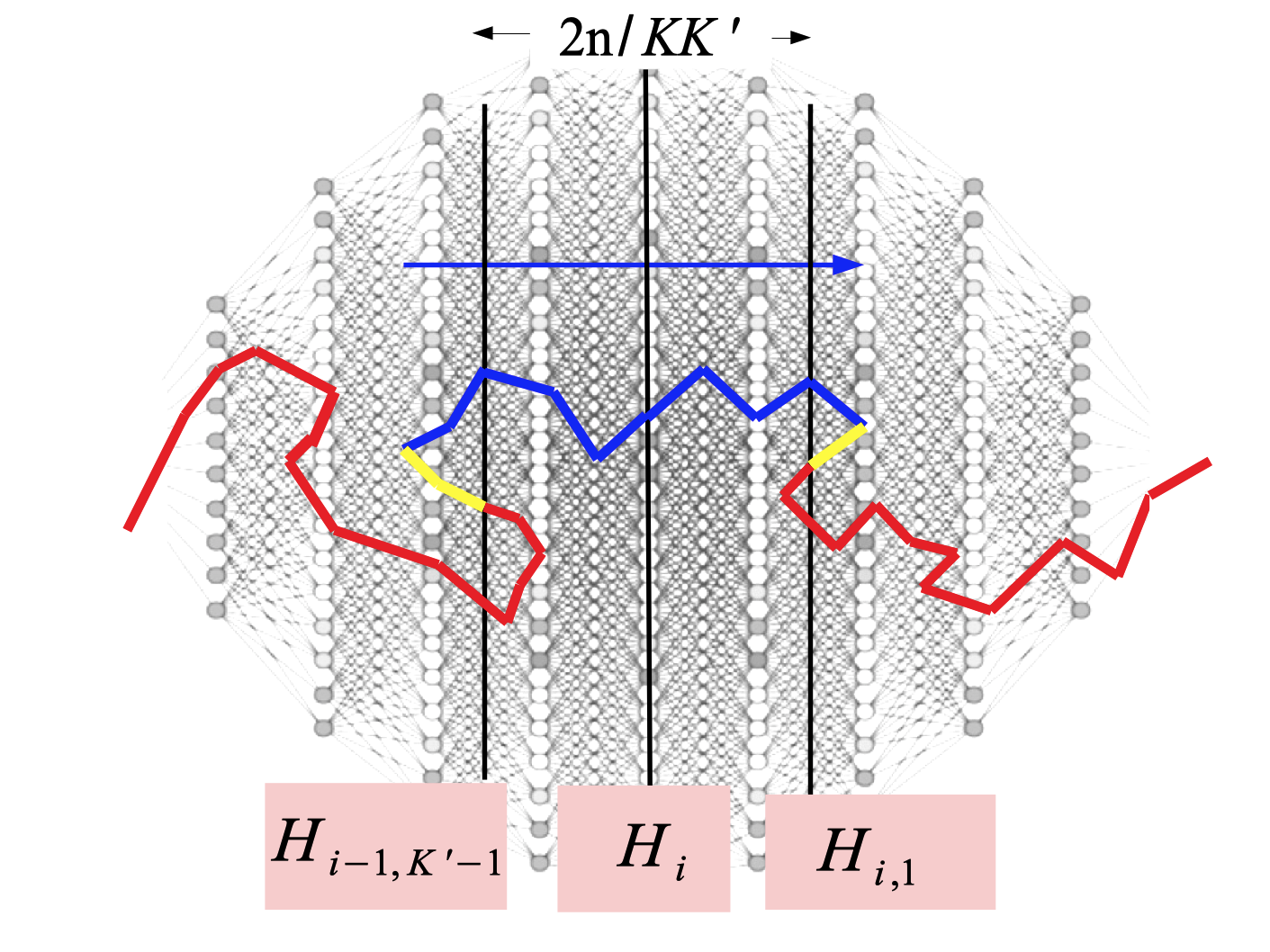}
    \caption{A path in $\mathcal P_{n,K,K'}^{\sf rep}$: red edges correspond to the free evolution of the path, yellow edges are backsteps, and blue edges are forward steps.}
\label{noback}
\end{figure}

Remark that, by construction, 
\beq \label{inclusions}
\mathcal P^{\sf rep}_{n,K,K'} \subset \mathcal P_{n,K, K'} \subset \mathcal P_{n, K}\left\{ \boldsymbol{ \mathsf d}_{opt} , {\boldsymbol \gamma}_{opt} \right\}\,.
\eeq
Our main Theorem \ref{geo_optimal_paths} will therefore follow as soon as we prove that one can find 
polymers in $\mathcal P^{\sf rep}_{n,K,K'}$ which reach the ground state energy. Before seeing how this goes, here is the aforementioned result stating that $H$-hyperplanes are indeed repulsive: 

\begin{lem}\label{rem} For $K \geq 1$ the following holds true: a polymer $\pi \in \mathcal P_{n,K,K'}^{\sf rep}$ crosses the hyperplanes $H_1, \dots, H_K$ only once. 
\end{lem} 
\begin{proof}
The statement is trivial in the directed phase, so let $i \in \{m \dots K-m\}$. 

There is of course a certain {\it directivity} in the polymers' evolution: this is captured by the fact that $\mathsf{ef}_i > \mathsf{eb}_i$ for all $i=1\dots K$ (see in particular the second relation in \eqref{b-f}), and graphically represented by evolutions "from the left to the right". 

Sticking to this graphical convention, we begin with the case "to the right of the $H_i$-hyperplane":  after crossing this hyperplane, a path $\pi \in \mathcal P_{n,K,K'}$ is bound to first make $(\mathsf {ef}_i \; \hat n_{K'})$ steps to the right (forward) and only then to make $(\mathsf {eb}_i \; \hat n_{K'})$ steps to the left (backwards).  At this point, and by construction, the polymer will find itself on $H_{i,1}'$. 
Continuing its evolution, the polymer will eventually reach from there the next hyperplane $H_{i,2}'$, again through 
$(\mathsf {ef}_i \; \hat n_{K'})$ steps to the right, and $(\mathsf {eb}_i \; \hat n_{K'})$ steps to the left. Since in this phase no restriction is imposed on the {\it order} of back- and forwardsteps, it could thus happen that the polymer first performs  all available steps to the left, in one fell swoop: this would increase the proximity of the polymer to $H_i$, with the hyperplane potentially even crossed for a second time. However, we claim that even in such worst case scenario, the polymer will find itself well to the right of $H_i$. In other words we claim that
\beq\label{boo'}
\mathsf{ef}_in_{K'}-2\mathsf{eb}_i n_{K'} >0,
\eeq
or, which is the same, that 
\beq \label{boo}
\mathsf{ef}_i-2\mathsf{eb}_i>0.
\eeq 
Indeed, it follows from  \eqref{elementaire} that
\beq\bea \label{bobo}
\mathsf{ef}_i-2\mathsf{eb}_i & = \frac{\mathsf{d}_i}{2}+\frac{1}{2K}-2\left(\frac{\mathsf d_i}{2}-\frac{1}{2K} \right)   \\
& =   \frac{1}{K} - \left( \frac{\mathsf d_i}{2} -\frac{1}{2K}\right) \\
& = \frac{1}{K}-\mathsf{eb}_i , 
\eea\eeq
the last step again by \eqref{elementaire}. Our new claim thus states that for large enough $K$,
\beq \label{fundi}
\frac{1}{K}-\mathsf{eb}_i > 0.
\eeq
To see this, we recall that by \eqref{e_1}, the number of effective backsteps between hyperplanes in the stretched phase satisfies
\beq \label{bs_2}
\mathsf{eb}_i=\sinh(\overline{\ma}_{i-1} \mathsf E)\sinh(\ma_{i} \mathsf E)\sinh(\underline{\ma}_i \mathsf E).
\eeq
Real analysis shows that 
\beq \label{m_sinh}
\arg \max_{y \in [0,1]} \sinh(y \mathsf E)\sinh((1-y)\mathsf E) = \frac{1}{2}\,.
\eeq
Furthermore, by \eqref{control_alpha}, 
\beq \label{control_alpha_2}
 \ma_i \mathsf E \leq \frac{1}{K}\,,
\eeq
which, together with an elementary large-$K$ Taylor expansion, implies that
\beq \label{t_sinh}
\sinh(\ma_i \mathsf E)=\ma_i \mathsf E+\frac{(\ma_i \mathsf E)^3}{6} \leq \frac{1}{K}+ \frac{1}{6 K^3}\leq \frac{2}{K}\,,
\eeq 
for $K \geq 1$. Using \eqref{t_sinh} in \eqref{bs_2} we get
\beq\bea\label{b2}
\mathsf{eb}_i &\leq \sinh(\overline{\ma}_i \mathsf E)\sinh(\underline{\ma}_i \mathsf E) \times \frac{2}{K} \\
&\leq \sinh\left( \frac{\mathsf E}{2} \right)^2  \times \frac{2}{K}. \\
\eea\eeq
the second inequality by \eqref{m_sinh}. The first term on the r.h.s. above can be easily estimated: 
\beq \bea \label{easy_est}
\sinh\left(  \frac{\mathsf E}{2} \right)^2  & = \frac{1}{4}\left( e^{\mathsf E/2}-e^{-\mathsf E/2} \right)^2  
= \frac{1}{4} \left( e^{\mathsf E} - 2 + e^{-\mathsf E} \right) \\
& = \frac{1}{2}\left( \cosh(\mathsf E) -1 \right) = \frac{1}{2} \left( \sqrt{1+ \sinh^2(\mathsf E)} -1 \right)\\ 
& = \frac{1}{2} \left( \sqrt{2} - 1\right),
\eea \eeq
the step before last by the Pythagorean's identity for hyperbolic functions, and the last since $\sinh(\mathsf E)=1$ by definition. In particular, we see that 
\beq
\sinh\left(  \frac{\mathsf E}{2} \right)^2  \leq \frac{1}{4}\,.
\eeq
Using this in \eqref{b2} we thus get $\mathsf{eb}_i \leq \frac{1}{2K}$, hence 
\beq
\frac{1}{K}-\mathsf{eb}_i \geq \frac{1 }{2K} > 0,
\eeq
settling claim \eqref{fundi}, and therefore \eqref{boo}. \\

\noindent Summarizing the upshot of these considerations, we thus see that {\it after} crossing an $H$-plane for the first time, the polymer will forever remain "to its right". But by symmetry, a similar line of reasoning  holds also for the case "to the left", i.e. for paths making $(\mathsf{eb}_i \hat n_{K'})$ steps to the left,  and then $(\mathsf{ef}_i \hat n_{K'})$ steps to the right {\it before} reaching such hyperplane. Lemma \ref{rem} is therefore established. 
\end{proof}

\begin{rem}\label{remloopless} Polymers in $\mathcal P_{n,K,K'}^{\sf rep}$ are, in fact, loopless: this follows from Lemma \ref{rem}, and the property that paths make no detours between $H$-planes.
\end{rem}

\subsection{Partitioning the energy} \label{energies_sec}
We will eventually implement the multiscale refinement of the second moment method \cite{kistler}, a procedure which involves a number of steps. The first, and key, step is to break the self-similarity of the underlying random field: this can be achieved here by allowing the first and last edges of the polymers to carry an unusually large fraction of the energy, and handling these on different footing. This procedure has already been succesfully implemented for the problem of (directed)  first passage percolation in \cite{kss}, see also Remark \ref{facteur_n} below for more on this issue.  

We need some additional notation: since a path $\pi \in \mathcal P_{n,K,K'}^{\sf rep}$ consists of a set of edges which uniquely characterises the vertices visited by the polymer, by a a slight abuse of notation we will denote by $\pi \cap H_i$ the vertices that lie both in $H_i$ and between two edges of the $\pi$-path. 

For a polymer $\pi \in \mathcal P_{n,K,K'}^{\sf rep}$, we begin by writing its energy as
\beq
X_\pi = \mathcal F_\pi + \left\{ X_m(\pi) + \left[ \sum_{j=m+1}^{K-m} X_{j-1, j}(\pi) \right] + X_{K-m+1}(\pi) \right\} + \mathcal L_\pi\,,
\eeq
with the following notational conventions: 
\begin{itemize}
\item $\mathcal F_\pi \defi X_{[\pi]_1}$ is the energy of the first edge of the path;
\item $X_m(\pi) \defi \sum_{j=2}^{ m \hat n_K} \xi_{[\pi]_j}$ is the energy of the substrand connecting the second visited vertex to the $m^{th}$-hyperplane, i.e. $\boldsymbol 0$ to the $m^{th}$-hyperplane, {\it but with the first edge excluded};  
\item For $i=m+1\dots K-m$, 
\beq\bea
X_{i-1, i}(\pi) &\defi X_\pi( \pi \cap H_{i-1},\pi \cap H_{i})
\eea\eeq
is the energy of the substrand connecting consecutive $H$-hyperplanes;
\item $X_{K-m+1}(\pi)$ is the energy of the substrand connecting the $(K-m)^{th}$-hyperplane to 
$\boldsymbol 1$, {\it but with the last edge excluded};
\item $\mathcal L_\pi$ is the energy of the last edge of the path. 
\end{itemize}
For $\e>0$, recalling $\{\mathsf a_i\}_{ i=1}^{K}$ solutions of \eqref{if_one} and the convention 
$\overline{\mathsf a}_m=\sum_{i\leq m} \mathsf a_i$, we set 
\beq 
\tilde \ma_{m,\e} \defi \overline{\ma}_{m}\left(\mathsf E+\frac{\epsilon}{5}\right)+\frac{\epsilon}{5}\,, 
\eeq 
and 
\beq 
\tilde \ma_{K-m+1,\e} \defi \tilde \ma_{m,\e}\,, 
\eeq
and for $i=m+1 \dots K-m$,
\beq
\ma_{i,\e}\defi \ma_{i}\left(\mathsf E+\frac{\epsilon}{5}\right) \,. 
\eeq
We then introduce the following subsets of polymers:
\beq \bea
\mathcal E_{n, K, K'}^{1, \e}  \defi \; & \pi \in \mathcal P_{n, K, K'}^{\sf rep} \; \text{\sf such that} \;
\mathcal F_\pi, \; \mathcal L_\pi \leq \e/5. 
\eea \eeq

\beq \bea
\mathcal E_{n, K, K'}^{2, \e}   \defi \; & \pi \in \mathcal P_{n, K, K'}^{\sf rep} \; \text{\sf such that} \\
& X_m(\pi), \; X_{K-m+1}(\pi) \leq \tilde \ma_{m,\e}, \\
& X_{i-1, i}(\pi) \leq  \ma_{i,\e} \; \text{\sf for}\;  i=m+1 \dots K-m\,.
\eea \eeq
Recalling that $\overline{\mathsf a}_m+ \sum_{i=m+1}^{K-m} \mathsf a_i +\underline{\mathsf a}_{K-m} =1$, 
we emphasize that the newly constructed subset consists of polymers with sub-energies 
\beq \label{sub_en_mk}
\overline{X}_{m}^{K-m+1}(\pi) \defi X_m(\pi) + \left[ \sum_{j=m+1}^{K-m} X_{j-1, j}(\pi) \right] + X_{K-m+1}(\pi)   \leq \mathsf E+ \frac{3}{5}\e\,,
\eeq
and with first resp. last edges carrying unusually large an energy (potentially up to $\e/5$).  At last, we consider the sub-subset
\beq \bea
\overline{\mathcal E}_{n, K, K'}^{\e} \defi \; & \mathcal E_{n, K, K'}^{1, \e}\cap \mathcal E_{n, K, K'}^{2, \e} \,.
\eea \eeq
Thus, by definition, the polymers in $\overline{\mathcal E}_{n, K, K'}^{\e}$ have energies less than $\mathsf E+\e$. A graphical rendition of this set is given in Figure \ref{energies_spread} below.  
\begin{figure}[!h]
    \centering
\includegraphics[scale=0.3]{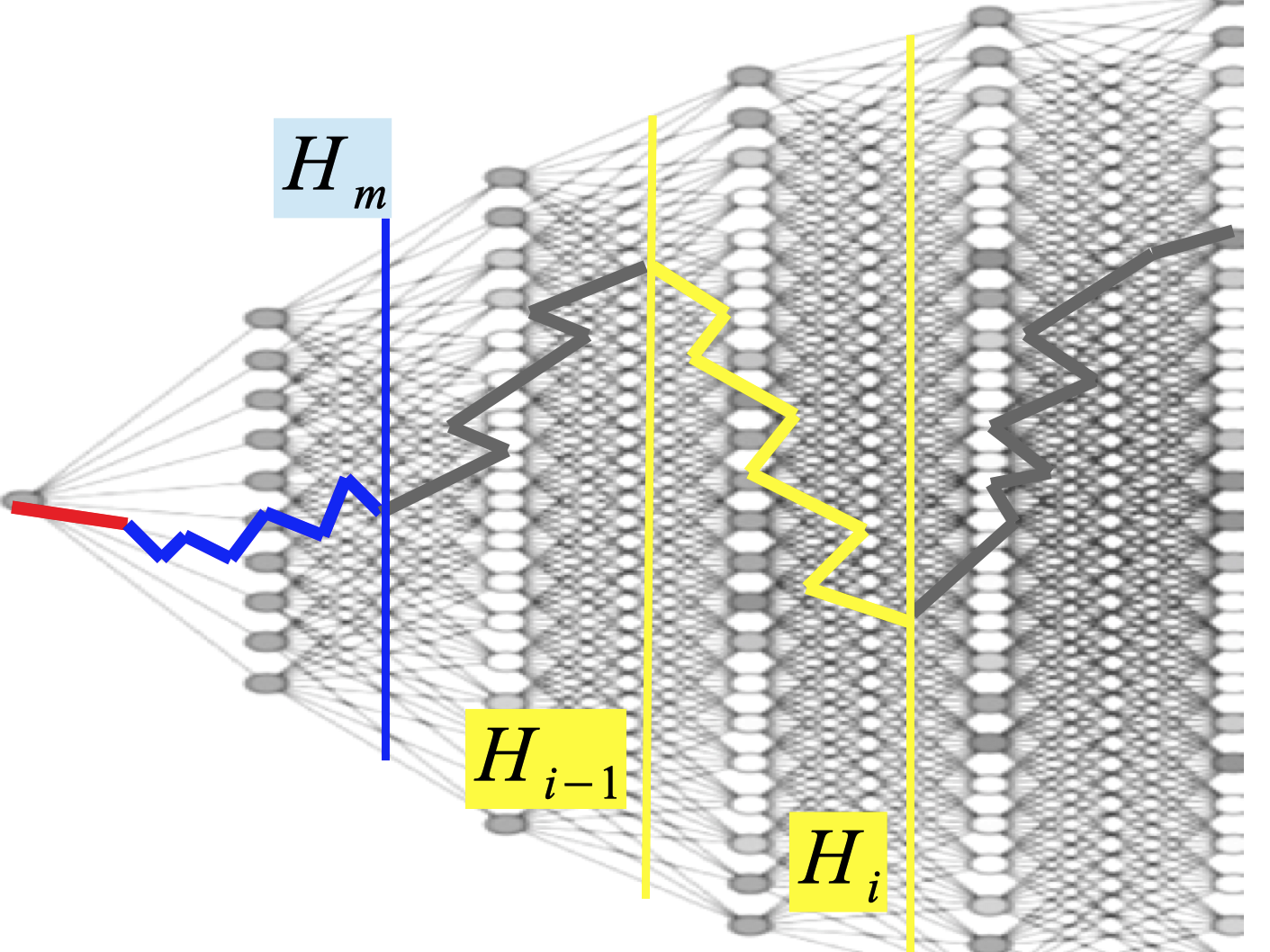}
    \caption{Distributing the energy in the first half of the hypercube. The first edge (red) has energy less than $\e/5$. The blue strand
is in the directed phase, and corresponds to $X_m(\pi) \leq \overline{\mathsf a}_m$. The yellow strand is in the stretched phase, it connects two consecutive $H$-hyperplanes with sub-energy less than ${\mathsf a}_{i,\e}$. For the second half of the hypercube, an analogous (mirror) picture holds. 
}
\label{energies_spread}
\end{figure}

\subsection{Connecting first and last region} \label{connec}
By definition, and recalling the inclusions \eqref{inclusions}, it clearly holds that 
\beq
\overline{\mathcal E}_{n, K, K'}^{\e} \subset \mathcal E_{n, K}^{\e}\,.
\eeq
In particular, non-emptiness of $\overline{\mathcal E}_{n, K, K'}^{\e} $ will  immediately yield our main Theorem \ref{geo_optimal_paths}, and this is indeed the route we take. Precisely, we will show that one can connect the first and last edges through polymers satisfying the energy requirements in the directed/stretched phases.  
To see how this goes, we begin with the observation that
\beq \bea \label{twrd}
& \PP\left( \#\overline{\mathcal E}_{n, K, K'}^{\e} \geq 1 \right)  \geq 
\PP\left( \# \overline{\mathcal E}_{n, K, K'}^{\e} \geq 1, \; \# \mathcal E_{n, K, K'}^{1, \e}   \geq \left \lfloor \frac{ \E \# \mathcal E_{n, K, K'}^{1, \e}  }{2} \right \rfloor \right) \\ 
& \quad = \PP\left( \# \overline{\mathcal E}_{n, K, K'}^{\e} \geq 1 \middle| \# \mathcal E_{n, K, K'}^{1, \e}   \geq \left \lfloor \frac{\E \# \mathcal E_{n, K, K'}^{1, \e} }{2}  \right \rfloor \right) \PP\left( \# \mathcal E_{n, K, K'}^{1, \e}   \geq \left \lfloor \frac{\E \# \mathcal E_{n, K, K'}^{1, \e} }{2}  \right \rfloor \right)\,.
\eea \eeq
By independence, it clearly holds that 
\beq \bea
\E \# \mathcal E_{n, K, K'}^{1, \e} &= \PP\left( \mathcal F_\pi \leq {\e \over 5} \right) \PP\left( \mathcal L_\pi \leq {\e \over 5} \right)  \# \mathcal P_{n, K, K'}^{\sf rep} = C(\e)^2  \# \mathcal P_{n, K, K'}^{\sf rep},
\eea \eeq
where 
\beq
C(\e) \defi  1- \exp(-\e/5)\,.
\eeq
We now claim that 
\beq\bea \label{new_g_2}
\lim_{n \to \infty} \PP\left( \# \mathcal E_{n, K, K'}^{1, \e}   \geq \left \lfloor \frac{\E \# \mathcal E_{n, K, K'}^{1, \e} }{2}  \right \rfloor \right) = 1.
\eea\eeq
Indeed, by Chebycheff's inequality, and for $\de >0$,
\beq\bea \label{cheby_1}
\PP\left(\left|\frac{\# \mathcal E_{n, K, K'}^{1, \e}}{\E(\# \mathcal E_{n, K, K'}^{1, \e})}-1\right|\geq \de \right)\leq \frac{1}{\de^2}\left\{ \frac{\E\left({\# \mathcal E_{n, K, K'}^{1, \e}}^2\right)}{\E\left(\# \mathcal E_{n, K, K'}^{1, \e}\right)^2}-1\right\}\,.
\eea\eeq
Let now $\pi$ $\in \mathcal P_{n, K, K'}^{\sf rep}$ and denote by 
\beq
f_{\pi}(n,k) \defi \text{\sf the number of paths in} \; \mathcal P_{n, K, K'}^{\sf rep} \; \text{\sf sharing}\;  k \; \text{\sf weigthed edges with} \; \pi.
\eeq
Since for paths in $\mathcal E_{n, K, K'}^{1, \e}$ only the first and the last edges are weighted,
\beq \bea
\E\left({\# \mathcal E_{n, K, K'}^{1, \e}}^2\right) \leq \E\left({\# \mathcal E_{n, K, K'}^{1, \e}}\right)^2
+ {\# \mathcal P_{n, K, K'}^{\sf rep}} \left\{ C(\e)^3 f_{\pi}(n,1)+C(\e)^2 f_{\pi}(n,2)\right\}\,,
\eea \eeq
the first term on the r.h.s. corresponding to the case of $k=0$ shared edges. Using that $C(\e) \leq 1$ and that
$f_\pi(n, 2) \leq f_\pi(n,1)$, the above becomes
\beq \bea
\E\left({\# \mathcal E_{n, K, K'}^{1, \e}}^2\right) & \leq \E\left({\# \mathcal E_{n, K, K'}^{1, \e}}\right)^2 + 2 
{\# \mathcal P_{n, K, K'}^{\sf rep}} f_{\pi}(n,1)\,. 
\eea \eeq
Therefore,  for the r.h.s. of \eqref{cheby_1} we have
\beq \label{cheby_2}
\frac{\E\left({\# \mathcal E_{n, K, K'}^{1, \e}}^2\right)}{\E\left(\# \mathcal E_{n, K, K'}^{1, \e}\right)^2}-1
\leq \frac{2 
 {\# \mathcal P_{n, K, K'}^{\sf rep}} f_{\pi}(n,1)}{  \E\left(\# \mathcal E_{n, K, K'}^{1, \e}\right)^2 } = \frac{2}{C(\e)^4} \frac{f_{\pi}(n,1)}{ \# \mathcal P_{n, K, K'}^{\sf rep}}\,.
\eeq
Let now $f_{\pi}^{l}(n,1)$ be the number of paths which share one edge with $\pi$ on the {\it left} of the hypercube. Clearly, $f_{\pi}^{l}(n,1) = 2 f_{\pi}(n,1)$, hence
\beq \label{cheby_3}
\frac{f_{\pi}(n,1)}{ \# \mathcal P_{n, K, K'}^{\sf rep}} = 2 \frac{f_{\pi}^l(n,1)}{ \# \mathcal P_{n, K, K'}^{\sf rep}} \leq 2 \frac{(m \hat n_K-1)!}{(m \hat n_K)!} = \left(\frac{2K}{m}\right) \frac{1}{n},
\eeq
where for the key inequality we have used that there are $(m \hat n_K)!$ possibilities to reach a given (admissible) vertex on the $H_m$-plane, but specifying the first edge reduces such possibilities to $(m \hat n_K-1)!$. Using \eqref{cheby_3} in \eqref{cheby_2} and then \eqref{cheby_1} we thus obtain
\beq
\PP\left(\left|\frac{\# \mathcal E_{n, K, K'}^{1, \e}}{\E(\# \mathcal E_{n, K, K'}^{1, \e})}-1\right|\geq \de \right)\lesssim 
\frac{1}{n} \longrightarrow 0\,,
\eeq
as $n\uparrow \infty$, which settles claim \eqref{new_g_2}. Using the latter in  \eqref{twrd} then yields
\beq \bea \label{new_g_3}
 \PP\left( \#\overline{\mathcal E}_{n, K, K'}^{\e} \geq 1 \right)  
 & \geq  \PP\left( \# \overline{\mathcal E}_{n, K, K'}^{\e} \geq 1 \middle| \# \mathcal E_{n, K, K'}^{1, \e}   \geq \left \lfloor \frac{\E \# \mathcal E_{n, K, K'}^{1, \e} }{2}  \right \rfloor \right)-o_n(1).
\eea \eeq
Now, for any $\mathsf J \leq  \# \mathcal P_{n, K, K'}^{\sf rep}$, it holds that
\beq \label{new_g_4}
\PP\left( \# \overline{\mathcal E}_{n, K, K'}^{\e} \geq 1 \middle| \# \mathcal E_{n, K, K'}^{1, \e} \geq \mathsf J \right)
\geq \PP\left( \# \overline{\mathcal E}_{n, K, K'}^{\e} \geq 1 \middle| \# \mathcal E_{n, K, K'}^{1, \e} = \mathsf J \right)\,,
\eeq
since the more paths survive the "thinning procedure" via the energy condition on first and last edge, the higher the chance to find at least a connecting polymer which satisfies the imposed energy requirements. See Figure \ref{l_f} for a graphical rendition. 
\begin{figure}[!h]
    \centering
\includegraphics[scale=0.32]{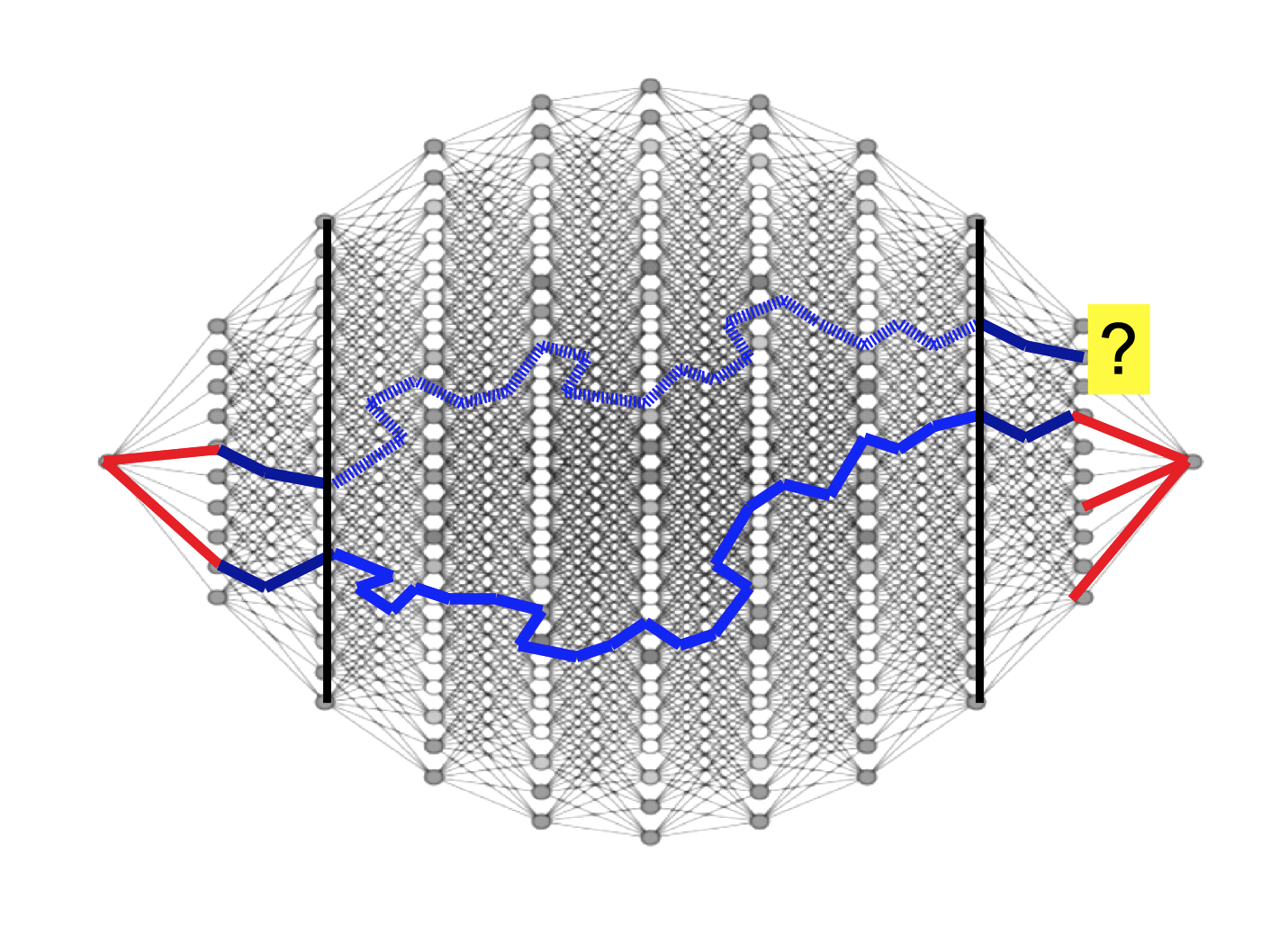}
    \caption{The first and last edges  carrying an energy less than $\e/5$ (hence surviving the thinning procedure) are drawn in red. The continuous blue strand manages to connect these edges while satisfying the energy constraints, whereas the dashed strand does not.}
\label{l_f}
\end{figure}

Using \eqref{new_g_4} with 
\beq
\mathsf J \defi \left\lfloor\frac{\E\# \mathcal E_{n, K, K'}^{1, \e}}{2}\right\rfloor \,, 
\eeq
and by the Paley-Zygmund inequality, we thus get
\beq\label{PZC}
\PP\left(\# \overline{\mathcal E}_{n, K, K'}^{\e} \geq 1 \middle| \# \mathcal E_{n, K, K'}^{1, \e} = \mathsf J \right) \geq \frac{\E\left(\# \overline{\mathcal E}_{n, K, K'}^{\e} \middle| \# \mathcal E_{n, K, K'}^{1, \e} = \mathsf J \right)^2}{\E\left( {\#\overline{\mathcal E}_{n, K, K'}^{\e}}^2 \middle| \# \mathcal E_{n, K, K'}^{1, \e} = \mathsf J \right)}.
\eeq
Consider now  {\it any} deterministic set $\mathcal J \subset \mathcal P_{n, K, K'}^{\sf rep}$ with cardinality $\# \mathcal J =\mathsf J$, and the subset
\beq \label{intersectingg}
 {\mathcal E}_{n, K, K'}^{\e} \defi \mathcal E_{n, K, K'}^{2, \e} \cap \mathcal J \,,
\eeq
which is obtained from $\mathcal E_{n, K, K'}^{2, \e}$ via thinning procedure. We shorten $\#{\mathcal E}_{n, K, K'}^{\e} \defi  {\mathcal N}_{n, K, K'}^{\e}$. By independence of the sigma algebras issued from first and last edges, and the sigma algebra involving all other edges, we clearly have that
\beq\label{PZC'}
\E\left(\# \overline{\mathcal E}_{n, K, K'}^{\e} \middle| \# \mathcal E_{n, K, K'}^{1, \e} = \mathsf J \right)= \E\left({\mathcal N}_{n, K, K'}^{\e}\right)
\eeq
and 
\beq\label{PZC''}
\E\left( {\#\overline{\mathcal E}_{n, K, K'}^{\e}}^2 \middle| \# \mathcal E_{n, K, K'}^{1, \e} = \mathsf J \right) = \E\left( {{\mathcal N}_{n, K, K'}^{\e}}^2 \right).
\eeq
Using \eqref{PZC'} and \eqref{PZC''} in \eqref{PZC}, and by \eqref{new_g_3}, we see that
\beq \bea 
 \PP\left( \#\overline{\mathcal E}_{n, K, K'}^{\e} \geq 1 \right)  
 & \geq  \frac{\E\left({\mathcal N}_{n, K, K'}^{\e}\right)^2}{\E\left( {{\mathcal N}_{n, K, K'}^{\e}}^2 \right)}-o_n(1).
\eea \eeq
Therefore, our main result Theorem \ref{fpp1}, will be an immediate consequence of
 
\begin{thmbis}{geo_optimal_paths}\label{connecting}  
For $\e>0$ there exists $K = K(\e) \in \N$ such that 
\beq \label{gainingind}
\lim_{n\to \infty} \frac{\E\left({\mathcal N}_{n, K, K'}^{\e}\right)^2}{\E\left( {{\mathcal N}_{n, K, K'}^{\e}}^2 \right)} =1,
\eeq
for any $K'>2 \log(2) \mathsf{L}K^2$.
\end{thmbis}


\section{$\Pi$ vs. $\mathcal P$, and a lower bound to the first moment} \label{changing_deal}

In Sections \ref{sprinckle}-\ref{energies_sec} we have altered the path-properties derived in Section \ref{gathering_insights}, and this of course has relevant consequences. The following result precisely quantifies the changes to the first moment as given in Theorem \ref{choosing_d} (which has been instrumental to all our considerations so far) once these modifications have been taken into account. 

\begin{thmbis}{choosing_d} \label{choosing_dprime}
For $\e>0$, shorten
\beq \label{defi_eee}
\e_E \defi \frac{\epsilon}{5 \mathsf E}, \qquad \e_{m, \mathsf E} \defi \frac{\epsilon}{5 \mathsf E}+\frac{\e}{5\overline{\ma}_m \mathsf E}\,.
\eeq
Let furthermore
\beq \label{entrop_stret}
S_{n,K, m} \defi \exp-n \left( \frac{1}{\sqrt{2}K}+\frac{\sqrt{2}m(m-1)}{K^2} \right)\,, \qquad R_{n,K} \defi \exp\left(- \frac{n}{K^2} \right)\,,
\eeq
and set
\beq \label{defi_ccc}
C_{n,K, m} \defi R_{n,K} \times S_{n,K, m}\,.
\eeq
Then for any $K'>2 \log(2)\mathsf{L} K^2$,
\beq\bea
\E\left( {\mathcal N}_{n, K, K'}^{\e} \right) &\geq C_{n,K,m} {(1+\e_{ \mathsf E})}^{\sum_{i=m+1}^{K-m}n\mathsf d_i} {(1+\e_{m, \mathsf E})}^{2m    
\hat n_K }  \frac{Q_n}{P_n}\,,
\eea\eeq
where $Q_n$ and $P_n$ are finite degree polynomials. 
\end{thmbis}

\begin{rem} It will become clear in the course of the proof that the $S$-term in Theorem \ref{choosing_dprime}  encodes the entropic cost for \emph{stretching} the paths in $\Pi$ in order to construct $\mathcal P_{n, K}$, whereas the  $R$-term relates to the entropic cost for rendering the $H$-planes \emph{repulsive}, i.e. in order to construct $\mathcal P_{n, K, K'}^{\sf rep}$ out of $\mathcal P_{n,K}$.
\end{rem}

\begin{proof}[Proof of Theorem \ref{choosing_dprime}]
We begin by computing the cardinality of $\mathcal P_{n, K}$. To do so, we recall that paths in this set are {\it directed} in the $m$ first (and last) H-planes: since there are $\left( m  \hat n_K \right)!$ ways to reach a vertex on the $m^{th}$-hyperplane starting from $\boldsymbol{0}$, and 
\beq 
\dbinom{n}{     m  \hat n_K   }
\eeq 
vertices on such hyperplane, we have, altogether,
\beq\bea
\left( m  \hat n_K   \right)! \dbinom{n}{  m  \hat n_K}
\eea\eeq
subpaths connecting $\boldsymbol{0}$ to $H_m$. Furthermore, there are  are 
\beq\bea
\left( m  \hat n_K   \right)!
\eea\eeq
subpaths connecting a given vertex in $H_{K-m}$ to $\boldsymbol 1$.  

As for the {\it stretched} phase, we will heavily rely on the fact already mentioned in Figure \ref{full_monty}, namely that a natural representation of paths in terms of permutations is available.  First we remark that for any two vertices $\boldsymbol v, \boldsymbol w$ of the hypercube, 
\beq \bea\label{Stretched}
\# \; \text{stretched paths between $\boldsymbol v$ and $\boldsymbol w$} =(n d(\boldsymbol v, \boldsymbol w))!\,,
\eea \eeq
and therefore, by definition of $\mathcal P_{n, K}$, 
\beq \label{PnK}
\# \mathcal P_{n, K}= \underbrace{\left( m  \hat n_K   \right)! \dbinom{n}{  m  \hat n_K}}_{\text{directed}} \underbrace{\left(\sum_{(\star_m)} \prod_{i=m+1}^{K-m} (n \mathsf d_i)! \right)}_{\text{stretched}}\underbrace{\left( m  \hat n_K   \right)!}_{\text{directed}},
\eeq
where the $(\star_m)$-sum runs over all possible vertices $\boldsymbol v \in H_i$. By definition, the  subpaths in $\mathcal P_{n, K}$ going through a given vertex of the $H_{i-1}$-plane can reach the same number of vertices on the  $H_{i}$-plane as  the subpaths in $\Pi_{\{1\dots K\}}^{\boldsymbol d}[\boldsymbol 0  \to \boldsymbol 1]$: the $(\star_m)$-sum thus runs over the same vertices as the $(\star)$-sum  in \eqref{cardi}, hence  
\beq \bea \label{card_m}
\# (\star_m)  & =   \prod_{i=m+1}^{K-m} \dbinom{\frac{i-1}{K} n}{\mathsf{eb}_i n} \dbinom{\left(1-\frac{i-1}{K}\right) n}{\mathsf{ef}_i n}.
\eea \eeq
Combining \eqref{PnK} and \eqref{card_m} thus yields
\beq \label{cardPnK}
\# \mathcal P_{n, K}=\left( m  \hat n_K   \right)!^2 \dbinom{n}{  m  \hat n_K}  \prod_{i=m+1}^{K-m} \dbinom{\frac{i-1}{K} n}{\mathsf{eb}_i n} \dbinom{\left(1-\frac{i-1}{K}\right) n}{\mathsf{ef}_i n} (n {\mathsf d}_i)!\,.
\eeq
We now quantify the difference in cardinality between $\mathcal P_{n,K}$ and $\mathcal P_{n,K,K'}$, and then, in a 
second step, between $\mathcal P_{n,K,K'}$  and $\mathcal P_{n, K, K'}^{\sf rep}$. To do so, the following observation is helpful:  in the stretched phase, since by  \eqref{b-f} it holds that $\mathsf{eb}_i+\mathsf{ef}_i = \mathsf d_i$, we may re-write the r.h.s. of \eqref{Stretched} as 
\beq \bea \label{simply_so}
(n \mathsf d_i)!=(n\mathsf{eb}_i)!(n\mathsf{ef}_i)!\dbinom{n \mathsf d_i}{n\mathsf{eb}_i}.
\eea \eeq
This elementary algebraic identity can be given an {\it interpretation} which proves useful for the purpose of computing the cardinality of $\mathcal P_{n, K, K'}$. To see this, let us assume that each step of the polymer is a ball which is both coloured {\it and}  labeled: backsteps are red whereas forward steps are blue; the labels correspond to which coordinate switches its value during the considered step: there are thus $(n \mathsf{eb}_i)$ labels for the red balls, and $(n \mathsf{ef}_i)$ labels for the blue balls. The first factorial on the r.h.s. of \eqref{simply_so}  then stands for the number of possible ways of listing the red balls while discriminating according to the labels, 
and similarly for the second factorial corresponding to the blue balls. Finally, the binomial factor on the r.h.s. of \eqref{simply_so} accounts for the number of ways to place the red and blue balls, but {\it without} discriminating among labels. 

Now, the subset $\mathcal P_{n,K,K'}$ is constructed out of $\mathcal P_{n, K}$ by adding an additional layer of coarse graining,  and modifying the order of appearance of balls while discriminating according to their colors, but disregarding the labels.  Adapting the interpretation of \eqref{simply_so} discussed in the previous paragraph, it is clear that there are now
\beq \bea\label{modif1}
(n\mathsf{eb}_i)!(n\mathsf{ef}_i)!{\dbinom{   \mathsf d_i \hat n_{K'} }{     \mathsf{eb}_i \hat n_{K'}     }}^{K'}.
\eea \eeq
subpaths in $\mathcal P_{n,K,K'}$ connecting two vertices in $H_{i-1}$ and $H_{i}$ at Hamming distance $n \mathsf d_i$.

The subset $\mathcal P_{n, K, K'}^{\sf rep}$ differs from $\mathcal P_{n,K,K'}$ in that the order of backsteps and forward steps between $H_{i-1}$ and $H_{i-1,1}'$, and between $H_{i-1,K'-1}'$ and $H_{i}$, is totally specified. This evidently reduces the cardinality: instead of \eqref{modif1}, there are only
\beq \bea\label{modif2}
(n\mathsf{eb}_i)!(n\mathsf{ef}_i)!{\dbinom{   \mathsf d_i \hat n_{K'} }{     \mathsf{eb}_i \hat n_{K'}     }}^{K'-2}.
\eea \eeq
subpaths between any two given vertices connecting the $H_{i-1}$ and $H_{i}$ hyperplanes. 

\noindent To compare quantitatively the cardinality of all these sets we write
\beq \bea\label{quotient}
\frac{\# \mathcal P_{n, K}}{\# \mathcal P_{n, K, K'}^{\sf rep}}=\frac{\# \mathcal P_{n, K}}{\# \mathcal P_{n, K, K'}}\times \frac{\# \mathcal P_{n, K,K'}}{\# \mathcal P_{n, K, K'}^{\sf rep}}.
\eea \eeq
By  \eqref{modif1}, it holds that
\beq \bea\label{quotient1}
\frac{\# \mathcal P_{n, K}}{\# \mathcal P_{n, K, K'}}& = \prod_{i=m+1}^{K-m}
\frac{(n \mathsf d_i)!}{(n\mathsf{eb}_i)!(n\mathsf{ef}_i)!{\dbinom{   \mathsf d_i \hat n_{K'} }{     \mathsf{eb}_i \hat n_{K'}     }}^{K'}}\\
&=\prod_{i=m+1}^{K-m} \frac{(n\mathsf{eb}_i)!(n\mathsf{ef}_i)!\dbinom{n \mathsf d_i}{n\mathsf{eb}_i}}{(n\mathsf{eb}_i)!(n\mathsf{ef}_i)!{\dbinom{  \mathsf d_i \hat n_{K'} }{   \mathsf{eb}_i \hat n_{K'}  } }^{K'}}\\
&\lesssim \prod_{i=m+1}^{K-m} \frac{\sqrt{2\pi n \mathsf d_i}}{\sqrt{2\pi n\mathsf{eb}_i}\sqrt{2\pi n\mathsf{ef}_i}} {\left(\frac{\sqrt{2\pi  \mathsf{eb}_i \hat n_{K'}  }\sqrt{2\pi 
 \mathsf{ef}_i \hat n_{K'} }}{\sqrt{2\pi   \mathsf{d}_i \hat n_{K'} }} \right)}^{K'}\,,
\eea \eeq
the last step by elementary Stirling approximation (this time including the lower order, polynomial terms). The r.h.s. of \eqref{quotient1} is, up to irrelevant numerical constant, {\it at most} 
\beq \bea\label{quotient1'}
\eqref{quotient1} \lesssim \prod_{i=m+1}^{K-m} n^{\frac{K'-1}{2}}=  n^{\frac{(K'-1)(K-2m)}{2}}.
\eea\eeq
Furthermore, one has
\beq\bea \label{obsi_o}
\frac{\# \mathcal P_{n, K,K'}}{\# \mathcal P_{n, K, K'}^{\sf rep}}&= \prod_{i=m+1}^{K-m}{\dbinom{   \mathsf{d}_i \hat n_{K'}  }{\mathsf{eb}_i \hat n_{K'} }^2} \\
&\lesssim {\left\{\prod_{i=m+1}^{K-m}\left(1-\frac{\mathsf{ef}_i}{\mathsf d_i}\right)^{\mathsf d_i-\mathsf{ef}_i}\left(\frac{\mathsf{ef}_i}{\mathsf d_i}\right)^{\mathsf{ef}_i}\right\}}^{ -2 \hat n_{K'} }\prod_{i=m+1}^{K-m}\frac{K' \mathsf d_i}{2\pi n \mathsf{eb}_i \mathsf{ef}_i},
\eea\eeq
the last inequality again by Stirling approximation. Since the term in the curly bracket is raised to a negative power, we will use the following lower bound 
\beq\bea \label{obsi}
\prod_{i=m+1}^{K-m}\left \{\left(1-\frac{\mathsf{ef}_i}{\mathsf d_i}\right)^{1-\frac{\mathsf{ef}_i}{\mathsf d_i}}\left(\frac{\mathsf{ef}_i}{\mathsf d_i}\right)^{\frac{\mathsf{ef}_i}{\mathsf d_i}}\right \}^{\mathsf d_i} \geq \prod_{i=m+1}^{K-m}{\left(\frac{1}{2}\right)}^{\mathsf d_i} \geq {\left(\frac{1}{2}\right)}^{\mathsf L}\,,
\eea\eeq
where the second inequality holds true since the function  $x \mapsto (1-x)^{1-x} x^x$  is convex, and attains its minimal value $1/2$ in $x=1/2$, as can be plainly checked. Plugging the bound \eqref{obsi} in \eqref{obsi_o} then yields
\beq\bea\label{quotient2}
\frac{\# \mathcal P_{n, K,K'}}{\# \mathcal P_{n, K, K'}^{\sf rep}}&\lesssim n^{-(K-2m)} \exp \left( \frac{\mathsf L n}{K'}  2 \log 2 \right).
\eea\eeq
Remark that for any $K'>2K^2L\log 2$, it holds that 
\beq\bea\label{quotient3}
 \exp \left( \frac{\mathsf L n}{K'}  2 \log 2 \right) \leq \exp\left( \frac{n}{K^2}\right).
\eea\eeq
Combining \eqref{quotient}, \eqref{quotient1'}, \eqref{quotient2} and \eqref{quotient3} therefore implies that the entropic cost for rendering the hyperplanes repulsive is 
\beq \bea\label{quotient4}
\frac{\# \mathcal P_{n, K}}{\# \mathcal P_{n, K, K'}^{\sf rep}}\lesssim n^{(\frac{K'-1}{2}-1)(K-2m) } \times 
\exp\left( \frac{n}{K^2}\right).
\eea \eeq
Using \eqref{cardPnK} in \eqref{quotient4} then yields
\beq \bea \label{cardrep}
\# \mathcal P_{n, K, K'}^{\sf rep}\gtrsim (  m  \hat n_K)!^2\dbinom{n}{  m  \hat n_K}\prod_{i=m+1}^{K-m} (n \mathsf d_i)!\dbinom{\frac{i-1}{K} n}{\mathsf eb_i n} \dbinom{(1-\frac{i-1}{K}) n}{\mathsf ef_i n}\times \frac{R_{n,K}}{P_n},
\eea \eeq
where we have shortened 
\beq
P_n \defi n^{(\frac{K'-1}{2}-1)(K-2m)}, \quad \quad  R_{n, K} \defi \exp\left( -\frac{n}{K^2}\right).
\eeq
Recall that by Remark \ref{remloopless},  polymers in $\mathcal P_{n, K, K'}^{\sf rep}$ are loopless: this property, the ensuing  independence of the sub-energies, and \eqref{cardrep} thus yield
\beq\bea\label{s2}
\E(\# \mathcal E_{n, K, K'}^{2, \e}) &\gtrsim(  m  \hat n_K)!\PP\left(X_m(\pi)\leq \ma_{m,\e}\right)\dbinom{n}{  m  \hat n_K}\\
& \hspace{2cm} \times \prod_{i=m+1}^{K-m} (n \mathsf d_i)! \PP\left(X_{i-1,i} \leq \ma_{i,\e}\right) \dbinom{\frac{i-1}{K} n}{\mathsf eb_i n} \dbinom{(1-\frac{i-1}{K}) n}{\mathsf ef_i n}\\
& \qquad \times (   m  \hat n_K )!\PP\left(X_{K-m+1}\leq \ma_{m,\e}\right)\frac{R_{n,K}}{P_n}.\\
\eea\eeq
Further, recalling that by the thinning procedure, it holds 
\beq\bea\label{s2'}
 \E\left({\mathcal N}_{n, K, K'}^{\e}\right)=\frac{C(\e)^2}{2}E(\# \mathcal E_{n, K, K'}^{2, \e})\,,
\eea\eeq 
and by the usual tail estimates, we thus see that
\beq\bea\label{s3}
&\E({\mathcal N}_{n, K, K'}^{\e})\gtrsim  \prod_{i=m+1}^{K-m}{(\ma_{i,\e})}^{nd_i} \dbinom{\frac{i-1}{K} n}{\mathsf{eb}_i n} \dbinom{(1-\frac{i-1}{K}) n}{\mathsf{ef_i} n}\dbinom{n}{  m  \hat n_K} {\ma_{m,\e}}^{2   m  \hat n_K -2} (  m  \hat n_K)^2 \frac{R_{n,K}}{P_n}\,.
\eea\eeq
But since $(m  \hat n_K)^2{\ma_{m,\e}}^{-2}>1$ for $n$ large enough, we have, altogether, that
\beq\bea\label{s4}
 \E\left({\mathcal N}_{n, K, K'}^{\e}\right) \gtrsim& \prod_{i=m+1}^{K-m} {(\ma_i \mathsf E)}^{n \mathsf d_i} \dbinom{\frac{i-1}{K} n}{\mathsf{eb}_i n} \dbinom{(1-\frac{i-1}{K}) n}{\mathsf{ef}_i n} \times \\
& \hspace{1cm} \times \dbinom{n}{   m  \hat n_K} { \left(\overline{\ma}_{m} \mathsf E\right)}^{2   m  \hat n_K}
 {(1+\e_{\mathsf E})}^{\sum_{i=m+1}^{K-m}n \mathsf d_i}{(1+\e_{m, \mathsf E})}^{2  m  \hat n_K  } \frac{R_{n,K}}{P_n}\,.
\eea\eeq

The first term on the r.h.s. of \eqref{s4} is reminiscent of the expression appearing in Theorem \ref{choosing_d}, but contrary to the latter, we are facing here  a product which runs over the indeces $i=m+1 \dots K-m$ only. The natural idea is thus to modify and then extend this partial product to a full product in order to  exploit the control already established in Theorem \ref{choosing_d}. To do so we first note that, since on the positive axis it holds that $x \geq \tanh(x)$, 
\beq
{(\ma_i \mathsf E)}^{n \mathsf d_i}  \geq {\tanh(\ma_i \mathsf E)}^{n \mathsf d_i}, \quad \text{and} \quad { \left(\overline{\ma}_{m} \mathsf E\right)}^{2   m  \hat n_K} \geq  {\tanh(\overline{\ma}_{m} \mathsf E)}^{2   m  \hat n_K}\,.
\eeq
Using this in \eqref{s4} yields
\beq\bea \label{sss}
 \E\left({\mathcal N}_{n, K, K'}^{\e}\right) \gtrsim& \prod_{i=m+1}^{K-m} {\tanh(\ma_i \mathsf E)}^{n \mathsf d_i} \dbinom{\frac{i-1}{K} n}{\mathsf{eb_i}n} \dbinom{(1-\frac{i-1}{K}) n}{\mathsf{ef}_i n} \\
& \quad \times \dbinom{n}{   m  \hat n_K} {\tanh(\overline{\ma}_{m} \mathsf E)}^{2   m  \hat n_K} {(1+\e_{\mathsf E})}^{\sum_{i=m+1}^{K-m}n \mathsf d_i}{(1+\e_{m, \mathsf E})}^{2   m  \hat n_K}\frac{R_{n,K}}{P_n}\,.
\eea \eeq
The new (partial) product is closer yet not quite the same as that appearing in  Theorem \ref{choosing_d}, so we  artificially introduce some $\cosh$-terms which however leave the r.h.s. above as a whole unaltered. Precisely, we rewrite \eqref{sss} as
\beq \bea \label{ssss}
& \E\left({\mathcal N}_{n, K, K'}^{\e}\right)  \gtrsim \prod_{i=m+1}^{K-m} {\tanh(\ma_i \mathsf E)}^{n \mathsf d_i}{\left(\frac{\cosh(\ma_i \mathsf E)}{\cosh(\ma_i \mathsf E)}\right)}^{n} \dbinom{\frac{i-1}{K} n}{\mathsf{eb_i}n} \dbinom{(1-\frac{i-1}{K}) n}{\mathsf{ef}_i n} \times \\
& \quad \times \dbinom{n}{   m  \hat n_K} {\tanh(\overline{\ma}_{m} \mathsf E)}^{2   m  \hat n_K} {\left(\frac{\cosh(\overline{\ma}_{m} \mathsf E)}{\cosh(\overline{\ma}_{m} \mathsf E)}\right)}^{2n}{(1+\e_{\mathsf E})}^{\sum_{i=m+1}^{K-m}n \mathsf d_i}{(1+\e_{m, \mathsf E})}^{2   m  \hat n_K}\frac{R_{n,K}}{P_n};
\eea\eeq
We can now move to the aforementioned procedure of extending the product to {\it all} indeces $i=1 \dots K$. This naturally requires a good control of the missing terms, i.e. for $i\leq m$ (a case which is referred to below as {\sf First}), and for $i\geq K-m+1$ ({\sf Second case}). \\

\noindent {\sf First case.} We begin noting that by the Evolution Lemma \ref{f},
\beq\bea\label{eq1}
\left[ \prod_{i=1}^{m} g_{i,K}( \mathsf d_i) \right]^n &={\left(\frac{\sinh(\overline{\ma}_{m} \mathsf E)}{\frac{m}{K}}\right)}^{m \hat n_K} {\left(\frac{\cosh(\overline{\ma}_{m} \mathsf E)}{1-\frac{m}{K}}\right)}^{n-m \hat n_K}\\
&={\tanh(\overline{\ma}_{m} \mathsf E)}^{m \hat n_K}\cosh(\overline{\ma}_{m} \mathsf E)^n \frac{n^n}{{( m\hat n_{K})}^{m \hat n_K}{(n- m\hat n_{K} )}^{n-m \hat n_K}}\,,
\eea\eeq
the second equality by elementary rearrangement. But  by "reverse" Stirling-approximation,
\beq
\frac{n^n}{{( m\hat n_{K})}^{m \hat n_K}{(n- m\hat n_{K} )}^{n-m \hat n_K}} \propto \sqrt{n} {n \choose m \hat n_K},
\eeq
and therefore 
\beq \label{eq1.5}
\left[ \prod_{i=1}^{m} g_{i,K}( \mathsf d_i) \right]^n  \propto \sqrt{n} {\tanh(\overline{\ma}_{m} \mathsf E)}^{m \hat n_K} {\cosh(\overline{\ma}_{m} \mathsf E)}^{n} \dbinom{n}{   m  \hat n_K }\,.
\eeq
Furthermore, by definition of the $g$-functions, and taking into account the lower orders in the Stirling approximation of the binomial factors, one also plainly checks that
\beq\bea\label{eq1.6}
\left[ \prod_{i=1}^{m} g_{i,K}( \mathsf d_i)\right]^n & \propto \sqrt{n} n^{m-1} \prod_{i=1}^{m} {\tanh(\ma_i \mathsf E)}^{n \mathsf d_i}{\cosh(\ma_i \mathsf E)}^n \dbinom{\frac{i-1}{K} n}{\mathsf{eb_i} n} \dbinom{(1-\frac{i-1}{K}) n}{\mathsf{ef}_i n} 
\,.
\eea\eeq
Equating  \eqref{eq1.5} and \eqref{eq1.6} therefore yields the asymptotic identity
\beq \bea \label{antoi}
& {\tanh(\overline{\ma}_{m} \mathsf E)}^{m \hat n_K} {\cosh(\overline{\ma}_{m} \mathsf E)}^{n} \dbinom{n}{   m  \hat n_K }  \\
& \hspace{2cm} \propto n^{m-1} \prod_{i=1}^{m} {\tanh(\ma_i \mathsf E)}^{n \mathsf d_i}{\cosh(\ma_i \mathsf E)}^n \dbinom{\frac{i-1}{K} n}{\mathsf{eb_i} n} \dbinom{(1-\frac{i-1}{K}) n}{\mathsf{ef}_i n}\,.
\eea \eeq
Remark, in particular, that what lies behind the l.h.s. above (these are terms contributing to \eqref{ssss}) are thus  the first $m$-terms (up to irrelevant, for our purposes below) polynomial factors, of the product analysed in Theorem \ref{choosing_d}.\\

\noindent {\sf Second case.}
Again by the Evolution Lemma \ref{f} it holds that
\beq
1= \prod_{i=1}^{K-m} g_{i,K}(\mathsf d_i) \times \prod_{i=K-(m-1)}^{K} g_{i,K}(\mathsf d_i)\,,
\eeq
and therefore 
\beq\bea\label{eq3}
\prod_{i=K-(m-1)}^{K} g_{i,K}(\mathsf d_i)&=\left[ \prod_{i=1}^{K-m} g_{i,K}(\mathsf d_i)\right]^{-1}\\
&=\left[ 
{\left(\frac{\sinh(\overline{\ma}_{K-m} \mathsf E)}{\frac{K-m}{K}}\right)}^{\frac{K-m}{K}} {\left(\frac{\cosh(\overline{\ma}_{K-m} \mathsf E)}{1-\frac{K-m}{K}}\right)}^{1-\frac{K-m}{K}}\right]^{-1}\,,
\eea \eeq
the second equality in virtue of \eqref{f_f}.  In order to get a handle on the r.h.s. above we use the fundamental relation \eqref{fundamental} which states that
\beq \label{bas}
\sinh\left(\overline{\mathsf a}_{K-m} \mathsf E \right)\cosh\left(\underline{\mathsf a}_{K-m} \mathsf E \right) = \frac{K-m}{K}\,,
\eeq
implying, in particular, that 
\beq \label{cu}
\left[ {\left(\frac{\sinh(\overline{\ma}_{K-m} \mathsf E)}{\frac{K-m}{K}}\right)}^{\frac{K-m}{K}} \right]^{-1} = {\cosh(\underline{\ma}_{K-m} \mathsf E)}^{\frac{K-m}{K}}\,.
\eeq
Furthermore, the following "mirror" version of \eqref{bas} holds in virtue of the addition formula for hyperbolic functions (see \eqref{fundie} for the detailed derivation):
\beq
\cosh\left(\overline{\mathsf a}_{K-m} \mathsf E \right)\sinh\left(\underline{\mathsf a}_{K-m} \mathsf E \right) = 1 - \frac{K-m}{K}\,,
\eeq 
hence
\beq \label{cd}
\left[ {\left(\frac{\cosh(\overline{\ma}_{K-m} \mathsf E)}{1-\frac{K-m}{K}}\right)}^{1-\frac{K-m}{K}} \right]^{-1} = {\sinh(\underline{\ma}_{K-m} \mathsf E)}^{1-\frac{K-m}{K}}\,.
\eeq
Using \eqref{cu} and \eqref{cd} in \eqref{eq3} we thus have 
\beq \bea \label{bbbbb}
\prod_{i=K-(m-1)}^{K} g_{i,K}(\mathsf d_i) &= {\cosh(\underline{\ma}_{K-m} \mathsf E)}^{\frac{K-m}{K}}  {\sinh(\underline{\ma}_{K-m} \mathsf E)}^{1-\frac{K-m}{K}} \\
& = {\cosh(\overline{\ma}_{m} \mathsf E)}^{\frac{K-m}{K}} {\sinh(\overline{\ma}_{m} \mathsf E)}^{1-\frac{K-m}{K}}\,,
\eea \eeq
the second identity since $\sum_{i=1}^{K}\ma_i=1$ and by symmetry of the $\ma's$. Raising \eqref{bbbbb}  to the $n^{th}$- power, and by simple rearrangement, we thus see that 
\beq \label{antoine}
\left[ \prod_{i=K-(m-1)}^{K} g_{i,K}( \mathsf d_i)\right]^n = {\tanh(\overline{\ma}_{m} \mathsf E)}^{m\hat n_{K}} {\cosh(\overline{\ma}_{m} \mathsf E)}^{n}\,.
\eeq
Again by the definition of the $g$-functions, and taking into account the lower orders in the Stirling approximation of the binomial factors, one plainly checks that 
\beq\bea\label{eq2}
\left[ \prod_{i=K-(m-1)}^{K} g_{i,K}( \mathsf d_i)\right]^n & \propto n^{m-1} \prod_{i=K-(m-1))}^{K} {\tanh(\ma_i \mathsf E)}^{n \mathsf d_i}{\cosh(\ma_i \mathsf E)}^n \dbinom{\frac{i-1}{K} n}{\mathsf{eb_i} n} \dbinom{(1-\frac{i-1}{K}) n}{\mathsf{ef}_i n} 
\,,
\eea\eeq
and therefore, equating \eqref{antoine} and \eqref{eq2}, we also obtain the following asymptotic equivalence 
\beq\bea\label{eq4}
&{\tanh(\overline{\ma}_{m} \mathsf E)}^{m\hat n_{K}} {\cosh(\overline{\ma}_{m} \mathsf E)}^{n} \\
& \hspace{1.5cm} \propto n^{m-1} 
 \prod_{i=K-(m-1)}^{K} {\tanh(\ma_i \mathsf E)}^{nd_i}{\cosh(\ma_i \mathsf E)}^n \dbinom{\frac{i-1}{K} n}{\mathsf{eb}_i n} \dbinom{(1-\frac{i-1}{K}) n}{\mathsf{ef}_i n}\,.
\eea\eeq
In full analogy to \eqref{antoi}, we therefore see that behind the l.h.s. above (these are also terms contributing to \eqref{ssss}) hide in fact  the last $m$-terms of the product analysed in Theorem \ref{choosing_d}. \\

\noindent Thanks to both \eqref{antoi} and \eqref{eq4}, we may now replace the corresponding terms on the r.h.s. of \eqref{ssss}: this indeed allows to extend the product to all indeces $i=1,\dots, K$, and seamlessly leads to the lower bound
\beq\bea\label{eq5}
\E\left({\mathcal N}_{n, K, K'}^{\e}\right)&\gtrsim \prod_{i=1}^{K} {\tanh(\ma_i \mathsf E)}^{n \mathsf d_i}{\cosh(\ma_i \mathsf E)}^{n} \dbinom{\frac{i-1}{K} n}{\mathsf{eb_i} n} \dbinom{(1-\frac{i-1}{K}) n}{\mathsf{ef}_i n} \\
&\quad \times \frac{Q_nR_{n,K}}{P_n\cosh(\overline{\ma}_{m} \mathsf E)^{2n}}\prod_{i=m+1}^{K-m}\frac{1}{\cosh(\ma_i \mathsf E)} {(1+\e_{\mathsf E})}^{\sum_{i=m+1}^{K-m}n \mathsf d_i} {(1+\e_{m, \mathsf E})}^{2 m\hat n_{K}},
\eea\eeq
where $Q_n\defi n^{2(m-1)}$ is yet another polynomial term. 

The  full product in the first line of the r.h.s. of \eqref{eq5} is easily taken care of. In fact, by elementary
rearrangement, it holds that
\beq \bea \label{choosing_d1}
& \prod_{i=1}^{K} {\tanh(\ma_i \mathsf E)}^{n \mathsf d_i}{\cosh(\ma_i \mathsf E)}^{n} \dbinom{\frac{i-1}{K} n}{\mathsf{eb_i} n} \dbinom{(1-\frac{i-1}{K}) n}{\mathsf{ef}_i n} \\
& \hspace{2.5cm} = \prod_{i=1}^{K} {\sinh(\ma_i \mathsf E)}^{n \mathsf d_i}{\cosh(\ma_i \mathsf E)}^{n(1-\mathsf d_i)} \dbinom{\frac{i-1}{K} n}{\mathsf{eb}_i n} \dbinom{(1-\frac{i-1}{K}) n}{\mathsf{ef}_i n}\,,
\eea \eeq
and by Stirling approximation to  second order, the r.h.s. of \eqref{choosing_d1} equals 
\beq\bea \label{choosing_dd}
\prod_{i=1}^{K} \left\{ \frac{ \sinh(\ma_i \mathsf E)^{\mathsf d_i} {\cosh(\ma_i \mathsf E)}^{1-\mathsf d_i}               \varphi\left( \frac{i-1}{K}\right) \varphi\left( 1-\frac{i-1}{K}\right)}{
\varphi(\mathsf{eb}_i)\varphi\left(\frac{i-1}{K}-\mathsf{eb}_i\right) \varphi\left(\mathsf{ef}_i\right) \varphi\left(1-\frac{i-1}{K}-\mathsf{ef}_i\right)} 
\right\}^n \times S_{n,K},
\eea\eeq
where $S_{n,K}$ corresponds to the lower order (polynomial) terms in the approximation. But by Theorem \ref{choosing_d}, the first term of \eqref{choosing_dd}, i.e. the full product, equals unity, whereas an elementary inspection of the polynomial terms further shows that
\beq\bea\label{Sa}
S_{n,K}&\gtrsim \frac{1}{\sqrt{\frac{2\pi n}{K}(1-\frac{1}{K})}}\prod_{i=2}^{K-1} {\left(\frac{(2\pi n)^2(\frac{i-1}{K})(1-\frac{i-1}{K})}{(2\pi n)^4 \mathsf{eb}_i(\frac{i}{K}-\mathsf{ef}_i)\mathsf{ef}_i(1-\frac{i}{K}-\mathsf{eb}_i)}\right)}^{\frac{1}{2}}\\
&\gtrsim  \frac{1}{{n}^{K-2+\frac{1}{2}}}\,.
\eea \eeq
Using all this in \eqref{eq5} yields 
\beq\bea\label{eq55}
\E\left({\mathcal N}_{n, K, K'}^{\e}\right) \geq& \frac{Q_nR_{n,K}}{P_n\cosh(\overline{\ma}_{m}\mathsf E)^{2n}}\prod_{i=m+1}^{K-m}\frac{1}{\cosh(\ma_i \mathsf E)^{n}}{(1+\e_{\mathsf E})}^{\sum_{i=m+1}^{K-m}n \mathsf d_i} {(1+\e_{\overline{\ma}_m,\mathsf E})}^{2 m\hat n_{K} }\,,
\eea\eeq
where $P_n\defi P_n{n}^{K-2+\frac{1}{2}}$ is yet another polynomial term. \\

It thus remains to control the $\cosh$-terms in \eqref{eq55}. To see how this goes we observe that by Taylor expanding  the $\cosh$-function to second order, 
\beq\bea\label{tl1}
\cosh(\overline{\ma}_{m} \mathsf E)^{-1} & = \exp\left[ - \log \cosh(\overline{\ma}_{m} \mathsf E) \right] \\
& \geq \exp -\log\left\{ 1+\frac{(\overline{\ma}_{m} \mathsf E)^{2}\cosh(\overline{\ma}_{m} \mathsf E)}{2} \right\} \\ & \geq \exp\left( -\frac{(\overline{\ma}_{m} \mathsf E)^{2}}{\sqrt{2}} \right)\,,
\eea\eeq
the second inequality since $\log(1+x)\leq x$, and using that $\cosh(\overline{\ma}_{m} \mathsf E)\leq \cosh(\mathsf E)=\sqrt{2}$. Moreover, by \eqref{control_alpha} it holds that $\ma_i \mathsf E\leq\frac{1}{K}$: summing over $i=1\dots m$ thus leads to $\overline{\ma}_{m} \mathsf E \leq\frac{m}{K}$, which combined with \eqref{tl1} yields
\beq\bea\label{tl2}
\cosh(\overline{\ma}_{m} \mathsf E)^{-1} \geq \exp\left( -\frac{m^2}{\sqrt{2}K^2}\right)\,.
\eea\eeq
A similar reasoning evidently yields 
\beq\bea\label{tl3}
\cosh(\mathsf a_i E)^{-1}  \geq \exp\left( -\frac{1}{\sqrt{2}K^2}\right)\,,
\eea\eeq
for any $i=1\dots K$. By \eqref{tl2} and \eqref{tl3} we thus have that
\beq\bea \label{cosh_weg}
\frac{1}{\cosh(\overline{\ma}_{m} \mathsf E)^{2n}} \times \prod_{i=m+1}^{K-m}\frac{1}{\cosh(\ma_i \mathsf E)^{n}} & \geq  \exp\left\{-\frac{n}{\sqrt{2}K}-\frac{2nm(m-1)}{\sqrt{2}K^2}\right\}\,,
\eea\eeq
which we recognize as the $S_{n,K,m}$-term announced in \eqref{entrop_stret}: {\it the entropic cost for  stretching the paths}.  Using 
\eqref{cosh_weg} in \eqref{eq55} finally yields
\beq\bea
\E\left({\mathcal N}_{n, K, K'}^{\e}\right) \geq {(1+\e_E)}^{\sum_{i=m+1}^{K-m}n \mathsf d_i} {(1+\e_{\overline{\ma}_m,\mathsf E})}^{2 m\hat n_{K}} \frac{S_{n,K,m} R_{n,K}Q_n}{P_n}\,,
\eea\eeq
and Theorem \ref{choosing_dprime} is thus settled. 
\end{proof}

\section{The second moment, and proof of Theorem \ref{connecting}} \label{proof_thm_two}
The goal of this section is to provide a proof of Therem \ref{connecting}. We begin with a technical input, concerning  tail estimates for the probability of two correlated sums of exponentials. 
\begin{lem}[Overlap probability] \label{op}
Consider independent  standard exponentials  $\{\xi_i\}$, and let $X_l \defi \sum_{i=1}^l \xi_i$.
Denote by $X'_l$ the sum of $l$ such $\xi$-exponentials, and assume that $X'_l$ shares exactly k edges with $X_l$. Then for $x>0$, it holds:
\beq
\PP\left( X_l \leq x, X'_l \leq x \right)\propto \frac{x^{2l-k}}{(l-k)!l!}g\left(\frac{k}{l}\right)^l.
\eeq
where 
\beq \label{defi_ggg}
\gamma \in [0,1] \mapsto g(\gamma) \defi \frac{   {\left\{ 4(1-\gamma)\right\} }^{1-\gamma}    }{ {\left\{ 2-\gamma\right\}}^{2-\gamma} }\,.
\eeq
In particular, $\|g \|_\infty \leq 1$.
\end{lem}

\begin{proof} Without loss of generality we may write 
\beq
X_l' = \sum_{i=1}^k \xi_i + \sum_{i=k+1}^l \xi_i' \,,
\eeq
for independent $\xi'$'s, which are also independent of the $\xi$-family. Remark that the first sum, the common trunk, is a $\Gamma(k,1)$-distributed r.v., whereas the second sum is 
$\Gamma(l-k,1)$-distributed. By conditioning on the common trunk, and by independence, it thus holds: 
\beq\bea \label{blaaa}
\PP\left( X_l \leq x, X'_l \leq x \right)&= \int_{0}^{+\infty}\PP\left(t+X_{l-k}\leq x\right)^2 
\PP(X_k \in dt)\\
& = \int_{0}^{+\infty}\PP\left(X_{l-k}\leq x-t \right)^2\frac{t^{k-1}e^{-t}}{(k-1)!}dt\\
&\propto \frac{1}{{(l-k)!}^2 (k-1)!}\int_{0}^{x}{\left(x-t\right)}^{2(l-k)}t^{k-1} dt\,,
\eea\eeq
the last step by the standard tail-estimates. Integration by parts then yields
\beq\bea
\int_{0}^{x}{\left(x-t\right)}^{2(l-k)}t^{k-1}dt= \frac{(k-1)!(2(l-k))!}{(2l-k)!}x^{2l-k}\,,
\eea\eeq
and therefore
\beq\bea \label{est1}
\PP\left( X_l \leq x, X'_l \leq x \right)&\propto  \frac{(2(l-k))!}{(2l-k)!{(l-k)!}^2}x^{2l-k}\\
& \propto \frac{x^{2l-k}}{(l-k)!l!}\frac{l!(2(l-k))!}{(2l-k)!{(l-k)!}}.\\
&\propto \frac{x^{2l-k}}{(l-k)!l!}\frac{(1-\frac{k}{l})^{l-k}}{2^k(1-\frac{k}{2l})^{2l-k}}\,,
\eea\eeq
the last inequality by Stirling approximation. 

Remark that with $\gamma \defi k/l \in [0,1]$, the second factor in the last term above can be written as
\beq \bea \label{g-fct}
\frac{(1-\frac{k}{l})^{l-k}}{2^k(1-\frac{k}{2l})^{2l-k}}={\left\{\frac{{(4(1-\gamma))}^{(1-\gamma)}}{{(2-\gamma)}^{(2-\gamma)}}\right\}}^l\defi g(\gamma)^l\,,
\eea \eeq
and using this in \eqref{est1} yields
\beq\bea \label{estimate}
\PP\left( X_l \leq x, X'_l \leq x \right)&\propto \frac{x^{2l-k}}{(l-k)!l!}g\left(\frac{k}{l}\right)^l\,,
\eea\eeq
concluding the proof of the estimate for the overlap probability. 
\end{proof}

\noindent We now address the second moment of $\mathcal N_{n, K, K'}^\e$, as required for a proof of Theorem \ref{connecting}. For this, some notation is needed: recall from \eqref{intersectingg} that $\mathcal J$ is a deterministic subset of polymers with cardinality $\# \mathcal J = \mathsf J = \left\lfloor \E\# \mathcal E_{n, K, K'}^{1, \e}/ 2 \right\rfloor$. Given a path $\pi \in \mathcal J $, we shorten:
\beq \bea
\mathcal J_{\pi}(n, k)  \defi & \;  \text{\sf all paths}\; \pi' \in \mathcal J\; \\
& \text{\sf which share}\; k \; \text{\sf edges with} \; \pi,\\
& \text{\sf whithout considering the first and the last edge,} 
\eea \eeq
and for its cardinality
\beq
f_{\pi}(n,k) \defi \# \mathcal J_{\pi}(n, k)\,.
\eeq
Analogously we shorten
\beq \bea
\mathcal J_{\pi}^{(d)}(n,k)   \defi & \;  \text{\sf all paths}\; \pi' \in \mathcal J\; \text{\sf which share} \; k \; \text{\sf edges with}\; \\
& \; \pi \; \text{\sf only in the directed phase, i.e between} \; \\
&\boldsymbol 0 \; \text{\sf and} \; H_m \; \text{\sf or} \; H_{K-m}  \;  \text{\sf and} \boldsymbol  \;  1,\\
& \text{\sf but without considering  first and  last edge,} 
\eea \eeq
and let
\beq
f_{\pi}^{(d)}(n,k) \defi \#  \mathcal J_{\pi}^{(d)}(n,k)\,,
\eeq
denote its cardinality. 

And finally, 
\beq \bea
\mathcal J_{\pi}^{(s)}(n,k)  \defi & \;  \text{\sf number of paths}\; \pi' \in \mathcal J\; \text{\sf which share} \; k \; \text{\sf edges with}\; \\
& \; \pi \; \text{\sf with at least one common edge in the stretched}  \\
& \text{\sf  phase, i.e between} \; H_m \; \text{\sf and} \; H_{K-m},\\
& \text{\sf but without considering  first and  last edge,} 
\eea \eeq
analogously shortening for its cardinality
\beq
f_{\pi}^{(s)}(n,k) \defi \sharp \mathcal J_{\pi}^{(s)}(n,k)\,.
\eeq

Remark that
\beq\bea\label{22}
f_{\pi}(n,k)=f_{\pi}^{(d)}(n,k)+f_{\pi}^{(s)}(n,k)\,.
\eea\eeq

We will also need the "worst case scenarios"
\beq \bea
f(n,k) & \defi \sup_{\pi \in \mathcal J}f_{\pi}(n,k)\,, \\
f^{(d)}(n,k) &\defi \sup_{\pi \in \mathcal J}f_{\pi}^{(d)}(n,k)\,,\\
f^{(s)}(n,k) & \defi \sup_{\pi \in \mathcal J}f_{\pi}^{(s)}(n,k)\,.
\eea \eeq
in which case it holds, in particular, that
\beq\bea\label{22}
f(n,k)\leq f^{(d)}(n,k)+f^{(s)}(n,k).
\eea\eeq

\noindent For $i=m+1\dots K-m$, and two polymers $\pi, \pi' \in \mathcal J$, we shorten 
\beq 
\PP_i(\pi) \defi \PP\left( X_{i-1,i}(\pi) \leq {\mathsf a}_{i, \e}\right),
\eeq
and
\beq 
\PP_i(\pi, \pi') \defi \PP\left( X_{i-1,i}(\pi) \leq {\mathsf a}_{i, \e}, \;  X_{i-1,i}(\pi') \leq {\mathsf a}_{i, \e} \right)\,.
\eeq
Furthermore, we shorten
\beq\bea
\PP_m(\pi)& \defi \PP\left( X_{m}(\pi) \leq {\mathsf a}_{m, \e}\right),\\
\PP_{K-m+1}(\pi)& \defi \PP\left( X_{K-m+1}(\pi) \leq {\mathsf a}_{K-m+1, \e}\right),
\eea\eeq
and 
\beq\bea
 \PP_m(\pi, \pi') & \defi \PP\left( X_{m}(\pi),  X_{m}(\pi') \leq {\mathsf a}_{m, \e}\right) \\
\PP_{K-m+1}(\pi, \pi') & \defi \PP\left( X_{K-m+1}(\pi), X_{K-m+1}(\pi') \leq {\mathsf a}_{K-m+1, \e}\right)\,,
\eea\eeq
as well as 
\beq \bea
\PP(\pi) & \defi \PP\Big(X_m(\pi) \leq {\mathsf a}_{m, \e}, \; X_{i-1,i}(\pi) \leq {\mathsf a}_{i, \e}\;
 i=m+1 \dots K-m,\; X_{K-m+1}(\pi) \leq {\mathsf a}_{K-m+1, \e} \Big),  
\eea \eeq
and
\beq\bea
\PP(\pi, \pi') \defi & \PP(X_m(\pi), \; X_m(\pi') \leq {\mathsf a}_{m, \e}, \; X_{i-1,i}(\pi),X_{i-1,i}(\pi') \leq {\mathsf a}_{i, \e}\; \text{for} \\
&\qquad \; i=m+1 \dots K-m, \; X_{K-m+1}(\pi), X_{K-m+1}(\pi') \leq {\mathsf a}_{K-m+1, \e})\,.
\eea\eeq
Remark that for loopless paths the substrand-energies are independent, hence, and with the above notation, 
\beq\label{first_out'}
\PP(\pi) = \prod_{i=m}^{K-m+1} \PP_i(\pi), \qquad \PP(\pi, \pi') = \prod_{i=m}^{K-m+1} \PP_i(\pi, \pi')\,,
\eeq
In particular, it holds that 
\beq\bea\label{first_out}
\E\left({\mathcal N}_{n, K, K'}^{\e}\right)={\mathsf J} \PP(\pi) = {\mathsf J} \prod_{i=m}^{K-m+1} \PP_i(\pi) \,.
\eea\eeq

\noindent Concerning the second moment, we write
\beq \bea \label{comincia}
\E\left({{\mathcal N}_{n, K, K'}^{\e}}^2\right) & = \sum_{\pi, \pi' \in \mathcal J } \PP(\pi, \pi') \\
& = \sum_{\pi \in \mathcal J}\sum_{k=0}^{\mathsf L_{opt}n-2} \sum_{\pi' \in \mathcal J_\pi(n,k)}   \PP(\pi, \pi')\,,
\eea \eeq
by arranging the sum according to the possible overlap-regimes. \\

\noindent The case $k=0$ is both crucial and easily taken care of by the following observations: first remark that
the distribution of the energies of a pair of polymers depends solely on the number of common edges; furthermore
the number of pairs of polymers with zero common edges is at most $\mathsf J^2$. Therefore, for any $(
\hat \pi, \tilde \pi) \in (\mathcal J, \mathcal J_{\hat \pi}(n,0))$ it holds: 
\beq \bea \label{k=0}
\sum_{\pi \in \mathcal J} \sum_{\pi' \in \mathcal J_\pi(n,0)}  \PP(\pi, \pi') & \leq {\mathsf J}^2 \PP(\hat \pi, \tilde \pi) =   {\mathsf J}^2 \PP(\hat \pi)^2\,,
\eea \eeq
the last equality holding true since in case of non-overlapping paths, the $\hat \pi, \tilde \pi$-energies are independent and identically distributed.  Using \eqref{first_out} in \eqref{k=0} therefore yields 
\beq \bea \label{yuhe}
\sum_{\pi \in \mathcal J} \sum_{\pi' \in \mathcal J_\pi(n,0)}  \PP( \pi,  \pi') \leq  \E\left({{\mathcal N}_{n, K, K'}^{\e}}\right)^2\,,
\eea \eeq
This settles the $k=0$ regime. \\

\begin{rem}\label{facteur_n} 
Recovering the first moment squared as in \eqref{yuhe} is absolutely crucial for the whole approach, and the main reason for treating first and last edge on different footing. Without such different treatment, one would get the first moment squared  \emph{up to a constant only}, and this would nullify the proof of Theorem \ref{geo_optimal_paths}. This feature is common to virtually all models in the REM-class, see  \cite{kistler} for more on this delicate issue.
\end{rem}

\noindent As for the remaining overlap-regimes, we will distinguish between 
\begin{itemize}
\item $1 \leq k\leq 200 \hat n_{K} $: this corresponds to the case of weak correlations (the overlap between the two polymers  is small); 
\item $k>200  \hat n_{K}$: this corresponds to the case of strong correlations (the two polymers 
strongly overlap).
\end{itemize}
We now rearrange the second moment according to the above dichotomy. Henceforth,
given $\pi \in \mathcal J$, and with $k\in\N$, we denote by $\pi_k^{(d)} \in \mathcal J_\pi^{(d)}(n,k)$ a polymer which shares $k$ edges with $\pi$, and in full analogy for $\pi_k^{(s)} \in \mathcal J_\pi^{(s)}(n,k)$ and $\pi_k \in \mathcal J_\pi(n,k)$.  With this notation, again using that specifying the number of common edges fixes the distribution of the pair of paths, and by \eqref{yuhe}, we thus have
\beq\bea\label{21}
\E\left({{\mathcal N}_{n, K, K'}^{\e}}^2\right) & \leq \E\left({{\mathcal N}_{n, K, K'}^{\e}}\right)^2+ \\
& \hspace{1.5cm} + \mathsf J \sum_{k=1}^{ 200  \hat n_{K}}f^{(d)}(n,k) \PP\left( \pi, \pi_k^{(d)}\right) \\
& \hspace{3cm} + \mathsf J \sum_{k=1}^{ 200  \hat n_{K}}f^{(s)}(n,k) \PP\left(\pi, \pi_k^{(s)}\right) \\
& \hspace{5cm} + \mathsf J \sum_{k= 200  \hat n_{K} +1}^{\mathsf L_{opt}n-2}f(n,k) \PP\left( \pi, \pi_k \right)\,.
\eea \eeq
On the other hand, by Jensen inequality it holds
\beq
\E\left({{\mathcal N}_{n, K, K'}^{\e}}^2\right) \geq \E\left({{\mathcal N}_{n, K, K'}^{\e}} \right)^2. 
\eeq
In order to establish Theorem \ref{connecting} it therefore suffices to show that the last three sums on the r.h.s. of \eqref{21} are of lower order when compared with the first moment squared. This is indeed our key claim: since its proof is long and technical, we formulate it in the form of three Propositions. 

\begin{prop}\label{sum1}
For any $K>m\e^{-2}$, it holds 
\beq\bea
\mathsf J \sum_{k=1}^{ 200  \hat n_{K}}f^{(d)}(n,k) \PP\left( \pi, \pi_k^{(d)}\right) = o\left( \E\left({{\mathcal N}_{n, K, K'}^{\e}}\right)^2 \right)\,,
\eea\eeq
for $n\to \infty$.
\end{prop}
\begin{prop}\label{sum3}
For any $K>\max(2\times 10^7,m\e^{-2})$ and $K'>2\log(2)\mathsf LK^2$, it holds
\beq\bea
\mathsf J \sum_{k= 200  \hat n_{K} +1}^{\mathsf L_{opt}n-2}f(n,k) \PP\left( \pi, \pi_k \right)=o\left( \E\left({{\mathcal N}_{n, K, K'}^{\e}}\right)^2 \right)\,,
\eea\eeq
for $n\to \infty$.
\end{prop}
\begin{prop}\label{sum2}
For any $K>2\times 10^7$ and $K'>2\log(2)\mathsf LK^2$, it holds
\beq\bea
\mathsf J \sum_{k=1}^{ 200  \hat n_{K}}f^{(s)}(n,k) \PP\left(\pi, \pi_k^{(s)}\right)=o\left( \E\left({{\mathcal N}_{n, K, K'}^{\e}}\right)^2 \right) ,
\eea\eeq
for $n\to \infty$.
\end{prop}

The following three sections are devoted to the proofs of the above statements. We anticipate that each proposition/treatment will require a good control of the asymptotics of the $f^{(d)}, f$- and $f^{(s)}$-terms: these will be formulated in the form of Lemmata whose proofs, relying on extremely technical combinatorial estimates, are however postponed to Section \ref{combinatorial_estimates}. 

The reason for tackling the $f$-regime before the $f^{(s)}$-one is that the treatment of the latter will require some technical inputs which are obtained in the analysis of the the former. 

\subsection{Proof of Proposition \ref{sum1}}
The goal is  to prove that
\beq \label{sum1_claim}
\lim_{n\to \infty }\frac{\mathsf J \sum_{k=1}^{ 200  \hat n_{K}}f^{(d)}(n,k) \PP\left( \pi, \pi_k^{(d)}\right)}{\E\left({{\mathcal N}_{n, K, K'}^{\e}}\right)^2 } = 0. 
\eeq
The combinatorial input here is  the following

\begin{lem} \label{comb1}
For all $k\leq 200\hat n_{K}$, one has
\beq\bea
&f^{(d)}(n,k)\leq \frac{\mathsf J (m\hat n_{K}-\lfloor \frac{k}{2} \rfloor)!(n-1-\lceil \frac{k}{2} \rceil)!}{{(m \hat n_{K})!}n!}l(k)\,,
\eea\eeq
where 
\beq \label{defi_l}
l(k) \defi \begin{cases}
32(k+1)^3 & k \leq n^{1/4} \\
16n^{13}(k+1) & \text{otherwise}. 
\end{cases}
\eeq
\end{lem}
\noindent The proof of this Lemma is postponed to Section \ref{combinatorial_estimates}.
Coming back to the task of proving \eqref{sum1_claim}, by \eqref{first_out'} and \eqref{first_out} we write
\beq\bea\label{23}
\frac{\mathsf J \sum_{k=1}^{ 200  \hat n_{K}}f^{(d)}(n,k) \PP\left( \pi, \pi_k^{(d)}\right)}{\E\left({{\mathcal N}_{n, K, K'}^{\e}}\right)^2 }&=\frac{\mathsf J \sum_{k=1}^{ 200  \hat n_{K}}f^{(d)}(n,k) \prod_{i=m}^{K-m+1}\PP_i\left( \pi, \pi_k^{(d)}\right)}{{\mathsf J}^2 \prod_{i=m}^{K-m+1} \PP_i(\pi)^2}\,,
\eea\eeq
In the considered regime, polymers share no edges but in the directed phase: the probabilities indexed by $i \in \{m+1, \dots, K-m\}$ therefore factor out in virtue of the ensuing independence, and the r.h.s. of \eqref{23} then takes the neater form 
\beq\bea\label{23'}
\sum_{k=1}^{ 200  \hat n_{K}}\frac{f^{(d)}(n,k)\PP_m\left( \pi, \pi_k^{(d)}\right)\PP_{K-m+1}\left( \pi, \pi_k^{(d)}\right) }{\mathsf J  \PP_m\left(\pi\right)^2 \PP_{K-m+1}\left(\pi\right)^2}\,.
\eea\eeq
Now, for $\pi \in \mathcal J$ and $\pi_k^{(d)} \in \mathcal J_\pi^{(d)}(n,k)$, let us denote by $k_l$ the number of common edges between $\boldsymbol 0$ and $H_m$, and by $k_r$ the number of common edges between $H_{K-m}$ and $\boldsymbol 1$ (in which case it evidently holds that $k=k_l+k_r$). By the estimates for the overlap probabilities from Lemma  \ref{op} (using the rough bound $\|g\|_\infty \leq 1$), it steadily follows that
\beq\bea\label{23''}
\PP_m\left( \pi, \pi_k^{(d)}\right)\PP_{K-m+1}\left( \pi, \pi_k^{(d)}\right)\lesssim \frac{{\ma_{m,\e}}^{4(m\hat n_{K}-1)-k}}{(m\hat n_{K}-1-k_l)!(m\hat n_{K}-1-k_r)!(m \hat n_{K}-1)!^2}.
\eea\eeq
We now proceed by worst case scenario and {\it maximize} the r.h.s. over all possible $(k_l, k_r)$-choices. This can be seamlessly identified thanks to the well-known log-convexity of factorials, which we recall is the property that for any $a \geq b \geq j\geq 0$ it holds
\beq \label{log_convexity}
(a+j)!(b-j)! \geq a! b! \,.
\eeq
Using \eqref{log_convexity} with
\beq
a \defi m\hat n_{K}-1-\lfloor \frac{k}{2} \rfloor, \qquad \text{and} \qquad b\defi m\hat n_{K}-1-\lceil {\frac{k}{2}} \rceil,
\eeq 
we see that the worst case on the r.h.s. of  \eqref{23''} is attained in $k_r \in \{\lfloor \frac{k}{2} \rfloor, \lceil \frac{k}{2} \rceil\}$, which is equivalent to $k_l \in \{\lfloor \frac{k}{2} \rfloor, \lceil \frac{k}{2} \rceil\}$ because $k=k_l+k_r$ , hence
\beq\bea\label{23'''}
\PP_m\left( \pi, \pi_k^{(d)}\right)\PP_{K-m+1}\left( \pi, \pi_k^{(d)}\right)  \leq \frac{{\ma_{m,\e}}^{4(m\hat n_{K}-1)-k}}{(m\hat n_{K}-1-\lfloor \frac{k}{2}\rfloor)!(m\hat n_{K}-1-\lceil \frac{k}{2}\rceil)!(m \hat n_{K}-1)!^2}.
\eea\eeq
Using the latter in \eqref{23'}, and by the usual tail estimates, we obtain
\beq\bea\label{24}
\frac{\mathsf J \sum_{k=1}^{ 200  \hat n_{K}}f^{(d)}(n,k) \PP\left( \pi, \pi_k^{(d)}\right)}{\E\left({{\mathcal N}_{n, K, K'}^{\e}}\right)^2 }
&\lesssim \sum_{k=1}^{ 200  \hat n_{K} } \frac{f^{(d)}(n,k)( m  \hat n_{K}-1)!^2}{\mathsf J(m\hat n_{K}-1-\lfloor \frac{k}{2}\rfloor)!(m\hat n_{K}-1-\lceil \frac{k}{2}\rceil)! (\ma_{m,\e})^{k}}.\\
\eea\eeq
To get a handle on the factorials in the r.h.s. above we employ the bound
\beq \label{ineq_fac}
 \frac{( m  \hat n_{K}-1)!^2}{(m\hat n_{K}-1-\lfloor \frac{k}{2}\rfloor)!(m\hat n_{K}-1-\lceil \frac{k}{2}\rceil)!}\leq \frac{( m  \hat n_{K})!^2}{(m\hat n_{K}-\lfloor \frac{k}{2}\rfloor)!(m\hat n_{K}-\lceil \frac{k}{2}\rceil)!}\,,
\eeq 
which can be plainly checked by writing out, and simplifying.  Using \eqref{ineq_fac}, and the combinatorial estimates of Lemma \ref{comb1} for the $f^{(d)}$-term, yields
\beq\bea\label{25}
(\ref{24})&\lesssim \sum_{k=1}^{ 200  \hat n_{K} } \frac{(m\hat n_{K}-\lfloor \frac{k}{2} \rfloor)!(n-1-\lceil \frac{k}{2} \rceil)!l(k)(m  \hat n_{K})!^2}{( m  \hat n_{K} )!n!(m\hat n_{K}-\lfloor \frac{k}{2}\rfloor)!(m\hat n_{K}-\lceil \frac{k}{2}\rceil)!(\ma_{m,\e})^{k}}\\&= \sum_{k=1}^{ 200  \hat n_{K} }\frac{(n-1-\lceil \frac{k}{2} \rceil)!l(k)( m  \hat n_{K} )!}{n!(m\hat n_{K}-\lceil \frac{k}{2}\rceil)!(\ma_{m,\e})^{k}}\,.
\eea \eeq
the second step in virtue of elementary, term by term, simplifications. 

Using  $(a-1)! = a!/a$ for the first factorial-term in the numerator on the r.h.s. above yields 
\beq \bea\label{25.5}
\eqref{25} &= \sum_{k=1}^{ 200  \hat n_{K} }\frac{(n-\lceil \frac{k}{2}\rceil)!l(k)( m  \hat n_{K} )!}{(n-\lceil \frac{k}{2}\rceil)n!( m  \hat n_{K}-\lceil \frac{k}{2}\rceil)!(\ma_{m,\e})^{k}}\\
&\lesssim \sum_{k=1}^{200\hat n_{K} }\frac{l(k)}{(n-\lceil \frac{k}{2}\rceil)} \cdot \frac{(1-\frac{1}{n}\lceil \frac{k}{2}\rceil)^{n-\lceil \frac{k}{2}\rceil}( \frac{m}{K})^{ m  \hat n_{K} }}{{(\frac{m}{K}-\frac{1}{n}\lceil \frac{k}{2}\rceil)}^{(m\hat n_{K}-\lceil \frac{k}{2}\rceil)}}\cdot \frac{1}{(\ma_{m,\e})^{k}}
\eea\eeq
the last inequality by Stirling's approximation. 

We  now focus on the middle term on the r.h.s. above. Omitting the rounding operation, and shortening
\beq
Q(x) \defi (1-x)\log(1-x)-\frac{m}{K}(1-x\frac{K}{m})\log\left(1-x\frac{K}{m} \right),
\eeq
we may rewrite this middle term as 
\beq\bea\label{26}
& \frac{(1-\frac{k}{2n})^{n-\frac{k}{2}}(\frac{m}{K})^{m \hat n_K}}{{(\frac{m}{K}-\frac{k}{2n})}^{( m \hat n_K-\frac{k}{2})}} = \left(\sqrt{\frac{m}{K}}\right)^{k} \exp n Q \left(\frac{k}{n} \right) \,.
\eea\eeq
It is plainly checked that, for $k/n \in [0,1]$, the $Q$-function is in fact {\it negative} (for $K>m$), hence 
\beq\bea\label{26'}
\eqref{26}&\leq \left(\sqrt{\frac{m}{K}}\right)^{k}.
\eea\eeq
By definition, 
\beq
(\ma_{m,\e})^k = \left( \overline{\ma}_{m}(\mathsf E+\epsilon)+\epsilon \right)^k \geq \e^k\,,
\eeq
the inequality by elementary minorization: this, as well as the bound  \eqref{26'}, imply that \eqref{25.5} is {\it at most}
\beq\bea\label{27}
\sum_{k=1}^{ 200 \hat n_{K} }\frac{l(k)}{(n-\frac{k}{2})} \frac{1}{\left(\sqrt{\frac{K}{m}}\epsilon\right)^{k}}
&= \left( \sum_{k=1}^{n^{\frac{1}{4}}}+\sum_{k=n^{\frac{1}{4}}+1}^{ 200 \hat n_{K}} \right) \frac{l(k)}{(n-\frac{k}{2})} \frac{1}{\left(\sqrt{\frac{K}{m}}\epsilon \right)^{k}}.\\
\eea\eeq
If we now take $K$ large enough such that $\sqrt{\frac{K}{m}}\epsilon>1$, to wit:
\beq \label{27}
K>m\e^{-2},
\eeq
and recalling the definition of $l(k)$ as in \eqref{defi_l}, we obtain
\beq\bea\label{28}
(\ref{27})&\lesssim \frac{1}{(n-n^{\frac{1}{4}})}\sum_{k=1}^{n^{\frac{1}{4}}}\frac{(k+1)^3}{(\sqrt{\frac{K}{m}}\epsilon)^{k}}+\sum_{k=n^{\frac{1}{4}}+1}^{ 200 \hat n_{K}  }n^{13} \frac{(n+1)}{(\sqrt{\frac{K}{m}}\epsilon)^{k}}\,.
\eea\eeq
The first sum on the r.h.s is, in the large-$n$ limit, obviously convergent: its contribution therefore vanishes in virtue of the $(n-n^{1/4})$-normalization. The second sum converges exponentially fast to $0$. All in all, the r.h.s. of \eqref{28} tends to $0$ as $n\to \infty$: this settles the proof of claim \eqref{sum1_claim}, and therefore of Proposition \ref{sum1}.\\
${} \hfill \square$

\subsection{Proof of Proposition \ref{sum3}}
We will need here two technical inputs. The first one is similar in nature to Lemma \ref{op}, and provides tail-estimates for the energies of overlapping polymers. As the proof is short and elementary, it will be given right away. 

\begin{lem}\label{prob}
Consider independent  standard exponentials  $\{\xi_i\}$, and let $X_l \defi \sum_{i=1}^l \xi_i$.
Denote by $X'_l$ the sum of $l$ such $\xi$-exponentials, and assume that $X'_l$ shares exactly k edges with $X_l$. 
Then, for $a, b>0$, it holds:
\beq\bea\label{30}
&\PP\left( X_l \leq a+b,\; X_l'\leq a+b\right)\lesssim \PP\left( X_l \leq a,\;  X_l' \leq a\right)\left(1+\frac{b}{a}\right)^{2l-k}\,.
\eea\eeq
\end{lem}
\begin{proof} 
Recalling that 
\beq
g(\gamma) \defi \frac{{\left\{ 4(1-\gamma)\right\}}^{1-\gamma}}{{\left\{2-\gamma\right\}}^{2-\gamma}},
\eeq 
by Lemma \ref{op}, it holds
\beq\bea\label{31}
&\PP\left( X_l \leq a+b,\; X_l'\leq a+b\right) \lesssim  \frac{\left(a+b\right)^{2l-k}}{(l-k)!l!}g\left(\frac{k}{l}\right)^{l}\,.
\eea\eeq
Using that $\left(a+b\right)^{2l-k} =a^{2l-k} \left(1+\frac{b}{a} \right)^{2l-k}$, we rephrase the r.h.s. of \eqref{31}, to wit
\beq\bea \label{pluplu}
\PP\left( X_l \leq a+b,\; X_l'\leq a+b\right) &\lesssim \frac{a^{2l-k}}{(l-k)!l!}g\left(\frac{k}{l}\right)^{l}{\left(1+\frac{b}{a}\right)}^{2l-k}.
\eea \eeq
Again by Lemma \ref{op}, for the first two terms on the r.h.s. above we have that
\beq \bea
\frac{a^{2l-k}}{(l-k)!l!}g\left(\frac{k}{l}\right)^{l} \lesssim \PP\left( X_l\leq a,\; X_l' \leq a\right)\,,
\eea\eeq
and plugging this in \eqref{pluplu} yields the claim of the Lemma. 
\end{proof}
The second technical input concerns the asymptotic of the $f$-terms. Here and below, we will denote by $P_n, Q_n$ finite degree polynomials, not necessarily the same at different occurences, and which depend on the hypercube dimension only. 

\begin{lem}\label{r2}
For all $k\leq \mathsf L_{opt}n$, it holds
\beq\bea
&f(n,k) \leq {\tanh\left( \mathsf E\left\{1-\frac{k}{\mathsf L_{opt}n}\right\}\right)}^{\max \left(n-k,\frac{\mathsf L_{opt}n-k}{4} \right)}\\
& \hspace{3cm} {\cosh\left( \mathsf E\left\{1-\frac{k}{\mathsf L_{opt}n}\right\}\right)}^{n}{\left(\frac{\mathsf L_{opt}n}{e \mathsf E}\right)}^{\mathsf L_{opt}n-k} n^{Kn^{\alpha}}P_n,
\eea\eeq
where $P_n$ is polynomial with finite degree and $\alpha \defi \frac{5}{6}$.
\end{lem}
The proof of this Lemma is also postponed to Section \ref{combinatorial_estimates}: here we will  use it for the

\begin{proof}[Proof of Proposition \ref{sum3}] By  \eqref{first_out'}, it holds that
\beq\bea\label{129}
\frac{\mathsf J \sum_{k=200 \hat n_{K}+1}^{ \mathsf L_{opt}n-2}f(n,k) \PP\left(\pi, \pi_k\right)}{\E\left({\mathcal N}_{n, K, K'}^{\e}\right)^2}=\frac{\mathsf J \sum_{k=200 \hat n_{K}+1}^{ \mathsf L_{opt}n-2}f(n,k) \prod_{i=m}^{K-m+1}\PP_i\left(\pi, \pi_k\right)}{\E\left({\mathcal N}_{n, K, K'}^{\e}\right)^2}\,.
\eea\eeq
We claim that the r.h.s. of \eqref{129} converges to $0$ as $n \to \infty$. To see this, some notation is needed: given two paths $\pi,\pi' \in \mathcal J$ which share $k$ edges, we denote by 
\begin{itemize}
\item $k_l$ the number of common edges between $\boldsymbol 0$ and $H_m$, 
\item $k_m$ the number of common edges between $H_m$ and $H_{K-m}$, 
\item $k_r$ the number of shared edges between $H_{K-m}$ and $\boldsymbol 1$. 
\end{itemize}
It clearly holds that $k=k_l+k_m+k_r$.  Using Lemma \ref{prob}, we obtain
\beq\bea\label{132}
\prod_{i=m}^{K-m+1}\PP_i\left(\pi, \pi_k\right) & \lesssim \PP\left( X_m(\pi), X_m(\pi_k)\leq \overline{\ma}_m \mathsf E \right)\times \\
& \hspace{2cm} \times  \prod_{i=m+1}^{K-m} \PP\left( X_{i-1,i}(\pi), X_{i-1,i}(\pi_k)\leq \ma_i \mathsf E \right) \times \\
& \hspace{1cm} \times \PP\left( X_{K-m+1}(\pi), X_{K-m+1}(\pi_k)\leq \overline{\ma}_m \mathsf E \right) \times \\ 
& \hspace{2cm} \times (1+\e_{\mathsf E})^{2\sum_{i=m+1}^{K-m}\mathsf nd_i-k_m}(1+\e_{m, \mathsf E} )^{4m\hat n_{K}-2-k_l-k_r}.
\eea\eeq
By definition of $\e_{m, \mathsf E}$ and $\e_{\mathsf E}$, see \eqref{defi_eee}, the following lower bound plainly holds  
\beq\bea\label{133}
1+\e_{m, \mathsf E} &\geq 1+\e_{\mathsf E}\,.
\eea\eeq
Using the independence of sub-energies we rewrite 
\beq\bea \label{133min}
&\PP\left( X_m(\pi), X_m(\pi_k)\leq \overline{\ma}_m \mathsf E \right)\times \\
& \hspace{2cm} \times  \prod_{i=m+1}^{K-m} \PP\left( X_{i-1,i}(\pi), X_{i-1,i}(\pi_k)\leq \ma_i \mathsf E \right) \times \\
& \hspace{1cm} \times \PP\left( X_{K-m+1}(\pi), X_{K-m+1}(\pi_k)\leq \overline{\ma}_m \mathsf E \right) \\
& = \PP\Bigg( X_m(\pi), X_m(\pi_k)\leq \overline{\ma}_m \mathsf E, \\
& \hspace{3cm} X_{i-1,i}(\pi), X_{i-1,i}(\pi_k)\leq \ma_i \mathsf E, \; i=m+1\dots K-m, \\
& \hspace{2cm}  X_{K-m+1}(\pi), X_{K-m+1}(\pi_k)\leq \overline{\ma}_m \mathsf E \Bigg)\,.
\eea\eeq
Since $\sum_{i=1}^K \mathsf a_i =1$, and by monotonicity of the probabilities, the r.h.s. of \eqref{133min} is {\it at most} 
\beq \label{133bis}
\PP\left( \overline{X}_{m}^{K-m+1}(\pi), \overline{X}_{m}^{K-m+1}\left( \pi_k\right) \leq \mathsf E \right).
\eeq
Using \eqref{133} and \eqref{133bis} in \eqref{132} thus yields
\beq\bea \label{133'}
\prod_{i=m}^{K-m+1}\PP_i\left(\pi, \pi_k\right) & \leq \PP\left( \overline{X}_{m}^{K-m+1}(\pi), \overline{X}_{m}^{K-m+1}\left( \pi_k\right) \leq \mathsf E \right) \times \\
& \hspace{1cm} \times \frac{(1+\e_E)^{2\sum_{i=m+1}^{K-m} \mathsf nd_i}(1+\e_{\overline{\ma}_m,E} )^{4 m \hat n_{K}-2}}{(1+\e_{\mathsf E})^k}\,,
\eea\eeq
which no longer depends on $k_l,k_r, k_m$, but only on their total sum. Using Lemma \ref{op} in \eqref{133'} we thus obtain 
\beq\bea\label{134} 
\prod_{i=m}^{K-m+1}\PP_i\left(\pi, \pi_k\right) \lesssim&\frac{\mathsf E^{2 \mathsf L_{opt}n-2-k}g(\frac{k}{\mathsf L_{opt}n-2})^{\mathsf L_{opt}n-2}}{(\mathsf L_{opt}n-2)!(\mathsf L_{opt}n-2-k)!}\frac{(1+\e_{\mathsf E})^{2\sum_{i=m+1}^{K-m} \mathsf nd_i}(1+\e_{m, \mathsf E} )^{4m \hat n_{K}-2}}{(1+\e_{\mathsf E})^k}.
\eea\eeq

We now come back to \eqref{129}: using the lower bound to the first moment of ${\mathcal N}_{n, K, K'}^{\e}$ established in Theorem \ref{choosing_dprime} for the denominator, and \eqref{134} for the numerator,  we see that
\beq\bea\label{135'}
&\eqref{129}\leq \frac{P_n^2}{Q_n^2}\mathsf J \sum_{k=200 \hat n_{K}+1}^{ \mathsf L_{opt}n-2} \frac{f(n,k) {\mathsf E}^{2\mathsf L_{opt}n-2-k}{g(\frac{k}{\mathsf L_{opt}n-2})^{\mathsf L_{opt}n-2}}
}{{(1+\e_{E})}^{k} (\mathsf L_{opt}n-2)!(\mathsf L_{opt}n-2-k)! C_{n,K,m}^2 }.\\
\eea\eeq
(Recall the convention that $P_n$ stands for some finite degree polynomial, not necessarily the same at different occurences).
It is immediate to check that the following inequality holds 
\beq \label{1gleft}
g\left(\frac{k}{ \mathsf L_{opt}n-2} \right)^{\mathsf L_{opt}n-2}<g\left(\frac{k}{ \mathsf L_{opt}n}\right)^{\mathsf L_{opt}n} P_n\,.
\eeq 
 Furthermore, 
\beq \label{1ohoh}
(\mathsf L_{opt}n-2)! = \frac{(\mathsf L_{opt}n)!}{(\mathsf L_{opt}n) (\mathsf L_{opt}n-1) } = \frac{(\mathsf L_{opt}n)!}{P_n},
\eeq
where $P_n$ is a polynomial of finite (quadratic) degree, and analogously 
\beq \label{1ohohoh}
(\mathsf L_{opt}n-2-k)!  = \frac{(\mathsf L_{opt}n- k)!}{P_n}.
\eeq
Using \eqref{1gleft}, \eqref{1ohoh}, and \eqref{1ohohoh}, we thus see that 
\beq\bea\label{135}
&\eqref{135'} \leq \frac{P_n}{Q_n}\mathsf J \sum_{k=200 \hat n_{K}+1}^{ \mathsf L_{opt}n-2} \frac{f(n,k) {\mathsf E}^{2\mathsf L_{opt}n-k}{g(\frac{k}{\mathsf L_{opt}n})^{\mathsf L_{opt}n}}
}{{(1+\e_{E})}^{k} (\mathsf L_{opt}n)!(\mathsf L_{opt}n-k)! C_{n,K,m}^2 }\,,
\eea\eeq
for some (modified, but still finite degree) polynomials $P_n,Q_n$. 

The inclusion $\mathcal J \subset \Pi_{n,\mathsf L_{opt}n}$ holds by construction, hence
\beq\bea\label{136}
\mathsf J  \leq M_{n,\mathsf L_{opt}n}\leq {\sinh(\mathsf E)}^{n}\frac{(\mathsf L_{opt}n)!}{ \mathsf E^{\mathsf L_{opt}n}}=\frac{(\mathsf L_{opt}n)!}{ \mathsf E^{\mathsf L_{opt}n}}\,,
\eea\eeq
the second inequality by Stanley's M-bound \eqref{sf}  with $x := \mathsf E$, and the last step since $\mathsf E$ satisfies $\sinh(\mathsf E)=1$. Plugging \eqref{136} into \eqref{135}, we obtain 
\beq\bea\label{137'}
(\ref{135})&\leq \frac{P_n}{C_{n,K,m}^2 Q_n} \sum_{k=200 \hat n_{K}+1}^{ \mathsf L_{opt}n-2}\frac{f(n,k) {\mathsf E}^{\mathsf L_{opt}n-k}{g(\frac{k}{\mathsf L_{opt}n})^{\mathsf L_{opt}n}}}{{(1+\e_{E})}^{k}(\mathsf L_{opt}n-k)!}\\
& \leq \frac{P_n}{C_{n,K,m}^2Q_n} \sum_{k=200 \hat n_{K}+1}^{ \mathsf L_{opt}n} \frac{f(n,k) {(e \mathsf E)}^{\mathsf L_{opt}n-k}{g(\frac{k}{\mathsf L_{opt}n})^{\mathsf L_{opt}n}}}{{(1+\e_{E})}^{k}(\mathsf L_{opt}n-k)^{\mathsf L_{opt}n-k}},
\eea\eeq
the last inequality by Stirling's approximation, and extending the sum up to $\mathsf L_{opt}n$ (the terms are positive anyhow). The estimates of Lemma \ref{r2} applied to \eqref{137'} yield
\beq\bea\label{137}
&\eqref{137'}\leq \frac{n^{Kn^{\alpha}}P_n}{C_{n,K,m}^2 Q_n} \sum_{k=200 \hat n_{K}+1}^{ \mathsf L_{opt}n}\Bigg[ {\tanh\left( \mathsf E\left\{1-\frac{k}{\mathsf L_{opt}n}\right\}\right)}^{\max \left(n-k,\frac{\mathsf L_{opt}n-k}{4} \right)}\times \\
& \hspace{7cm} \times \frac{{\cosh\left( \mathsf E\left\{1-\frac{k}{\mathsf L_{opt}n}\right\}\right)}^{n}g\left(\frac{k}{\mathsf L_{opt}n}\right)^{\mathsf L_{opt}n}}{{(1+\e_{\mathsf E})}^{k}{\left(1-\frac{k}{\mathsf L_{opt}n}\right)}^{\mathsf L_{opt}n-k}} \Bigg].
\eea\eeq
Recalling the definition \eqref{defi_ggg} of the $g$-function, one plainly checks that 
\beq \bea \label{1simpler_g}
\frac{g\left(\frac{k}{\mathsf L_{opt}n}\right)^{\mathsf L_{opt}n}}{
\left(1-\frac{k}{\mathsf L_{opt}n}\right)^{\mathsf L_{opt}n-k} } & = \left[ 
\frac{4^{1 - \frac{k}{\mathsf L_{opt} n} }}{\left( 2- \frac{k}{\mathsf L_{opt} n} \right)^{2- \frac{k}{\mathsf L_{opt} n} }} \right]^{\mathsf L_{opt} n}\,.
\eea \eeq
We lighten notation by setting, for $x\in [0,1]$, 
\beq \label{1teta}
\widehat \Theta(x) \defi \frac{ 4^{1-x}  }{(2-x)^{2-x}} {\tanh\left( \mathsf E\left\{1-x\right\}\right)}^{\max \left(\frac{1}{\mathsf L_{opt}}-x,\frac{1-x}{4} \right)} {\cosh\left( \mathsf E\left\{1-x\right\}\right)}^{\frac{1}{\mathsf L_{opt}}}\,.
\eeq
With this notation, the r.h.s. of \eqref{137}  then reads
\beq\bea\label{140}
& \frac{n^{Kn^{\alpha}}P_n}{C_{n,K,m}^2Q_n} \sum_{k=200 \hat n_{K}+1}^{ \mathsf L_{opt}n} \frac{1}{{(1+\e_{E})}^{k}} \widehat \Theta \left(\frac{k}{\mathsf L_{opt}n} \right)^{ \mathsf L_{opt}n}   \\
& \qquad\qquad =\frac{n^{Kn^{\alpha}}P_n}{C_{n,K,m}^2Q_n} \left(\sum_{k= 200 \hat n_{K}+1}^{\frac{\mathsf L_{opt}n}{5}}+\sum_{k=\frac{\mathsf L_{opt}n}{5}+1}^{\mathsf L_{opt}n} \right) \frac{1}{{(1+\e_{E})}^{k}} \widehat \Theta\left(\frac{k}{\mathsf L_{opt}n} \right)^{ \mathsf L_{opt}n} \\
&\qquad \qquad =: (A)+(B),
\eea\eeq
say. In order to prove that these two terms vanish as $n \uparrow \infty$, we need the following

\begin{lem}\label{function} It holds:
\beq\bea\label{1f2}
\sup_{x \leq 1} \; \widehat \Theta(x) \leq 1\,.
\eea\eeq
Furthermore, for $x \leq \frac{1}{5}$,
\beq\bea\label{1f3}
\widehat \Theta(x) \leq \exp\left({-\frac{x}{100}}\right)\,.
\eea\eeq
\end{lem}
The proof of Lemma \ref{function} is given at the end of this section. We first use it to conclude the proof of  Proposition \ref{sum3}: using the bound \eqref{1f3} for the $(A)$-term yields
\beq\bea\label{461}
(A) &\leq \frac{n^{Kn^{\alpha}}P_n}{C_{n,K,m}^2Q_n} \sum_{k= 200 \hat n_{K}+1}^{\frac{\mathsf L_{opt}n}{5}} \frac{\exp{-\frac{k}{100}}}{{(1+\e_{E})}^{k}}\\
& \leq \frac{\exp{-\frac{200 n}{100 K}}}{C_{n,K,m}^2} \frac{n^{Kn^{\alpha}}P_n}{Q_n}\sum_{k= 200 \hat n_{K}+1}^{\frac{\mathsf L_{opt}n}{5}} \frac{1}{{(1+\e_{E})}^{k}},
\eea \eeq
since $x \mapsto \exp(-x)$ is decreasing.  Furthemore using that the above sum is convergent we thus see that
\beq\label{461.5}
(A)  \lesssim \frac{\exp{- 2 \frac{n}{K}}}{C_{n,K,m}^2} \frac{n^{Kn^{\alpha}}P_n}{Q_n} \,,
\eeq
Finally plugging the definition \eqref{defi_ccc} of $C_{n, K, m}$ into \eqref{461.5}, yields
\beq\bea\label{vanish2}
(A) & \leq \exp n \left[\frac{\sqrt{2}-2}{K}+\frac{2\sqrt{2}m(m-1)+2}{K^2} \right] \times \frac{n^{Kn^{\alpha}}P_n}{Q_n}\,.
\eea\eeq
But for $K>10^{7}$, the exponent on the r.h.s. above is $<0$, hence  the $(A)$-term  vanishes as $n \uparrow \infty$, settling the first claim. \\

\noindent As for the $(B)$-term ,  using \eqref{1f2} yields
\beq\bea\label{47}
(B) &\leq \frac{n^{Kn^{\alpha}}P_n}{C_{n,K,m}^2Q_n}\times \sum_{k=\frac{\mathsf L_{opt}n}{5}+1}^{\mathsf L_{opt}n} \frac{1}{{(1+\e_{\mathsf E})}^{k}}\\
&\leq  \frac{n^{Kn^{\alpha}}P_n}{C_{n,K,m}^2Q_n}\times \frac{\mathsf L_{opt}n}{{(1+\e_{\mathsf E})}^{\frac{\mathsf L_{opt}n}{5}}},
\eea\eeq
the last inequality majorizing with the largest term of the sum. Again plugging the definition \eqref{defi_ccc} of $C_{n, K, m}$ in \eqref{47}, and absorbing the $n$-factor in the $P$-polynomial, yields
\beq\bea\label{481}
(B) \leq  \exp n \left[\frac{\sqrt{2}}{K}+\frac{2\sqrt{2}m(m-1)+2}{K^2} \right] \times \frac{1}{{(1+\e_{\mathsf E})}^{\frac{\mathsf L_{opt}n}{5}}}\times \frac{n^{Kn^{\alpha}}P_n}{Q_n}.
\eea\eeq
By \eqref{diff_length}, it holds that $\mathsf L_{opt}> \mathsf L-\frac{m}{K}$, clearly implying that for any
$K>10^5$,
\beq\bea\label{l'}
1.25 \geq \mathsf L_{opt}\geq 1.24\,.
\eea\eeq
Using this in \eqref{481} yields 
\beq\bea\label{491}
(B) &\leq  \exp n \left[\frac{\sqrt{2}}{K}+\frac{2\sqrt{2}m(m-1)+2}{K^2} \right] \times \frac{1}{{(1+\e_{\mathsf E})}^{\frac{1.24 n}{5}}}\times \frac{n^{Kn^{\alpha}}P_n}{Q_n}\\
&= \exp n \left[\frac{\sqrt{2}}{K}+\frac{2\sqrt{2}m(m-1)+2}{K^2}-\frac{1.24 n}{5}\log\left(1+\e_{\mathsf E}\right) \right] \times \frac{n^{Kn^{\alpha}}P_n}{Q_n}.
\eea\eeq
Using the lower bound $\log(1+x)\geq x-\frac{x^2}{2}$ in \eqref{491} finally yields
\beq\bea \label{finitooooo}
(B)\leq \exp n \left[\frac{\sqrt{2}}{K}+\frac{2\sqrt{2}m(m-1)+2}{K^2}-\left(\e_{\mathsf E}-\frac{\e_{\mathsf E}^2}{2}\right)\frac{1.24}{5} \right] \times \frac{n^{Kn^{\alpha}}P_n}{Q_n}\,.
\eea\eeq
But for $K>\max(10^7,\e^{-2})$, the exponent is definitely strictly negative, hence the $(B)$-terms also vanishes as $n \uparrow \infty$, concluding the proof of the second claim. \\

In order to conclude the proof of Proposition \ref{sum3} we therefore owe to the reader a

\begin{proof}[ Proof of Lemma \ref{function}.] 
We first address claim \eqref{1f3}: since $\mathsf L_{opt}\leq \sqrt{2} \mathsf E \leq 1.25$, one plainly checks that for all $x \leq \frac{1}{5}$ it holds 
\beq
\max \left(\frac{1}{\mathsf L_{opt}}-x,\frac{1-x}{4} \right)=\frac{1}{\mathsf L_{opt}}-x\,,
\eeq
therefore 
\beq\bea\label{41'}
\widehat \Theta(x)&= \frac{{4}^{1-x}}{{(2-x)}^{2-x}} {\tanh\left(\mathsf E\left\{ 1-x\right\} \right)}^{\frac{1}{\mathsf L_{opt}}-x}{\cosh\left(\mathsf E\left\{ 1-x\right\} \right)}^{\frac{1}{\mathsf L_{opt}}}\\
&=\frac{{4}^{1-x}}{{(2-x)}^{2-x}} {\sinh\left(\mathsf E\left\{ 1-x\right\} \right)}^{\frac{1}{\mathsf L_{opt}}-x}{\cosh\left(\mathsf E\left\{ 1-x\right\} \right)}^{x}\,.
\eea\eeq
The following inequalities can be easily checked using the convexity of $x \mapsto \sinh\left(\mathsf E(1-x)\right)$, and of $x \mapsto \cosh\left(\mathsf E(1-x)\right)$, and constructing the corresponding chords between $x=0$ and $x=1$: it holds 
\beq\bea\label{41}
\sinh\left( \mathsf E(1-x)\right)\leq \left(1-x\right)\,, \; \text{and} \; \cosh\left( \mathsf E(1-x)\right)\leq \sqrt{2}+(1-\sqrt{2})x\,,
\eea\eeq
Combining \eqref{41'} and \eqref{41}, we obtain
\beq\bea\label{42}
\widehat \Theta(x) &\leq \frac{{4}^{1-x}}{{(2-x)}^{2-x}} {\left(1-x\right)}^{\frac{1}{\mathsf L_{opt}}-x} {\left(\sqrt{2}+(1-\sqrt{2})x \right)}^{x} \\
& = \frac{{2}^{2(1-x)-(2-x)}{\left(1-x \right)}^{\frac{1}{\mathsf L_{opt}}-x} {\left(\sqrt{2}+(1-\sqrt{2})x \right)}^{x}}{{\left(1-\frac{x}{2} \right)}^{\left(2-x\right)}},
\eea\eeq
the last step by rearrangement. Moreover, it holds that
\beq\bea\label{43}
1-x \leq \left(1-\frac{x}{2}\right)^2\,.
\eea\eeq
Simplifying the exponent of the first term in the numerator on the r.h.s. of \eqref{42}, and using \eqref{43} for the middle term, yields
\beq\bea\label{44}
\widehat \Theta(x)&\leq \frac{{2}^{-x}{\left(1-\frac{x}{2}\right)}^{2(\frac{1}{\mathsf L_{opt}}-x)} {\left(\sqrt{2}+(1-\sqrt{2})x\right)}^{x}}{{\left(1-\frac{x}{2}\right)}^{\left(2-x\right)}}\\
&= \left(\frac{1+\frac{(1-\sqrt{2})}{\sqrt{2}}x}{\sqrt{2}(1-\frac{x}{2})}\right)^{x}\times \frac{1}{\left(1-\frac{x}{2}\right)^{2(1-\frac{1}{\mathsf L_{opt}})}},
\eea\eeq
the last step again by simple rearrangements. 

Elementary inspection of the first derivative shows that, on the interval $[0, 1/5]$, the function
\beq
x \mapsto \frac{1+\frac{(1-\sqrt{2})}{\sqrt{2}}x}{(1-\frac{x}{2})}
\eeq 
is, in fact, increasing: bounding the function with its largest value attained in $x=1/5$, and plugging in \eqref{44}, yields
\beq\bea\label{45}
\widehat \Theta(x) &\leq \left(\frac{1+\frac{(1-\sqrt{2})}{\sqrt{2}}\frac{1}{5}}{\sqrt{2}\frac{9}{10}}\right)^{x}\times \frac{1}{(1-\frac{x}{2})^{2(1-\frac{1}{\mathsf L_{opt}})}}\\
&\leq \left(\frac{3}{4}\right)^{x}\times \frac{1}{(1-\frac{x}{2})^{2(1-\frac{1}{\mathsf L_{opt}})}}\,,
\eea\eeq
the second inequality by elementary numerical estimates.  Exponentiating the second term on the r.h.s. above then leads to 
\beq\bea\label{45'} 
\widehat \Theta(x) &\leq \left(\frac{3}{4}\right)^x \exp\left[ -2(1-\frac{1}{\mathsf L_{opt}})\log \left(1-\frac{x}{2} \right)\right]\\
&\leq\left(\frac{3}{4}\right)^x \exp\left[x \left(1-\frac{1}{\mathsf L_{opt}}\right) 2\log(2) \right],
\eea\eeq
where in the second step we have used that  
\beq 
-\log(1-\frac{x}{2})\leq x\log(2),
\eeq 
which is an immediate consequence of the convexity of $x \mapsto -\log(1-\frac{x}{2})$.  Recalling \eqref{diff_length}, and the ensuing elementary estimate $\mathsf L_{opt} < \sqrt{2} \mathsf E < 1.25$,
we thus see that
\beq\bea\label{46}
\widehat \Theta(x) &\leq \exp{ x\left[\log\left(\frac{3}{4} \right)+\left( 1-\frac{1}{1.25} \right) 2\log(2)\right] } \leq \exp\left[{-\frac{x}{100}}\right],
\eea\eeq
the second inequality by straightforward numerical evaluation: claim \eqref{1f3} is thus settled. \\

\noindent We now move to claim \eqref{1f2}. We recall that
\beq\bea\label{expl}
\widehat \Theta(x)&=\frac{ 4^{1-x}  }{(2-x)^{2-x}} {\tanh\left( \mathsf E\left\{1-x\right\}\right)}^{\max \left(\frac{1}{\mathsf L_{opt}}-x,\frac{1-x}{4} \right)} {\cosh\left( \mathsf E\left\{1-x\right\}\right)}^{\frac{1}{\mathsf L_{opt}}}\\
&=\frac{ 4^{1-x}}{(2-x)^{2-x}} {\sinh\left( \mathsf E\left\{1-x\right\}\right)}^{\frac{1}{\mathsf L_{opt}}-x} {\cosh\left( \mathsf E\left\{1-x\right\}\right)}^{x} 1_{\{ \max \left(\frac{1}{\mathsf L_{opt}}-x,\frac{1-x}{4} \right)=\frac{1}{\mathsf L_{opt}}-x\}}\\
&+\frac{ 4^{1-x}}{(2-x)^{2-x}} {\sinh\left( \mathsf E\left\{1-x\right\}\right)}^{\frac{1-x}{4}} {\cosh\left( \mathsf E\left\{1-x\right\}\right)}^{\frac{1}{\mathsf L_{opt}}-\frac{1-x}{4}} 1_{\{ \max \left(\frac{1}{\mathsf L_{opt}}-x,\frac{1-x}{4} \right)=\frac{1-x}{4}\}}.
\eea\eeq
By \eqref{diff_length}, it holds that $\mathsf L_{opt}> \mathsf L-\frac{m}{K}$ and this implies that for any
$K>10^5$,
\beq\bea\label{l}
\frac{1}{1.24}\geq \frac{1}{\mathsf L_{opt}}\geq \frac{1}{1.25}\,.
\eea\eeq
Let now
\beq
g_1(x) \defi \frac{{4}^{1-x}}{{(2-x)}^{2-x}}{\sinh\left(\mathsf E(1-x)\right)}^{\frac{1}{1.25}-x} {\cosh\left( \mathsf E(1-x)\right)}^{x}\,,
\eeq
\beq 
g_2(x)\defi \frac{{4}^{1-x}}{{(2-x)}^{2-x}}{\sinh\left(\mathsf E(1-x)\right)}^{\frac{1-x}{4}}{\cosh\left(\mathsf E(1-x)\right)}^{\frac{1}{1.24}-\frac{1-x}{4}}.
\eeq

In virtue of \eqref{l}, $g_1$ is larger than the first term in $\eqref{expl}$, whereas $g_2$ ls larger than the second one. In particular, setting $g_3\defi \min(g_1,g_2)$,  we see that in order to establish \eqref{1f2}  it suffices to prove that
\beq\bea
\sup_{x \in [0,1]}\; g_3(x)\leq 1\,,
\eea\eeq
which is our new claim. A plot of these two functions is given in Figure \ref{g1_g2} below. 

\begin{figure}[!h]
    \centering
    \includegraphics[scale=0.8]{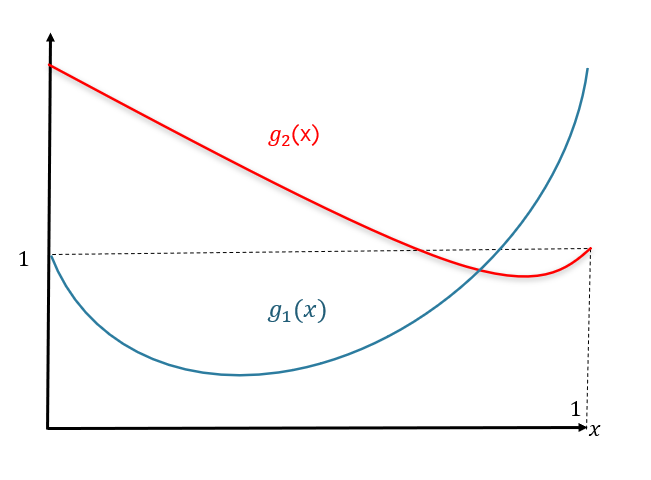}
\caption{The functions $g_1$ and $g_2$. One clearly sees that the minimum of these functions is always below 1. }
\label{g1_g2}
\end{figure}

To see this, we first note that  \eqref{1f3} already shows that 
\beq
\sup_{x\in [0, 1/5]} \; g_1(x)\leq 1\,.
\eeq
We now claim that
\beq\label{claimg1}
g_1 \; \text{is convex on} \; [0.12,0.73], \; g_1(0.12)\leq 1 \; \text{and} \; g_1(0.73) \leq 1.
\eeq
and that
\beq\label{claimg2}
g_2 \; \text{is convex on} \; [0.71,1], \; g_2(0.71) \leq 1 \; \text{and} \; g_2(1)=1.
\eeq
Assuming the validity of these two claims for the time being, it follows that 
\beq\label{g1}
g_1(x)\leq 1 \ \forall x\leq 0.73\,,
\eeq
and 
\beq\label{g2}
g_2(x)\leq 1 \ \forall x\geq 0.71\,.
\eeq
Combining \eqref{g1} and \eqref{g2} thus yields
\beq
\sup_{x \in [0,1]}\; g_3(x)\leq 1,
\eeq
and claim \eqref{1f2} is verified. \\

\noindent To conclude the proof of Lemma \ref{function} it thus remains to prove \eqref{claimg1} and \eqref{claimg2}. We begin with the convexity of $g_1$ on the interval $[0.12,0.73]$. Since $g_1>0$, 
\beq
\frac{d^2 \log(g_1)}{dx^2}=\frac{g_1''g_1-g_1'^2}{g_1^2} \geq 0 \Longrightarrow g_1''(x) \geq 0, 
\eeq
hence convexity of $\log(g_1)$ implies convexity of $g_1$: we will check the former by showing positivity of its second derivative. It holds:
\beq\bea\label{49}
\frac{d^2\log(g_1(x))}{dx^2}=& \frac{d^2}{{dx}^2}\Big[ (1-x)\log(4)+(-2+x)\log(2-x)\Big]+\\
& \hspace{ 3cm} +\frac{d^2}{{dx}^2}\Big[ x\log(\cosh\left(\mathsf E(1-x)\right)) \Big] + \\
& +\frac{d^2}{dx^2} \Big[ \left(-x+\frac{1}{1.25} \right)\log\sinh\left(\mathsf E(1-x)\right)
\Big]\,.
\eea\eeq
By elementary computations, we see that:
\beq\bea\label{50}
\frac{d^2}{dx^2}\Big[ (1-x)\log(4)+(-2+x)\log(2-x) \Big] =\frac{-1}{2-x},
\eea\eeq
\beq\bea\label{51}
\frac{d^2}{dx^2}\Big[ x \log(\cosh\left(\mathsf E(1-x)\right))\Big] &=\frac{d}{dx} \Big[ \log(\cosh\left(\mathsf E(1-x)\right))-x \mathsf E\tanh\left(\mathsf E(1-x)\right) \Big] \\
&=-2E\tanh\left(\mathsf E(1-x)\right)+\frac{x \mathsf E^2}{\cosh\left(\mathsf E(1-x)\right)^2}\,,
\eea\eeq
and finally
\beq\bea\label{52}
& \frac{d^2}{dx^2} \Big[ \left(-x+\frac{1}{1.25} \right)\log \sinh\left(\mathsf E(1-x)\right) \Big] \\
& \qquad =\frac{d}{dx} \Big[ -\log(\sinh\left(\mathsf E(1-x)\right))+\mathsf  E \left(x-\frac{1}{1.25} \right)\coth\left(\mathsf E(1-x)\right)\Big] \\
&\qquad =2 \mathsf E\coth\left(\mathsf E(1-x)\right)+\frac{\mathsf  E^2(x-\frac{1}{1.25})}{\sinh\left(\mathsf E(1-x)\right)^2}\,.
\eea\eeq
Since $1/5 \geq 0.12$, say, by the previous considerations we see that $g_1(0.12) \leq 1$. We may thus restrict to 
to $x \in [0.12, 0.73]$: we first note that the first function on the r.h.s. of \eqref{50} is decreasing. In particular, it holds that
\beq \label{49bis}
\frac{-1}{2-x} \geq \frac{-1}{2-0.73} \geq - 0.8\,.
\eeq

Plugging \eqref{50}-\eqref{52} in \eqref{49}, and then using \eqref{49bis} and the fact that $\frac{x \mathsf E^2}{\cosh\left(\mathsf E(1-x)\right)^2}\geq 0$, thus yields
\beq\bea\label{53}
\eqref{49}&\geq -0.8 -2 \mathsf E\tanh\left(\mathsf E(1-x)\right)+2 \mathsf E\coth\left(\mathsf E(1-x)\right)+\frac{\mathsf E^2(x-\frac{1}{1.25})}{\sinh\left(\mathsf E(1-x)\right)^2} 
\eea\eeq
We now make two observations.
\begin{itemize}
\item First of all we note that the r.h.s. of \eqref{53} consists of three increasing functions.
\item Furthermore, by Taylor expansions to fifth order, and some elementary yet tedious numerical estimates (which will be here omitted) one plainly checks that in $x=0.12$ the r.h.s. of \eqref{53} is, in fact, {\it positive}, whereas $g_1(0.73) \leq 1$. 
\end{itemize}
Combining the above items we see, in particular, that  $g_1$ is indeed convex on $[0.12,0.73]$, and the proof of claim \eqref{claimg1} is therefore concluded.\\

\noindent We now move to the analysis of $g_2$. Simple computations show that 
\beq\bea\label{54}
\frac{d^2}{dx^2}\Big[ \log(g_2(x))\Big] &=\frac{-1}{2-x}-\frac{\mathsf E}{2}\tanh\left(\mathsf E(1-x)\right)+\frac{\mathsf E^2(\frac{1}{1.24}+\frac{x-1}{4})}{\cosh\left(\mathsf E(1-x)\right)^2}+\frac{\mathsf E}{2}\coth\left(\mathsf E(1-x)\right)\\
&\hspace{ 3cm} +\frac{\mathsf E^2(x-1)}{4\sinh\left(\mathsf E(1-x)\right)^2}\\
&\geq -1-\frac{\mathsf E}{2}\tanh\left(\mathsf E(1-x)\right)+\frac{\mathsf E^2(\frac{1}{1.24}+\frac{x-1}{4})}{\cosh\left(\mathsf E(1-x)\right)^2}+\frac{\mathsf E}{2}\coth\left(\mathsf E(1-x)\right)\\
&\hspace{ 3cm} +\frac{\mathsf E^2(x-1)}{4\sinh\left(\mathsf E(1-x)\right)^2}\,,
\eea\eeq
the last inequality using that $\frac{-1}{2-x} \geq -1$. We now proceed in full analogy to \eqref{53}:
\begin{itemize}
\item 
First we note that the r.h.s. of  \eqref{54} consists of four increasing functions.
\item Furthermore, and again by some tedious yet elementary numerical estimates via Taylor expansions to fifth order (also omitted), one plainly checks that in $x=0.71$, say, the r.h.s. of  \eqref{54} is, in fact, {\it positive}, and $g_2(0.71) \leq 1$. 
\end{itemize}
Since the above items clearly imply, in particular, that $g_2$ is convex on $[0.71, 1]$, the second claim \eqref{claimg2} is also settled, and the proof of Lemma \ref{function} is thus concluded.
\end{proof}
\end{proof}

\subsection{Proof of Proposition \ref{sum2}}
We first state the technical input concerning the asymptotic of the $f^{(s)}$-terms. (As usual, $P_n, Q_n$ stand for finite degree polynomials, not necessarily the same at different occurences).

\begin{lem} \label{r}
For any $k\leq 200\hat n_{K}$, it holds
\beq\bea
f^{(s)}(n,k)\leq & {\left(\frac{3}{4}\right)}^{(m-200)\hat n_{K}} {\tanh\left( \mathsf E(1-\frac{k}{\mathsf L_{opt}n})\right)}^{n-k} \times \\
& \qquad \times {\cosh\left(\mathsf E(1-\frac{k}{\mathsf L_{opt}n})\right)}^{n}{\left(\frac{\mathsf L_{opt}n}{e \mathsf E}\right)}^{\mathsf L_{opt}n-k}n^{Kn^{\alpha}} P_n \,.
\eea\eeq
\end{lem}
The proof of this Lemma is also postponed to Section \ref{combinatorial_estimates}.

\begin{proof}[Proof of Proposition \ref{sum2}]. By  \eqref{first_out'}, it holds that
\beq\bea\label{29}
\frac{\mathsf J \sum_{k=1}^{ 200  \hat n_{K}}f^{(s)}(n,k) \PP\left(\pi, \pi_k^{(s)}\right)}{\E\left({\mathcal N}_{n, K, K'}^{\e}\right)^2}=\frac{\mathsf J \sum_{k=1}^{ 200  \hat n_{K}}f^{(s)}(n,k) \prod_{i=m}^{K-m+1}\PP_i\left(\pi, \pi_k^{(s)}\right)}{\E\left({\mathcal N}_{n, K, K'}^{\e}\right)^2}\,.
\eea\eeq
We claim that the r.h.s. of \eqref{29} converges to $0$ as $n \to \infty$. To see this, we follow {\it exactly} the same steps which from \eqref{129} lead to \eqref{137'}, this time of course with $f^{(s)}$ instead of $f$. Omitting the details, the upshot is that the r.h.s. of \eqref{29} is {\it at most} 
\beq\bea\label{37'}
\frac{P_n}{C_{n,K,m}^2Q_n} \sum_{k=1}^{200 \hat n_{K}}\frac{f^{(s)}(n,k) {(e \mathsf E)}^{\mathsf L_{opt}n-k}{g(\frac{k}{\mathsf L_{opt}n})^{\mathsf L_{opt}n}}}{{(1+\e_{E})}^{k}(\mathsf L_{opt}n-k)^{\mathsf L_{opt}n-k}},
\eea\eeq
The estimates from Lemma \ref{r} applied to \eqref{37'} then yield
\beq\bea\label{37}
&\eqref{37'}\leq \frac{{(\frac{3}{4})}^{(m-200)\hat n_{K}}n^{Kn^{\alpha}}P_n}{C_{n,K,m}^2 Q_n} \sum_{k=1}^{200 \hat n_{K}} \frac{{{\tanh(\mathsf E(1-\frac{k}{\mathsf L_{opt}n}))}^{n-k} {\cosh(\mathsf E(1-\frac{k}{\mathsf L_{opt}n}))}^{n}g\left(\frac{k}{\mathsf L_{opt}n}\right)^{\mathsf L_{opt}n}}}{{(1+\e_{\mathsf E})}^{k}{\left(1-\frac{k}{\mathsf L_{opt}n}\right)}^{\mathsf L_{opt}n-k}}.
\eea\eeq
As in \eqref{1simpler_g}, it holds that
\beq \bea \label{simpler_g}
\frac{g\left(\frac{k}{\mathsf L_{opt}n}\right)^{\mathsf L_{opt}n}}{
\left(1-\frac{k}{\mathsf L_{opt}n}\right)^{\mathsf L_{opt}n-k} } & = \left[ 
\frac{4^{1 - \frac{k}{\mathsf L_{opt} n} }}{\left( 2- \frac{k}{\mathsf L_{opt} n} \right)^{2- \frac{k}{\mathsf L_{opt} n} }} \right]^{\mathsf L_{opt} n}\,.
\eea \eeq
We lighten notation by setting, for $x\in [0,1/\mathsf L_{opt}]$, 
\beq \label{teta}
\Theta(x) \defi \frac{ 4^{1-x}  }{(2-x)^{2-x}} \tanh\left(\mathsf E(1-x) \right)^{\frac{1}{\mathsf L_{opt}} -x} \cosh\left( \mathsf E(1-x)\right)^{\frac{1}{\mathsf L_{opt}}}\,.
\eeq
Using this,  together with  \eqref{simpler_g}, the r.h.s. of \eqref{37} then takes the neater form 
\beq \label{38}
{\left(\frac{3}{4}\right)}^{(m-200)\hat n_{K}} \frac{n^{Kn^{\alpha}}P_n}{C_{n,K,m}^2Q_n} \sum_{k=1}^{200 \hat n_{K}} \frac{1}{(1+\e_{\mathsf E})^{k}}  \Theta \left( \frac{k}{\mathsf L_{opt} n}\right)^{\mathsf L_{opt} n}
\eeq

We recall that
\beq \label{choice_K}
K > 2 \times 10^7.
\eeq 
Thus, in the regime $k \leq 200 \hat n_K$, and since $\mathsf L_{opt}\geq 1$, we have
\beq
\frac{k}{\mathsf L_{opt} n} \leq \frac{200 n}{\mathsf L_{opt} K n} \leq \frac{200}{K} \leq 10^{-5}.
\eeq
We now claim that for all $x \leq 10^{-5}$,
\beq\label{equality}
\Theta(x)= \widehat \Theta(x)\,.
\eeq 
In fact,  for any $x \leq 10^{-5}$, 
\beq
\max\left(\frac{1}{\mathsf L_{opt}}-x, \frac{1-x}{4}\right)=\frac{1}{\mathsf L_{opt}}-x\,,
\eeq
as a simple numerical inspection shows: this proves \eqref{equality}.  

Combining Lemma \ref{function} and \eqref{equality}, thus yields
\beq\bea\label{f2}
\sup_{x \leq 10^{-5}}\; \Theta(x) \leq 1.
\eea\eeq
Using \eqref{f2} in \eqref{38} then  gives that
\beq\bea\label{bb}
\eqref{38} &\leq {\left(\frac{3}{4}\right)}^{(m-200)\hat n_{K}} \frac{n^{Kn^{\alpha}}P_n}{C_{n,K,m}^2Q_n} \sum_{k=1}^{200 \hat n_{K}} \frac{1}{(1+\e_{\mathsf E})^{k}} \\
& \lesssim {\left(\frac{3}{4}\right)}^{(m-200)\hat n_{K}} \frac{n^{Kn^{\alpha}}P_n}{C_{n,K,m}^2Q_n}\,,
\eea\eeq
since the sum is evidently convergent. Furthermore recalling the definition \eqref{defi_ccc} of $C_{n, K, m}$, we thus see that
\beq \bea
\eqref{bb} &  \lesssim {\left(\frac{3}{4}\right)}^{(m-200)\hat n_{K}}  \times \exp n \left[\frac{\sqrt{2}}{K}+\frac{2\sqrt{2}m(m-1)+2}{K^2} \right] \times \frac{n^{Kn^{\alpha}}P_n}{Q_n}\\
& =  \exp n \left[    \frac{1}{K} \left\{ (m-200) \log\left( \frac{3}{4} \right) + \sqrt{2} \right\}+\frac{2\sqrt{2}m(m-1)+2}{K^2} \right] \times \frac{n^{Kn^{\alpha}}P_n}{Q_n}
\eea \eeq
Since $m= 205$, 
\beq\bea\label{num}
(m-200)\log\left(\frac{3}{4}\right)+\sqrt{2}<-\frac{1}{100}\,,
\eea\eeq
(this bound is, as a matter of fact, the reason for choosing $m$ as we do), plugging \eqref{num} in \eqref{bb}, yields
\beq\bea\label{finito}
&(\ref{bb})\lesssim \exp{n\left[-\frac{1}{100K}+\frac{2\sqrt{2}m(m-1)+2}{K^2}\right]} \times \frac{n^{Kn^{\alpha}}P_n}{Q_n}.
\eea\eeq
But again in virtue of \eqref{choice_K}, and with $m=205$, 
\beq\label{choixK}
-\frac{1}{100K}+\frac{2\sqrt{2}m(m-1)+2}{K^2} < 0\,,
\eeq
as can be immediately checked:  the r.h.s. of \eqref{finito} is therefore vanishing as $n\uparrow \infty$, and the proof of Proposition \ref{sum2} is concluded. 
\end{proof}

\section{Combinatorial estimates} \label{combinatorial_estimates}

To control the asymptotics of the $f^{(d)}, f^{(s)}$ and $f$-terms requires some delicate path-counting.

\subsection{Counting directed paths, and proof of Lemma \ref{comb1}} \label{counting_dir}

Key to the whole treatment are estimates for the number of pairs of {\it directed} paths with prescribed overlaps which are formulated in Lemma \ref{path_counting} below. We shall emphasize that the estimates \eqref{path_counting_ii} and \eqref{path_counting_iiii} have been established by Fill and Pemantle \cite[Lemma 2.2, 2.4]{Fill_Pemantle}, whereas \eqref{path_counting_iii} can be found in \cite[Lemma 6]{kss1}.

\begin{lem}[Path counting directed, Fill and Pemantle] \label{path_counting} Let $\pi'$ be any reference path on the $n$-dim hypercube connecting $\boldsymbol{0}$ and
$\boldsymbol 1$, say $\pi'=12...n$. For $k \geq 1$, denote by $F(n, k)$ the number of directed paths $\pi$ that share \emph{precisely} $k$ edges with $\pi'$, and by $F^*(n, k)$ the number of paths that share $k$ edges with $\pi'$, without considering the first and the last edge. Finally, shorten $\mathfrak{n_{e}} \defi n-5e(n+3)^{2/3}$. It holds: 
\begin{itemize}
\item For all $k \geq  \mathfrak{n_{e}}$, we have
\beq \label{path_counting_i}
F(n,k)\leq (n-k)! {{n}\choose{ \mathfrak{n_{e}}}} \,,
\eeq
\item suppose $k\leq \mathfrak{n_{e}}$ for $n\geq 25$. Then, it holds
\beq \label{path_counting_ii}
F(n,k)\leq (n-k)! n^6 \,.
\eeq
\item For $k \leq n^{1/4}$, the stronger bounds hold
\beq \label{path_counting_iiii}
F(n,k) \leq (n-k)!(k+1)(1+o_n(1)), \, 
\eeq
and
\beq \label{path_counting_iii}
F^*(n,k) \leq (n-k-1)! (k+1)(1+o_n(1)), \, 
\eeq 
as $n \uparrow \infty$, uniformly in $k$. 
\end{itemize}
\end{lem}

\begin{proof} As mentioned, we only need to prove  \eqref{path_counting_i}: to this end, consider a directed path $\pi$ which shares precisely $k$ edges with the reference path  $\pi' =12\dots n$. We set $r_i = l$ if the $l^{th}$ traversed edge by $\pi$ is the $i^{th}$ edge shared by $\pi$ and $\pi'$. (We set by convention $r_0 \defi 0$, and $r_{k+1} \defi n+1$). Furthemore let $\textbf{r}\defi \textbf{r}(\pi)=(r_0,...,r_{k+1})$. For any sequence $\textbf{r}_0=(r_0,...,r_{k+1})$ with 0 = $r_0<r_1<...<r_k<r_{k+1}=n+1$, let $C(\textbf{r}_0)$ denote the number of paths $\pi$ with $\textbf{r}(\pi)=\textbf{r}_0$.  Since the values $\pi_{r_i+1},...,\pi_{r_{i+1}-1}$ must be a permutation of $\{r_i+1,...,r_{i+1}-1\}$,
it clearly holds that $C(\textbf{r})\leq G(\textbf{r})$, where 
\beq\label{G}
G(\textbf{r})\defi \prod\limits_{i=0}^{k}(r_{i+1}-r_i-1)! \,.
\eeq
Iterating the log-convexity \eqref{log_convexity} of factorials in its simplest form: $a! b! \leq (a+b)!$, yields
\beq\label{G}
G(\textbf{r})\leq \left(\sum_{i=0}^{k} r_{i+1}-r_i-1\right)!=\left(n+1-(k+1)\right)!=\left(n-k\right)! \,,
\eeq
which implies, in particular, that there are at most $(n-k)!$ paths sharing $k$ edges with a reference-path $\pi'$ for {\it given} \textbf{r}-sequence. But since there are $\binom {n}{k}$ ways to choose such  \textbf{r}-sequences
we obtain
\beq\bea
F(n,k)\leq (n-k)! \binom {n}{k}\,.
\eea\eeq
Since the factorial term on the r.h.s. above is decreasing in $k$ for $k \geq \lceil \frac{n}{2}\rceil$, we deduce that for $k\geq  \mathfrak{n_{e}} \gg \frac{n}{2}$,
\beq
(n-k)! \binom {n}{k} \leq (n-k)!\binom {n}{ \mathfrak{n_{e}}},
\eeq
settling the proof of \eqref{path_counting_i}. 
\end{proof}

Armed with the above estimates on the number of directed paths with prescribed overlaps, we can move to the 

\begin{proof} [Proof of Lemma \ref{comb1}]
For $\pi \in \mathcal J$ and $\pi_k^{(d)} \in \mathcal J_\pi^{(d)}(n,k)$, let us denote by $k_l$ the number of common edges between $\boldsymbol 0$ and $H_m$, and by $k_r$ the number of common edges between $H_{K-m}$ and $\boldsymbol 1$ (in which case it evidently holds that $k=k_l+k_r$). Furthermore, let
\beq \bea
f^{(d)}_{\pi}(n,k,k_l)\defi & \;  \text{\sf all paths}\; \pi' \in \mathcal J\; \text{\sf which share} \; k \; \text{\sf edges with}\; \\
& \; \pi \; \text{\sf only in the directed phase, i.e between} \; \\
&\boldsymbol 0 \; \text{\sf and} \; H_m \; \text{\sf or} \; H_{K-m}  \;  \text{\sf and} \boldsymbol  \;  1,\\
& \text{\sf with $k_l$ edges in common between $\boldsymbol{0}$ and $H_m$,}\\
& \text{\sf but without considering  first and  last edge.}
\eea \eeq
We have
\beq\bea\label{claimcomb}
f^{(d)}_{\pi}(n,k)&= \sum_{k_l \geq k_r} f_{\pi}^{(d)}(n,k,k_l)+ \sum_{k_l < k_r}  f_{\pi}^{(d)}(n,k,k_l)\\
&\leq \sum_{k_l \geq k_r}  f_{\pi}^{(d)}(n,k,k_l)+ \sum_{k_l \leq k_r}  f_{\pi}^{(d)}(n,k,k_l).
\eea\eeq
We claim that 
\beq\label{claimcomb'}
\sum_{k_l \geq k_r} f_{\pi}^{(d)}(n,k,k_l)=\sum_{k_l \leq k_r}  f_{\pi}^{(d)}(n,k,k_l).
\eeq
This claim is perhaps surprising at first sight, as $k_l$ and $k_r$ cannot be simply swapped. The idea is over to work through bijections relating the (pair) of paths appearing in the first sum to those in the second one.

Indeed, each vertex on the right side of the hypercube stands in one to one correspondence with a vertex on the left side: the (trivial) bijection here amounts to changing the $1's$ into $0's$  (and the $0's$ into $1's$). 

Furthermore, by \eqref{sym_k}, backsteps and forward steps are symmetric around the center of the hypercube, meaning that for $i\in \{m+1, K-m\}$, 
\beq
\mathsf{eb}_i= \mathsf{eb}_{K-i+1} \qquad \text{and} \qquad  \mathsf{ef}_i= \mathsf{eb}_{K-i+1}\,.
\eeq
This, together with the fact that polymers are stretched, implies that the number of subpaths reaching two given vertices between $H_i$ and $H_{i+1}$, and the number of those between $H_{K-(i+1)}$ and $H_{K-i}$ do in fact coincide.  

Finally, we note that the "cone" of vertices in $H_{i+1}$ which are attainable from a vertex in $H_i$  in the first half of the hypercube is in one-to-one correspondence with the vertices in $H_{K-(i+1)}$ which lead to a given vertex in $H_{K-i}$ (this can immediately seen by changing the $1$-coordinates of a vertex into $0$, or the other way around). Thus, for each cone on the left side of the hypercube, we find a cone on the right side which evolves in  the opposite direction, settling claim \eqref{claimcomb'}.

Using  \eqref{claimcomb'} in \eqref{claimcomb} yields
\beq\bea\label{claimsym}
f_{\pi}^{(d)}(n,k)& \leq 2\sum_{k_l \geq k_r}  f_{\pi}^{(d)}(n,k,k_l).
\eea\eeq
We now make the following  key observation: counting the number of directed subpaths which share $k_l$ edges with $\pi$ (disregarding the first edge) between $\boldsymbol{0}$ and any admissible point of $H_m$ is equivalent to counting the number of directed subpaths $\pi'$ that share $k_l$ edges with the directed subpath of $\pi$, but on a hypercube of dimension $m \hat n_K$ (again disregarding the first edge). By symmetry, the same of course holds true for the number of subpaths between $H_{K-m}$ and $\boldsymbol{1}$ (this time disregarding the last edge). The new goal is thus to solve the path-counting problem on these hypercubes of smaller dimensions. In order to do so, we focus on the rightmost edge shared by both polymers, and denote by
\beq
d_l \defi d\left( \pi_{r_{k_l}}, H_{m} \right)
\eeq
its Hamming distance to the $H_m$-plane. We now distinguish between two cases: the first case concerns the situation where $d_l=0$, whereas the second case concerns $d_l>0$.

If $d_l=0$, the rightmost common edge leads directly into the $H_m$-plane. Any subpath sharing $k_l$ edges with $\pi$ can thus reach one vertex only on the target plane: counting the number of subpaths connecting $\boldsymbol{0}$ and this prescribed vertex,  {\it while disregarding the first edge}, is therefore equivalent to estimating the number of directed paths which share $k_l-1$ edges on a hypercube of dimension $m \hat n_K-1$, also disregarding the first edge. We will solve the latter problem with the help of $F$, in which case a small detail must be taken into account. In fact, contrary to our current situation, the first edge does matter in the definition of $F$. We thus have to distinguish between the case whether the first edge is shared, respectively: not shared, by both paths. In both cases we need to specify $k_l-1$ common edges disregarding first and "last" edge: in the first case the number of commond edges is, in fact, $(k_l-1)+1=k_l$, and this leads to at most $F\left(m \hat n_K-1,k_l\right)$ ways to choose them. In the second case the problem of the "hidden" (first) shared edge is not present, and we simply have {\it at most} $F\left(m \hat n_K-1,k_l-1\right)$ possibilities to choose the common edges. All in all, for the number of directed paths sharing $k_l$ common edges (first one excluded), and $d_l=0$,
we have the rough bound
\beq \bea \label{F2}
F\left(m \hat n_K-1,k_l\right)+F\left(m \hat n_K-1,k_l-1\right) \leq 2F\left(m \hat n_K-1,k_l-1\right),
\eea \eeq
using for the inequality that $j \mapsto F(n,j)$ is decreasing. \\

We now move to the case $d_l >0$ and first note that by  definition of $f^{(d)}(n,k)$, 
neither first nor the last edges can be a common edge. The number of subpaths, which are sharing $k_l$ edges between $\boldsymbol{0}$ and $H_m$ with $\pi$ whithout considering the first and the last edge is thus at most
\beq \label{sainttropez}
\left( \# \; \text{admissible vertices in} \; H_m \right) \times F^*(m \hat n_K,k_l) \,.
\eeq
We claim that 
\beq \label{binom}
\# \; \text{admissible vertices in} \; H_m = {{n-(m \hat n_K-d_l)}\choose{d_l}}\,.
\eeq
Indeed, of the $n$ possible $1$-coordinates, $(m \hat n_K-d_l)$ many are already specified by the rightmost common edge; furthermore, in order to reach any of the admissible points on $H_m$ we may switch, out of $n-(m \hat n_K-d_l)$ $0$-coordinates, $d_l$ many into $1's$: \eqref{binom} thus follows by simple counting. 

Next we claim that $j \mapsto {n+ j \choose j}$ is increasing. To see this, we write
\beq
\binom{n+j}{j} = \frac{(n+j) \dots (j+1)}{n!},
\eeq
and observe that the term in the numerator on the r.h.s. above is increasing. It follows in particular, that the r.h.s. of \eqref{binom} is maximized for $d_l = m \hat n_K-k_l-1$ (recall that we are not considering the first edge), and therefore
\beq \label{F3}
\eqref{sainttropez} \leq {{n-k_l-1}\choose{m \hat n_K-k_l-1}} \times F^*(m \hat n_K,k_l)\,.
\eeq
Combining \eqref{F2} and \eqref{F3}, we thus see that the overall number of subpaths sharing $k_l$ edges on the "left side" of the hypercube (i.e. between $\boldsymbol{0}$ and $H_m$, but without considering the first edge) with a reference path $\pi$ is less than
\beq\label{debutdirected}
2 F(m \hat n_K-1,k_l-1)+F^*(m \hat n_K,k_l) \times {{n-k_l-1}\choose{m \hat n_K-k_l-1}}.
\eeq

We next move to the "right side" of the hypercube: in full analogy to the considerations leading to \eqref{F2}, one sees that the number of subpaths sharing $k_r$ edges between a point on $H_{K-m}$ and $\boldsymbol{1}$ with a given reference path $\pi$ (disregarding, in this case, the last edge), is less than
\beq\label{findirected}
F(m \hat n_K,k_r)+F(m \hat n_K,k_r+1) \leq 2F(m \hat n_K,k_r).
\eeq

The bounds \eqref{debutdirected} and \eqref{findirected} address "left" and "right" side of the hypercube on separate footing: for these bounds to be of any use in estimating the $f^{(d)}(n, k, k_l)$-terms appearing in \eqref{claimsym}, left and right side must be connected.  We will do so by slightly "overshooting", insofar we do not take into account the fact that the number of subpaths connecting $H_m$ and $H_{K-m}$ is {\it reduced}, as soon as shared edges on the right region are specified. Recalling that $\mathsf J = \#\mathcal J$ takes the  form
\beq\label{form J}
\mathsf J= \underbrace{\left( m  \hat n_K   \right)! \dbinom{n}{  m  \hat n_K}}_{\text{directed}} \times \underbrace{\mathsf J_s}_{\text{stretched}} \times \underbrace{\left( m  \hat n_K   \right)!}_{\text{directed}},
\eeq
with $\mathsf J_s$ denoting the number of subpaths between a given vertex on $H_m$ and the $H_{K-m}$-plane,  it follows from \eqref{debutdirected}, \eqref{findirected} and the aforementioned overshooting, that 
\beq\bea\label{estimfd}
 f_{\pi}^{(d)}(n,k,k_l)&\leq \left( 2F(m \hat n_K-1,k_l-1)+F^*(m \hat n_K,k_l) \times {{n-k_l-1}\choose{m \hat n_K-k_l-1}} \right) \times\\
& \hspace {4 cm} \mathsf J_s  \times 2F(m \hat n_K, k_r).\\
\eea\eeq

The above is our fundamental estimate. Remark in particular, that it holds uniformly over $k = k_l+ k_r$. To proceed further we will now distinguish two cases: either $k \leq n^{\frac{1}{4}}$ or $k > n^{\frac{1}{4}}$. \\

\noindent {\sf First case: $k \leq n^{\frac{1}{4}}$}. We begin with an estimate for the terms in the large brackets of the r.h.s. of \eqref{estimfd}.  In the considered $k$-regime, we may use the bounds provided by Lemma \ref{path_counting}: display \eqref{path_counting_iiii} yields the bound
\beq \label{tante}
F(m \hat n_K-1,k_l-1) \leq (k_l+1)(m \hat n_K-k_l)! \left[ 1+o_n(1) \right] \leq 2 (k_l+1)(m \hat n_K-k_l)! \,,
\eeq
for $n$ large enough, whereas display \eqref{path_counting_iii} of the same Lemma yields, for the $F^*$-term on the r.h.s. of \eqref{estimfd} the bound
\beq \label{tante1}
F^*(m \hat n_K,k_l) \leq  (k_l+1)(m \hat n_K-k_l-1)! \left[ 1+o_n(1) \right] \leq 2(k_l+1)(m \hat n_K-k_l-1)!\,,
\eeq
which holds again for large enough $n$. Combining \eqref{tante} and \eqref{tante1} we thus get that the terms in the large brackets of the r.h.s. of \eqref{estimfd} are {\it at most}
\beq\bea\label{debutdirected1}
& 4(k_l+1)(m \hat n_K-k_l)!+2(k_l+1)(m \hat n_K-k_l-1)! {{n-k_l-1}\choose{m \hat n_K-k_l-1}} \\
& \hspace{5cm} \leq 4(k_l+1)(m \hat n_K-k_l-1)! \times {{n-k_l-1}\choose{m \hat n_K-k_l-1}},
\eea \eeq
the second inequality since $n-k_l-1 \geq m \hat n_K-k_l-1\geq 5 \hat n_K-1$ (see $m=205$ and $k\leq 200 \hat n_K$) implies that the second term on the l.h.s. above is (exponentially) larger than the first one.

We may again use the bounds provided by Lemma \ref{path_counting}, display \eqref{path_counting_iiii}, akin to \eqref{tante}, and we obtain
\beq\label{findirected1}
2F(m \hat n_K,k_r) \leq 4(k_r+1) (m \hat n_K-k_r)!.
\eeq

Plugging the estimates \eqref{debutdirected1} and \eqref{findirected1} into \eqref{estimfd}, we obtain
\beq\bea\label{find*}
 f_{\pi}^{(d)}(n,k,k_l)&\leq 16 (k_l+1)(m \hat n_K-k_l-1)!{{n-k_l-1}\choose{m \hat n_K-k_l-1}} \mathsf J_s  (k_r+1)(m \hat n_K-k_r)! \\
&=16 (k_l+1)(m \hat n_K-k_l-1)!{{n-k_l-1}\choose{m \hat n_K-k_l-1}}  \frac{\mathsf J (k_r+1)(m \hat n_K-k_r)! }{{( m \hat n_K)!}^2\dbinom{n}{ m \hat n_K}},\\
\eea\eeq
the last equality expressing $\mathsf J_s$ as a function of $\mathsf J$ via the relation \eqref{form J}. Writing out the binomials, and after some elementary simplifications, \eqref{find*} becomes 
\beq\bea\label{find}
 f_{\pi}^{(d)}(n,k,k_l)&\leq 16  (k_l+1)(k_r+1) (n-k_l-1)! (m \hat n_K- k_r)!  \frac{\mathsf J}{( m \hat n_K)!n!}.\\
\eea\eeq
In order to estimate the r.h.s. of  \eqref{find}, we recall that $k=k_r+k_l$, hence
\beq\label{t1}
\left(k_l+1\right)\left(k_r+1\right)\leq \left(k+1\right)^2.
\eeq
Furthermore, we claim that
\beq\bea\label{wc}
(n-k_l-1)! (m \hat n_K- k_r)!&\leq \left(n-\lceil \frac{k}{2} \rceil-1\right)! \left(m \hat n_K-\lfloor \frac{k}{2} \rfloor \right)!.\\
\eea\eeq
To see this, we will make use of the log-convexity \eqref{log_convexity} with $a\defi \lceil n-k_l-1 \rceil$ and $b \defi \lfloor m \hat n_K- k_r \rfloor$, in which case it clearly holds that $a>b$ for any choice of $k_l=k-k_r$ with $k \leq 200 \hat n_K$. By log-convexity we thus obtain
\beq\bea\label{wc'}
(n-k_l-1)! (m \hat n_K- k_r)!&\leq (n-k_l-1+1)! (m \hat n_K- k_r-1)!\\
&=\left(n-k_l'-1\right)! \left(m \hat n_K- k_r'\right)!\, ,
\eea\eeq
where $k_l'\defi k_l-1$ and $k_r' \defi k_r+1$. Iterating the $\log$-convexity as in \eqref{wc'} and taking into account that $k_l \geq k_r$ gives that the r.h.s. of  \eqref{wc'} is maximized in $k_l=\lceil \frac{k}{2} \rceil$, 
settling the claim \eqref{wc}.

Plugging \eqref{t1}  and \eqref{wc} into \eqref{find} then yields
\beq\bea\label{find1}
 f_{\pi}^{(d)}(n,k,k_l)&\leq 16 (k+1)^2 \left(n-\lceil \frac{k}{2} \rceil-1 \right)! \left(m \hat n_K- \lfloor \frac{k}{2} \rfloor \right)!  \frac{\mathsf J}{( m \hat n_K)!n!}.\\
\eea\eeq
All in all, using \eqref{claimsym} and \eqref{find1}, we have
\beq\bea\label{find11}
f_{\pi}^{(d)}(n,k)&\leq 2 \sum_{k_l\geq k_r} 16(k+1)^2 \left(n-\lceil \frac{k}{2} \rceil-1 \right)! \left(m \hat n_K- \lfloor \frac{k}{2} \rfloor\right)!  \frac{\mathsf J}{( m \hat n_K)!n!}\\
& \leq 32(k+1)^3 \left(n-\lceil \frac{k}{2} \rceil-1 \right)! \left(m \hat n_K- \lfloor \frac{k}{2} \rfloor \right)!  \frac{\mathsf J}{( m \hat n_K)!n!},
\eea\eeq
the last inequality since $k_l+k_r = k$, implying that the sum consists at most of $k+1$ terms.\\

\noindent {\sf Second case: $k > n^{\frac{1}{4}}$}. Note that we additionally require that $k \leq  200 \hat n_K$.
On the other hand,  $200 \hat n_K \leq \mathfrak{n_{e}}$, by definition.  This implies, in particular, that $k \leq \mathfrak{n_{e}}$: we are thus in the \eqref{path_counting_ii}-regime. Recalling the definition of $F^{*}$, the upperbound clearly holds 
\beq\bea\label{tante1*}
F^*(m \hat n_K,k_l) &\leq F(m \hat n_K,k_l)+ F(m \hat n_K,k_l+1)+F(m \hat n_K,k_l+2)\\
& \leq n^6 (n-k_l)! \left(1+\frac{1}{(n-k_l)}+\frac{1}{(n-k_l)(n-k_l-1)} \right) \\ 
& \leq n^7 (n-k_l-1)!2,  \hspace{4cm} \text{(n large enough)}
\eea\eeq
the second inequality by \eqref{path_counting_ii}. Following {\it exactly} the same steps which lead from \eqref{estimfd}  to \eqref{find11}, again using the Lemma \ref{path_counting} but this time with the estimate \eqref{path_counting_ii} and replacing \eqref{tante1} by \eqref{tante1*}, one immediately obtains 
\beq\bea
f_{\pi}^{(d)}(n,k)&\leq 16n^{13} (k+1) \left(n-1-\lceil \frac{k}{2} \rceil \right)! \left(m \hat n_K- \lfloor \frac{k}{2} \rfloor \right)!  \frac{\mathsf J}{( m \hat n_K)!n!},
\eea\eeq
for all $\pi \in \mathcal J$, concluding the proof of Lemma \ref{comb1}.
\end{proof}

\subsection{Counting undirected paths, and proofs of Lemmata \ref{r2} and \ref{r}} \label{counting_undir}
Thanks to the repulsive nature of the $H$-planes, if two paths share two edges between a different pair of $H$-planes, the common edge with the smaller Hamming distance to $\boldsymbol 0$ is evidently crossed first. Given that paths eventually proceed according to the inherent directivity of the problem ("from left to right"), one may ask a similar question for the way two (or more) common edges between two successive $H$-planes (in the stretched phase) are crossed. To address this question, we will distinguish between two concepts: {\it i)} {\sf directionality}, i.e. whether the path performs, while crossing the considered edge, a forward- or a backstep, and {\it ii)} {\sf order} in which the considered edges are crossed\footnote{In hindsight, we only need two distinctions here: either the two paths cross the edges in the same, or in reverse order. We will avoid explicit definitions for this intuitive concept, but provide an example: assuming that the common edges are labeled {\sf a,b,c,d, etc.}, the order in which a path crosses them is simply  the order of the labels: assume the path $\pi$ crosses the edges in the order {\sf a-b-c-d}; the path $\pi'$ can cross the same edges either in exactly the same order {\sf a-b-c-d}, or in reverse order {\sf d-c-b-a}.}.

\begin{lem}\label{rem1} 
Let $\pi, \pi' \in \mathcal J$ share edges between the $H_{i-1}$- and the $H_i$-plane, for some $i \in \{m+1, \dots, K-m\}$, and assume that the $\pi$-path crosses the common edges in a certain directionality and order. Then the $\pi'$-path has to cross the edges either 
\begin{itemize}
\item in the same directionality \emph{and} order, \\
or 
\item in opposite directionality \emph{and} reverse order. 
\end{itemize}
\end{lem}

\begin{figure}[!h]
    \centering
    \begin{subfigure}[b]{0.45\textwidth}
        \includegraphics[width=\textwidth]{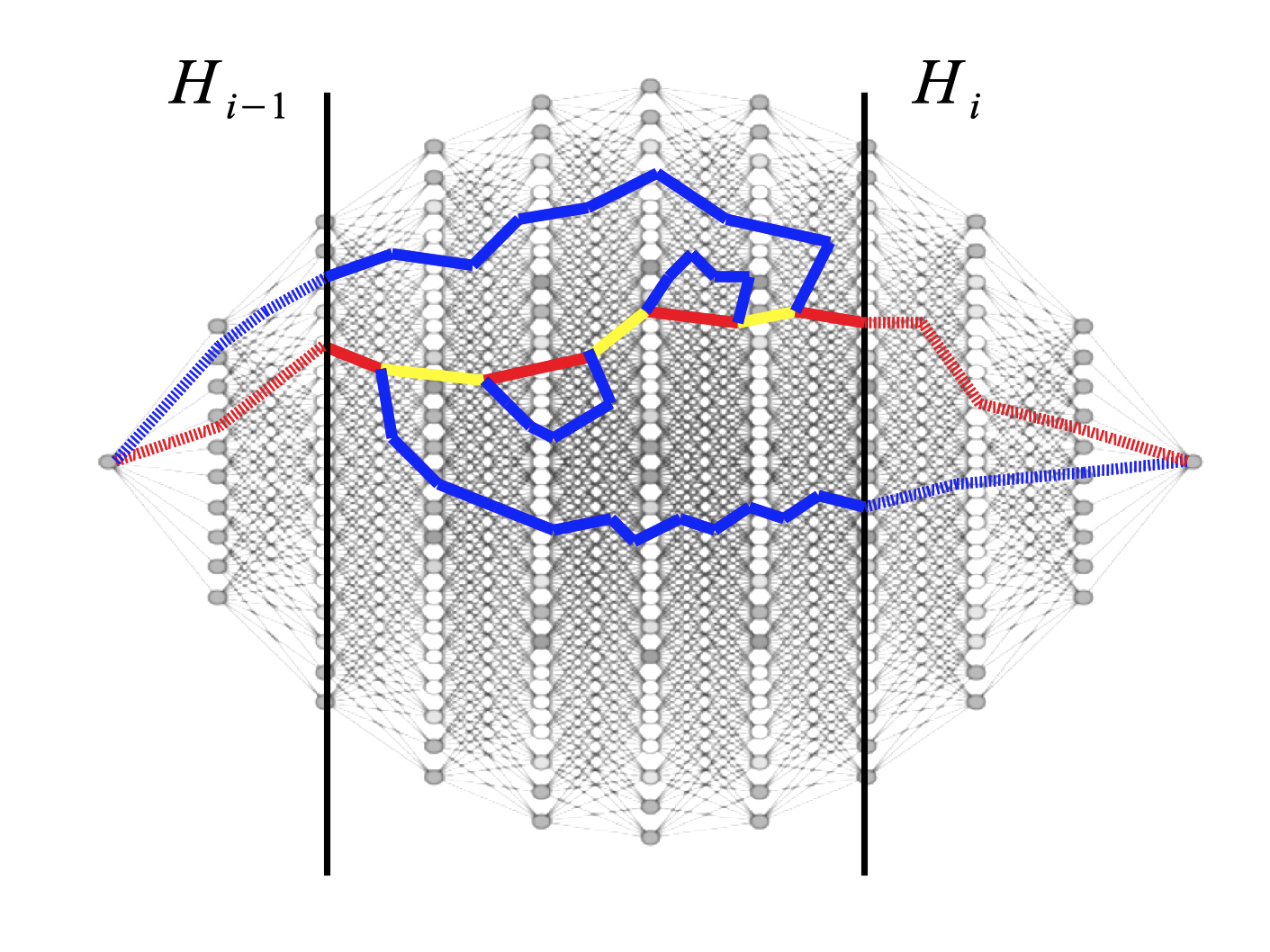}
    \end{subfigure}
    \begin{subfigure}[b]{0.45\textwidth}
        \includegraphics[width=\textwidth]{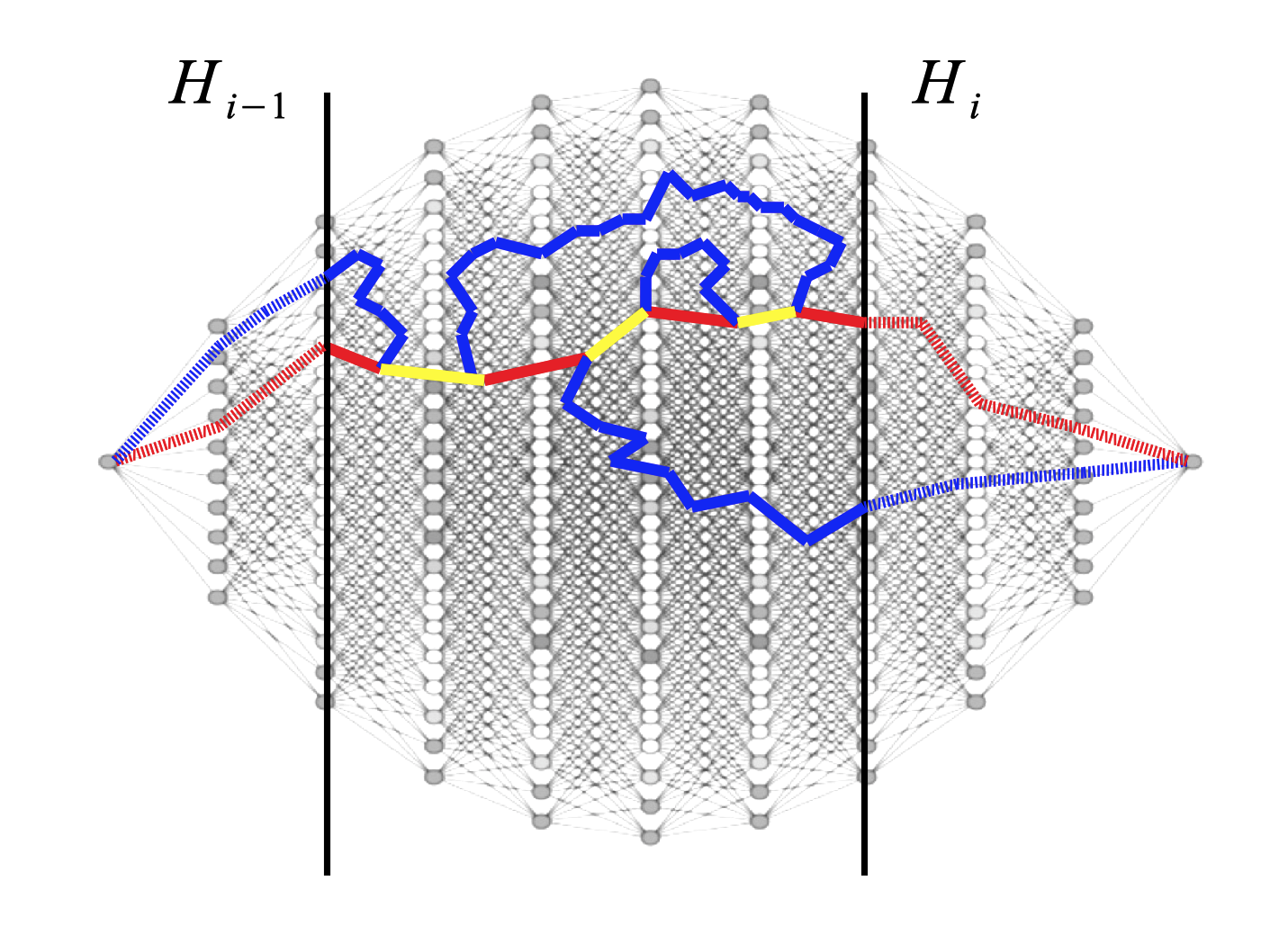}
    \end{subfigure}
    \caption{The yellow edges are shared by both polymers. The picture on the left satisfies the directionality: the red polymer crosses the yellow edges in graphical order "from left to right", while the blue polymer crosses the yellow edges in reversed order and opposite directionality. The picture on the right does not: the blue polymer first crosses the first common edge, but then reverts both order and directionality.}
\label{dir_nodir}
\end{figure}

\begin{proof} [Proof of Lemma \ref{rem1}] 

Consider a path $\pi$, and the associated directionality/order in which it crosses the prescribed, common edges. A second path $\pi'$  which does not follow such directionality and order (nor its complete reversal)  will move away from one of the shared edges which are bound to be crossed in a future step. The second path will thus  have to make up for this "departure", eventually, but this can only happen if it performs, during its evolution, a {\it detour}, i.e. if it goes through an edge (parallel to one of the unit vectors) in {\it both} directions. Since detours are not possible in the stretched phase at hand, the claim follows repeating the line of reasoning.  

\end{proof}

\begin{proof}[Proof of Lemma \ref{r2}]
Consider $\pi \in \mathcal J$, and  $\boldsymbol k=(k_1, k_2, \dots , k_K) \in {\N}^{K}$, such that $k_1+k_2+\dots+k_K=k$.  By a slight abuse of notation we denote by  $f_\pi(n, \boldsymbol k)$ the number of paths which share $k_i$ edges with $\pi$ between the hyperplanes $H_{i-1}$ and $H_{i}, \, i \in \{1,\dots, K\}$. It then holds
\beq\bea \label{estimate1}
f_{\pi}(n,k)=\sum_{\boldsymbol k}f_{\pi}(n, \boldsymbol k).
\eea\eeq
If $k_i>0$, let $v_{i}^{\text{fi}}$ be the {\it first} vertex which $\pi$ hits when crossing the first common edge between $H_{i-1}$ and $H_{i}$, and $v_{i}^{\text{la}}$ the {\it last} vertex from which $\pi$ departs after crossing the last common edge (also between $H_{i-1}$ and $H_{i}$).  Furthermore, denote by
\beq\bea
l_{i}^{\text{fi}}(\pi) \defi d\left(H_{i-1} \cap \pi, v_{i}^{\text{fi}} \right), \qquad l_{i}^{\text{la}}(\pi) \defi d\left(v_{i}^{\text{la}}, H_{i}\cap \pi \right),
\eea\eeq
the Hamming distance from (resp. to) the first (resp. last) vertex to the previous (resp. next) H-plane.  If $k_i=0$, we simply set $l_{i}^{\text{fi}}(\pi) \defi d(H_{i-1} \cap \pi, H_{i} \cap \pi)$ and $l_{i}^{\text{la}}\defi 0$. 

Finally, consider the whole list (vector) of Hamming distances 
\beq
\boldsymbol l(\pi) \defi \left( l_{1}^{\text{fi}}(\pi), l_{1}^{\text{la}}(\pi), l_{2}^{\text{fi}}(\pi), l_{2}^{\text{la}}(\pi), \dots, l_{K}^{\text{fi}}(\pi), l_{K}^{\text{la}}(\pi) \right) \in {\N}^{2K}.
\eeq

Let $f_\pi(n, \boldsymbol k, \boldsymbol l)$ the number of paths sharing $k_i$ edges with $\pi$ between the hyperplanes $H_{i-1}$ and $H_{i}$, $i=1=\dots K$, and with prescribed $\boldsymbol l$-vector. It then holds
\beq\bea \label{estimate2}
f_{\pi}(n,k)=\sum_{\boldsymbol k} \ \sum_{\boldsymbol l}f_{\pi}(n, \boldsymbol k, \boldsymbol l)\\
\eea\eeq
By Lemma \ref{rem1},  a path $\hat \pi \in \mathcal J_\pi(n,k)$ has  two ways only to travel through the common edges between successive H-planes: either in identical, or opposite directionality/order. In order to keep track of this, we consider the $\boldsymbol \sigma \defi ({\sigma}_1, \dots, {\sigma}_K) \in \{-1,1\}^{K}$ with coordinates given by
\beq
{\sigma}_i \defi+1, \qquad \text{\sf if}\; k_i=0, 
\eeq 
and  
\beq
{\sigma}_i \defi
\begin{cases}
+1, & \text{\sf if} \; \hat \pi \; \text{\sf crosses first} \; v_i^{\text{fi}}, \\
-1, & \text{\sf if} \; \hat \pi \; \text{\sf crosses first} \; v_i^{\text{la}}\,.
\end{cases}
\qquad \text{and} \; k_i>0. 
\eeq

We need some additional notation: if $k_i>0$ and in case of identical directionality/order, i.e. ${\sigma}_i =+1$, we set 
\beq \bea
\hat l_{i}^{\text{fi}}(\hat \pi) & \defi \text{\sf length of the substrand connecting the vertices}\; 
H_{i-1} \cap \hat \pi \; \text{\sf and} \; v_i^{\text{fi}}\,, \\
\hat l_{i}^{\text{la}}(\hat \pi) & \defi \text{\sf length of the substrand connecting} \; v_i^{\text{la}}\; \text{\sf and}\; H_{i} \cap \hat \pi, \\
\hat{v}_i^{\text{fi}}& \defi v_i^{\text{fi}}, \\
\hat{v}_i^{\text{la}}& \defi v_i^{\text{la}}\,.
\eea \eeq
If $k_i>0$ and in case of reverse directionality/order, i.e. $\s_i =-1$,  we set 
\beq \bea
\hat l_{i}^{\text{fi}}(\hat \pi) & \defi \text{\sf length of the substrand connecting the vertices}\; 
H_{i-1} \cap \hat \pi \; \text{\sf and} \; v_i^{\text{la}}\,, \\
\hat l_{i}^{\text{la}}(\hat \pi) & \defi \text{\sf length of the substrand connecting} \; H_{i} \cap \hat \pi \; \text{\sf and}\; v_i^{\text{fi}}, \\
\hat{v}_i^{\text{fi}}& \defi v_i^{\text{la}}, \\
\hat{v}_i^{\text{la}}& \defi v_i^{\text{fi}}\,.
\eea \eeq

If $k_i=0$, we simply set 
\beq \bea
\hat l_{i}^{\text{fi}}(\hat \pi) & \defi l_{\pi}(H_{i-1} \cap \pi, H_{i} \cap \pi), \\
\hat l_{i}^{\text{la}}(\hat \pi) & \defi 0. \\
\eea \eeq
Furthermore, let
\beq \bea
\hat v_0^{\text{la}} & \defi \boldsymbol 0, \\
 \hat v_{K+1}^{\text{fi}} &\defi \boldsymbol 1.
\eea \eeq

In full analogy with $\boldsymbol l$, we denote by $\hat{\boldsymbol l}$ the list (vector) of $\hat l$-lengths. \\

Let us now go back to \eqref{estimate2}: with $f_\pi(n, \boldsymbol k, \boldsymbol l, \boldsymbol \sigma,\hat{\boldsymbol l})$ standing for the number of $\hat \pi$-paths which share $k_i$ edges with $\pi$ between the hyperplanes $H_{i-1}$ and $H_{i}$ with prescribed lengths $\boldsymbol l$ (for $\pi$), $\hat{\boldsymbol l}$ (for $\hat \pi$) and with $\boldsymbol \sigma$ directionality/order, it holds
\beq\bea \label{estimate3}
f_{\pi}(n,k)=\sum_{\boldsymbol k} \sum_{\boldsymbol l}  \sum_{\boldsymbol \sigma} \sum_{\hat{\boldsymbol l}}f_\pi(n, \boldsymbol k, \boldsymbol l, \boldsymbol \sigma,\hat{\boldsymbol l}).\\
\eea\eeq
We will now derive a formula for the $f_\pi$-summands on the r.h.s. above in terms of the number of paths satisfying the prescriptions {\it locally}: this requires discriminating between the cases where first and last common edge both lie within the same slab (i.e. between successive H-planes), or in two different slabs. Let $h(i)\defi \min\{a, a \geq i, k_a>0\}$ and $h(i)=K+1$ if $\{a, a\geq i, k_a>0\}$ is empty or $i=K+1$. Finally, $h(0)\defi 0$.
\begin{itemize}
\item {\sf Same slab}. 
\begin{itemize}
\item For $k_{h(i)} \geq 1$, we denote by $ \mathring{f}_\pi(n, \boldsymbol k, \boldsymbol l, \boldsymbol \sigma , \hat{\boldsymbol l},i)$ the number of stretched subpaths sharing $k_{h(i)}$ edges with $\pi$ between $v_{h(i)}^{\text{fi}}$ and  $v_{h(i)}^{\text{la}}$, knowing that first and last edge are in common. 
\end{itemize}
\item {\sf Different slabs.}
\begin{itemize}
\item We denote by $\overline{f}_\pi(n, \boldsymbol k, \boldsymbol l, \boldsymbol \sigma ,\hat{\boldsymbol l},i)$ the number of paths connecting $\hat{v}_{h(i)}^{\text{la}}$ to $\hat{v}_{h(i+1)}^{\text{fi}}$.  
\end{itemize}
\end{itemize}
See below for a graphical rendition:
\begin{figure}[!h]
    \centering
  \includegraphics[scale=0.6]{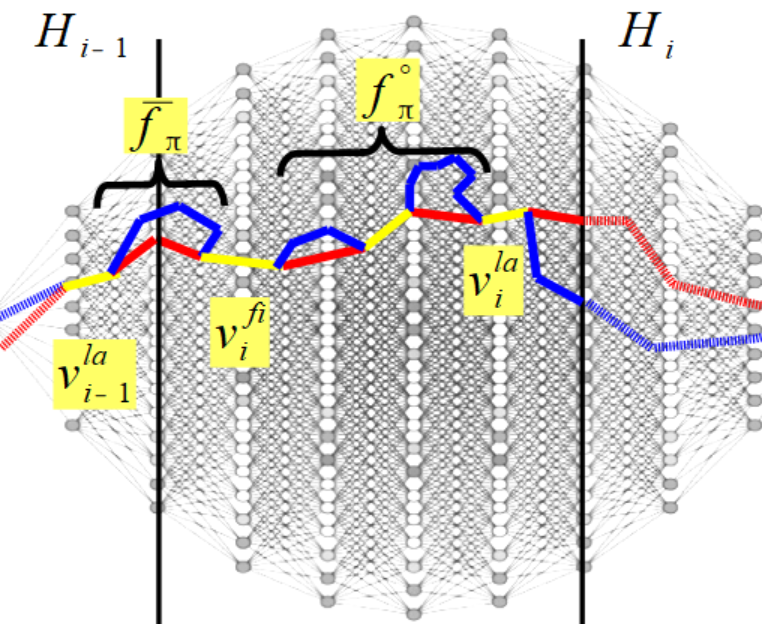}
\caption{The blue and the red paths are admissible paths, which cross the different common edges in yellow.}
\label{combi}
\end{figure}

With these definitions, denoting by $b\defi \#\{i:\; k_i>0\}$, it clearly holds that
\beq \bea \label{stimo}
& f_\pi(n, \boldsymbol k, \boldsymbol l, \boldsymbol \sigma,\hat{\boldsymbol l})  = \prod_{i=1}^b  \mathring{f}_\pi(n, \boldsymbol k, \boldsymbol l, \boldsymbol \sigma , \hat{\boldsymbol l},i)  \prod_{i=0}^{b} \overline{f}_\pi(n, \boldsymbol k, \boldsymbol l, \boldsymbol \sigma, \hat{\boldsymbol l},i)\,.
\eea \eeq
The new goal is to get a handle on the $\mathring{f}_\pi$ and $\overline{f}_\pi$-terms. As for the former,  we claim that for $n$ big enough, for $i, \, k_{h(i)}>0$ and with $\alpha \defi \frac{5}{6}$,
\beq \bea \label{stima_due}
& \mathring{f}_\pi(n, \boldsymbol k, \boldsymbol l, \boldsymbol \sigma , \hat{\boldsymbol l},i)  \leq{\tanh\left(\mathsf E\frac{d(v_{h(i)}^{\text{fi}},v_{h(i)}^{\text{la}})-k_{h(i)}}{\mathsf L_{opt} n}\right)}^{d(v_{h(i)}^{\text{fi}},v_{h(i)}^{\text{la}})-k_{h(i)}} \\
& \hspace{5cm}{ \times \cosh\left(\mathsf E\frac{\left(d(v_{h(i)}^{\text{fi}},v_{h(i)}^{\text{la}}\right)-k_{h(i)})}{\mathsf L_{opt} n}\right)}^{n} \\
& \hspace{7cm} \times {\left(\frac{\mathsf L_{opt} n}{e \mathsf E}\right)}^{d(v_{h(i)}^{\text{fi}},v_{h(i)}^{\text{la}})-k_{h(i)}} n^{n^{\alpha}}n^{\frac{1}{2}}\,.
\eea \eeq
In order to see this, we first observe that substrands are stretched between successive H-planes: the number of subpaths which share $k_{h(i)}\geq 2$ edges with $\pi$ between $v_{h(i)}^{\text{fi}}$ and $v_{h(i)}^{\text{la}}$ therefore equals the number of directed subpaths that share $k_{h(i)}-2$ edges with the subpath of $\pi$ between $v_{h(i)}^{\text{fi}}$ and $v_{h(i)}^{\text{la}}$ on a hypercube of dimension $d\left(v_{h(i)}^{\text{fi}},v_{h(i)}^{\text{la}}\right)-2$. Hence
\beq \bea \label{Fnk}
 \mathring{f}_\pi(n, \boldsymbol k, \boldsymbol l, \boldsymbol \sigma , \hat{\boldsymbol l},i)  \leq F\left(d\left(v_{h(i)}^{\text{fi}},v_{h(i)}^{\text{la}}\right)-2,k_{h(i)}-2\right).
\eea \eeq
Next we note that for $n$ large enough,
\beq 
n^6\leq {{n}\choose{\mathfrak{n_{e}}}},
\eeq
and therefore, by Lemma \ref{path_counting}, the following rough bound holds {\it for all} $k\leq n$:
\beq\label{gles}
F(n,k)\leq (n-k)!{{n}\choose{\mathfrak{n_{e}}}}\,.
\eeq
Using this in  \eqref{Fnk} yields
\beq \label{Fdk}
F\left(d\left(v_{h(i)}^{\text{fi}},v_{h(i)}^{\text{la}}\right)-2,k_{h(i)}-2\right) \leq \left(d\left(v_{h(i)}^{\text{fi}},v_{h(i)}^{\text{la}}\right)-k_{h(i)}\right)!{{n}\choose{\mathfrak{n_{e}}}}\,.
\eeq
Furthermore, 
\beq\label{esterror}
{{n}\choose{\mathfrak{n_{e}}}}\leq \frac{n!}{\mathfrak{n_{e}}!}= \frac{n!}{(n-5e(n+3)^{2/3})!}\leq n^{5e(n+3)^{2/3}}\leq n^{n^{\alpha}}\,,
\eeq
for $n$ big enough, where $\alpha \defi \frac{5}{6}$. Using this in \eqref{Fdk}, and plugging the ensuing estimates in \eqref{Fnk} we obtain
\beq \label{e8_00}
 \mathring{f}_\pi(n, \boldsymbol k, \boldsymbol l, \boldsymbol \sigma , \hat{\boldsymbol l},i) \leq (d(v_{h(i)}^{\text{fi}},v_{h(i)}^{\text{la}})-k_{h(i)})!  n^{n^{\alpha}}\,.
\eeq
By elementary Stirling approximation, 
\beq \bea \label{stirlo}
(d(v_{h(i)}^{\text{fi}},v_{h(i)}^{\text{la}})-k_{h(i)})! & \lesssim \left( d(v_{h(i)}^{\text{fi}},v_{h(i)}^{\text{la}})-k_{h(i)} \right)^{1/2} \left[ \frac{ d(v_{h(i)}^{\text{fi}},v_{h(i)}^{\text{la}})-k_{h(i)}   }{e}\right]^{d(v_{h(i)}^{\text{fi}},v_{h(i)}^{\text{la}})-k_{h(i)}} \\
& \lesssim n^{1/2} \left[ \frac{ d(v_{h(i)}^{\text{fi}},v_{h(i)}^{\text{la}})-k_{h(i)}   }{e}\right]^{d(v_{h(i)}^{\text{fi}},v_{h(i)}^{\text{la}})-k_{h(i)}}, 
\eea \eeq
the last inequality using that the dimension of an hypercube embedded between two hyperplanes is bounded above by their distance, i.e  $d(v_{h(i)}^{\text{fi}},v_{h(i)}^{\text{la}}) \leq \frac{n}{K}<n$. 

Plugging \eqref{stirlo} in \eqref{e8_00} yields
\beq \label{e8_0}
 \mathring{f}_\pi(n, \boldsymbol k, \boldsymbol l, \boldsymbol \sigma , \hat{\boldsymbol l},i) \lesssim 
n^{n^{\alpha}+1/2} \left[ \frac{ d(v_{h(i)}^{\text{fi}},v_{h(i)}^{\text{la}})-k_{h(i)}   }{e}\right]^{d(v_{h(i)}^{\text{fi}},v_{h(i)}^{\text{la}})-k_{h(i)}}\,.
\eeq
The above bound strongly depends on local specifications, which turn out to be rather untractable especially when it comes to the full product \eqref{stimo}. We will circumvent this problem by means of a series of tricks: in a first step we recognize the term involving the $d(v_{h(i)}^{\text{fi}},v_{h(i)}^{\text{la}})$ in \eqref{e8_0} as a constituent part of a Stanley's bound, which we thus introduce artificially. In a second step, we will perform a rather elementary asymptotic analysis of the product \eqref{stimo} which is enabled by some monotonicity properties of the hyperbolic functions. To see how the first step comes about, we note that $\sinh(x) \geq x$ and $\cosh(x) \geq 1$ for all $x>0$, hence the following holds 
\beq \bea
1 \leq \frac{\sinh(y)^d}{y^d} \cosh(y)^{n-d} = \tanh(y)^d \cosh(y)^n \frac{1}{y^d},
\eea \eeq
for any $y>0$ and $d\leq n$. We use this inequality with 
\beq
y := \mathsf E\frac{(d(v_{h(i)}^{\text{fi}},v_{h(i)}^{\text{la}})-k_{h(i)})}{\mathsf L_{opt} n}, \quad d:= d(v_{h(i)}^{\text{fi}},v_{h(i)}^{\text{la}})-k_{h(i)}, 
\eeq
in which case we see that 
\beq \bea \label{boooo}
1 & \leq \tanh\left(\mathsf E\frac{(d(v_{h(i)}^{\text{fi}},v_{h(i)}^{\text{la}})-k_{h(i)})}{\mathsf L_{opt} n}\right) \\
& \hspace{3cm} \times \cosh\left(\mathsf E\frac{(d(v_{h(i)}^{\text{fi}},v_{h(i)}^{\text{la}})-k_{h(i)})}{\mathsf L_{opt} n}\right)^n  \\
&  \hspace{7cm} \times \left[ \frac{\mathsf L_{opt} n}{\mathsf E (d(v_{h(i)}^{\text{fi}},v_{h(i)}^{\text{la}})-k_{h(i)})}\right]^{d(v_{h(i)}^{\text{fi}},v_{h(i)}^{\text{la}})-k_{h(i)}}\,.
\eea \eeq
Artificially upperbounding with the help of this estimate the r.h.s. of \eqref{e8_0}, 
and factoring out the $(d(v_{h(i)}^{\text{fi}},v_{h(i)}^{\text{la}})-k_{h(i)})^{d(v_{h(i)}^{\text{fi}},v_{h(i)}^{\text{la}})-k_{h(i)}}$-terms then yields
\beq\bea\label{e8''}
& \mathring{f}_\pi(n, \boldsymbol k, \boldsymbol l, \boldsymbol \sigma , \hat{\boldsymbol l},i) \lesssim  n^{n^{\alpha}+1/2} {\tanh\left(\mathsf E\frac{d(v_{h(i)}^{\text{fi}},v_{h(i)}^{\text{la}})-k_{h(i)}}{\mathsf L_{opt} n}\right)}^{d(v_{h(i)}^{\text{fi}},v_{h(i)}^{\text{la}})-k_{h(i)}} \times \\
& \hspace{5cm} \times {\cosh\left(\mathsf E\frac{d(v_{h(i)}^{\text{fi}},v_{h(i)}^{\text{la}})-k_{h(i)}}{\mathsf L_{opt} n}\right)}^{n} {\left(\frac{\mathsf L_{opt} n}{e \mathsf E}\right)}^{d(v_{h(i)}^{\text{fi}},v_{h(i)}^{\text{la}})-k_{h(i)}}.
\eea\eeq
Claim \eqref{stima_due} is therefore settled for $k_{h(i)} \geq 2$ and easely holds for $k_{h(i)}=1$.\\

We  now move to estimating the $\overline{f}_\pi(n, \boldsymbol k, \boldsymbol l, \boldsymbol \sigma, \hat{\boldsymbol l},i)$-terms. 
Note that $\boldsymbol l$ fixes the vertices $v_i^{\text{fi}},v_i^{\text{la}}$, and in particular the Hamming distance between two successive commons edges, which are not between the same $H$-planes, $\boldsymbol \sigma$ fixes $\hat v_i^{\text{fi}},\hat v_i^{\text{la}}$, while $\hat{\boldsymbol l}$ gives the length of the subpaths $\hat{\pi}$ between these common edges. \\
For all $i \in \{0\dots K\}$, we set
\beq \bea
\hat{l}_i\defi l_{\hat{\pi}}({\hat{v}_{h(i)}}^{\text{la}},\hat{v}_{h(i+1)}^{\text{fi}}).
\eea \eeq
We claim that
\beq \label{stima_uno}
\overline{f}_\pi(n, \boldsymbol k, \boldsymbol l, \boldsymbol \sigma, \hat{\boldsymbol l},i) \lesssim {\tanh\left( \frac{\mathsf E \hat{l}_i }{\mathsf L_{opt}n}\right)}^{d\left(\hat{v}_{h(i)}^{\text{la}},\hat{v}_{h(i+1)}^{\text{fi}}\right)}{\cosh\left(\frac{ \mathsf E  \hat{l}_i }{\mathsf L_{opt}n}\right)}^{n}{\left(\frac{\mathsf L_{opt}n}{\mathsf E e}\right)}^{\hat{l}_i}n^{\frac{1}{2}}\,.
\eeq
Indeed, it clearly holds that
\beq \label{stima_uno'}
\overline{f}_\pi(n, \boldsymbol k, \boldsymbol l, \boldsymbol \sigma, \hat{\boldsymbol l},i) \leq M_{n,\hat{l_i},d\left(\hat v_{h(i)}^{\text{la}},\hat v_{h(i+1)}^{\text{fi}}\right)}.
\eeq
To get a handle on the r.h.s. above we make use of the following estimate, the derivation of which
follows the by now classical route\footnote{Stanley's bound \eqref{sf} with $x:= \frac{l \mathsf E}{\mathsf L_{opt} n}$ / Stirling approximation / some elementary rearrangements.}, and is thus omitted:
 \beq \label{e7}
M_{n,l,nd} \lesssim n^{\frac{1}{2}} {\tanh\left(\frac{\mathsf E l}{\mathsf L_{opt}n}\right)}^{nd}{\cosh\left( \frac{\mathsf E l }{\mathsf L_{opt}n}\right)}^{n}{\left(\frac{\mathsf L_{opt}n}{\mathsf E e}\right)}^{l}\,.
\eeq
Using \eqref{e7} with $l:= \hat l_i, \, nd:= d(\hat v_{h(i)}^{\text{la}},\hat v_{h(i+1)}^{\text{fi}})$ in  \eqref{stima_uno'} steadily yields the claim \eqref{stima_uno}. \\

Having obtained explicit estimates for the $\mathring{f}_\pi$ and $\overline{f}_\pi$-terms, we need bounds to their products as appearing in \eqref{stimo}. This will be done exploiting the aforementioned monotonicity properties of hyperbolic functions:  for any $y_i,d_i\geq 0$, and $k \in \N$ it holds
\beq\bea\label{e8'''}
 \prod_{i=1}^{k}{\tanh\left(y_i\right)}^{d_i}\leq \prod_{i=1}^{k}{\tanh\left(\sum_{i=1}^k y_i\right)}^{d_i}={\tanh\left(\sum_{i=1}^k y_i\right)}^{\sum_{i=1}^k d_i}\,,
\eea\eeq  
since $\tanh$ is increasing, and
\beq\bea\label{e9}
\prod_{i=1}^{k}{\cosh\left(y_i\right)}\leq \cosh\left(\sum_{i=1}^k y_i\right)\,,
\eea\eeq  
which can be steadily checked iterating $\cosh\left(a+c\right)=\cosh\left(a\right)\cosh\left(c\right)+\sinh\left(a\right)\sinh\left(c\right)\geq \cosh\left(a\right)\cosh\left(c\right),$ for $a,c>0$. 

These bounds allow to remove most of the local dependencies appearing in the products \eqref{stimo}: shortening 
\beq \label{defi_debe}
\mathcal D_b \defi \sum_{i=1}^{b} \left[ d\left(v_{h(i)}^{\text{fi}},v_{h(i)}^{\text{la}}\right)-k_{h(i)} \right],
\eeq
and combining \eqref{e8'''}, \eqref{e9} and  \eqref{stima_due} we get 
\beq\bea\label{prod1}
&  \prod_{i=1}^{b}  \mathring{f}_\pi(n, \boldsymbol k, \boldsymbol l, \boldsymbol \sigma , \hat{\boldsymbol l},i) \lesssim  n^{Kn^{\alpha}+\frac{K}{2}} {\tanh\left( \frac{\mathsf E \mathcal D_b}{\mathsf L_{opt} n}\right)}^{\mathcal D_b} {\cosh\left(\frac{ \mathsf E \mathcal D_b }{\mathsf L_{opt} n}\right)}^{n}
{\left(\frac{\mathsf L_{opt} n}{e \mathsf E}\right)}^{\mathcal D_b}\,.
\eea\eeq
On the other hand, shortening
\beq \label{defi_debe_hat}
\widehat{\mathcal D}_b \defi \sum_{i=0}^{b} d\left(\hat{v}_{h(i)}^{\text{la}},\hat{v}_{h(i+1)}^{\text{fi}}\right)\,, \quad \widehat L_{b} \defi \sum_{i=0}^{b}\hat{l}_i\,,
\eeq
and combining  \eqref{e8'''}, \eqref{e9} with \eqref{stima_uno} we obtain
\beq\bea\label{prod2}
\prod_{i=0}^b \overline{f}_\pi(n, \boldsymbol k, \boldsymbol l, \boldsymbol \sigma, \hat{\boldsymbol l},i)&\lesssim n^{\frac{K+1}{2}} {\tanh\left(\frac{ \mathsf E \widehat L_{b}}{\mathsf L_{opt}n}\right)}^{\widehat{\mathcal D}_b}  {\cosh\left(\frac{\mathsf E \widehat L_{b} }{\mathsf L_{opt}n}\right)}^{n}{\left(\frac{\mathsf L_{opt}n}{\mathsf E e}\right)}^{\widehat L_{b} }.
\eea\eeq
Plugging \eqref{prod1} and \eqref{prod2} in \eqref{stimo} thus leads to 
\beq \bea \label{stimo1}
 f_\pi(n, \boldsymbol k, \boldsymbol l, \boldsymbol \sigma,\hat{\boldsymbol l})  \lesssim 
n^{\frac{2K+1}{2}+ K n^{\alpha}} & {\tanh\left( \frac{\mathsf E \mathcal D_b}{\mathsf L_{opt} n}\right)}^{\mathcal D_b } {\cosh\left(\frac{ \mathsf E \mathcal D_b }{\mathsf L_{opt} n}\right)}^{n}
{\left(\frac{\mathsf L_{opt} n}{ \mathsf E e}  \right)}^{\mathcal D_b} \times \\
& \qquad \times {\tanh\left(\frac{\mathsf E \widehat L_{b} }{\mathsf L_{opt}n}\right)}^{\widehat{\mathcal D}_b}  {\cosh\left(\frac{ \mathsf E \widehat L_{b} }{\mathsf L_{opt}n}\right)}^{n}{\left(\frac{\mathsf L_{opt}n}{\mathsf E e}\right)}^{\widehat L_{b} }\,.
\eea \eeq
The above estimate still involves the product of two $\tanh$-, and two $\cosh$-terms: using once more the monotonicity tricks \eqref{e8'''} and \eqref{e9} we get

\beq \bea \label{stimo_due}
& f_\pi(n, \boldsymbol k, \boldsymbol l, \boldsymbol \sigma,\hat{\boldsymbol l})  \lesssim \\
& \hspace{1cm}  n^{\frac{2K+1}{2}+ K n^{\alpha}} {\tanh\left(\mathsf E \frac{\mathcal D_b+\widehat L_b }{\mathsf L_{opt}n}\right)}^{ \mathcal D_b+ \widehat{\mathcal D}_b}  {\cosh\left(\mathsf E \frac{\mathcal D_b+\widehat L_b }{\mathsf L_{opt}n} \right)}^{n} {\left(\frac{\mathsf L_{opt} n}{\mathsf Ee}\right)}^{\mathcal D_b+\widehat L_b }\,.
\eea \eeq
But paths in $ \mathcal J$ have the same, prescribed length, and it holds that
\beq\label{a)''}
\mathcal D_b+ \widehat L_b =\mathsf L_{opt} n-k.
\eeq
Using this, \eqref{stimo_due} simplifies to 
\beq \bea \label{stimo_tre}
& f_\pi(n, \boldsymbol k, \boldsymbol l, \boldsymbol \sigma,\hat{\boldsymbol l})  \lesssim \\
& \hspace{1cm} n^{\frac{2K+1}{2}+ K n^{\alpha}} {\tanh\left(\mathsf E \frac{ \mathsf L_{opt} n-k}{\mathsf L_{opt} n} \right)}^{\mathcal D_b+ \widehat{\mathcal D}_b} {\cosh\left(\mathsf E \frac{\mathsf L_{opt} n-k}{\mathsf L_{opt} n} \right)}^{n}   {\left(\frac{\mathsf L_{opt} n}{\mathsf Ee}\right)}^{\mathsf L_{opt} n-k} \\
\eea \eeq

Remark, in particular, that the r.h.s. above depends on the  local prescriptions {\it only through the $\tanh$-exponent}. It will come hardly as a surprise that this feature leads to a dramatic simplification of the computations. As a matter of fact, even the exponent depends only very mildly on the local prescriptions: indeed,  we claim that 
\begin{lem} \label{distance}
\beq\bea
\mathcal D_b + \widehat{\mathcal D}_b  \geq \max\left(n-k,\frac{\mathsf L_{opt} n-k}{4}\right).
\eea\eeq
\end{lem}
Proving this claim will unfortunately require a fair amount of work, so we assume its validity for the time being.

By monotonicity, 
\beq
\tanh\left(\mathsf E \frac{ \mathsf L_{opt} n-k}{\mathsf L_{opt} n} \right) \leq \tanh\left(\mathsf E \right) = \frac{1}{\sqrt{2}} < 1,
\eeq
hence Lemma \ref{distance} applied to  \eqref{stimo_tre} yields the upperbound
\beq \bea \label{stimo_tree}
& f_\pi(n, \boldsymbol k, \boldsymbol l, \boldsymbol \sigma,\hat{\boldsymbol l})  \lesssim \\
& \hspace{1cm} n^{\frac{2K+1}{2}+ K n^{\alpha}} {\tanh\left(\mathsf E \frac{ \mathsf L_{opt} n-k}{\mathsf L_{opt} n} \right)}^{ \max\left(n-k,\frac{\mathsf L_{opt} n-k}{4}\right)} {\cosh\left(\mathsf E \frac{\mathsf L_{opt} n-k}{\mathsf L_{opt} n} \right)}^{n}   {\left(\frac{\mathsf L_{opt} n}{\mathsf Ee}\right)}^{\mathsf L_{opt} n-k} \,,
\eea \eeq
{\it no longer depends on} $\boldsymbol l, \boldsymbol \sigma, \hat{\boldsymbol l}, \boldsymbol k$; plugging this in \eqref{stimo}, and the ensuing estimate in \eqref{estimate3} therefore leads to
\beq\bea\label{a)}
f_{\pi}(n,k)\lesssim  n^{\frac{2K+1}{2}+ K n^{\alpha}} \sum_{\boldsymbol l} \sum_{\boldsymbol \sigma} \sum_{\hat{\boldsymbol l}} \sum_{\boldsymbol k} \mathfrak T(n,k)\,,
\eea\eeq
where
\beq
\mathfrak T(n,k) \defi  {\tanh\left(\frac{\mathsf E(\mathsf L_{opt} n-k)}{\mathsf L_{opt} n}\right)}^{\max\left(n-k,\frac{\mathsf L_{opt} n-k}{4}\right)} {\cosh\left(\frac{ \mathsf E(\mathsf L_{opt} n-k)}{\mathsf L_{opt} n}\right)}^{n}{\left(\frac{\mathsf L_{opt} n}{\mathsf Ee}\right)}^{\mathsf L_{opt} n-k} \,.
\eeq
Since $\mathfrak T(n,k)$ depends on the number of common edges, but not on the local prescriptions, we thus only need  estimates on the cardinalities of the sums appearing in \eqref{a)}. As for the first sum, since $v_\cdot^{\text{fi}}$ can only move along the path $\pi$ between two successive hyperplanes, the number of ways to place such $v_\cdot^{\text{fi}}$'s is {\it at most} $n$ (the same of course holds true for $v_\cdot^{\text{la}}$), hence
\beq\bea\label{somme2}
\sum_{\boldsymbol l}  \leq n^{2K},
\eea\eeq
and by analogous reasoning 
\beq\bea\label{somme4}
\sum_{\boldsymbol l'} \leq n^{2K}.
\eea\eeq
Moreover, it clearly holds that 
\beq\bea\label{somme3}
\sum_{\boldsymbol \sigma} \leq 2^{K}.
\eea\eeq
Finally, 
\beq\bea\label{tiroir}
\sum_{\boldsymbol k}  = &\sum_{\substack{ k_i,\\ k_1+k_2+...+k_K=k}}  =\binom{k+K-1}{K-1}\lesssim \frac{{(k+K-1)}^{k+K-1}}{{(K-1)}^{K-1}{k}^{k}}, \\
\eea\eeq
by Stirling approximation. Since $(K-1)^{K-1}\geq 1$, and $\log(1+x)\leq x$, we see that
\beq\bea\label{somme1}
\eqref{tiroir}& \leq {k}^{K-1}{\left(1+\frac{K-1}{k}\right)}^{k+K-1} = k^{K-1} \exp\left[ \left(k+K+1 \right) \log\left( 1+ \frac{K-1}{k} \right)\right] \\
& \leq{k}^{K-1}\exp\left[\left(k+K-1\right)\frac{K-1}{k} \right] \leq {k}^{K-1}{\exp{K\left(K-1\right)}}.
\eea\eeq
Combining \eqref{a)}, \eqref{somme2},  \eqref{somme4}, \eqref{somme3} and \eqref{somme1}, we obtain
\beq\bea\label{estimate2'}
f_{\pi}(n,k)\leq P_n n^{K n^{\alpha}}{\tanh\left(\mathsf E\frac{\mathsf L_{opt} n-k}{\mathsf L_{opt} n}\right)}^{\max\left(n-k,\frac{\mathsf L_{opt} n-k}{4}\right)} {\cosh\left(\mathsf E\frac{ \mathsf L_{opt} n-k}{\mathsf L_{opt} n}\right)}^{n}{\left(\frac{\mathsf L_{opt} n}{\mathsf Ee}\right)}^{\mathsf L_{opt} n-k},
\eea\eeq
where $P_n$ is a finite degree polynomial, which is indeed the claim of Lemma  \ref{r2}.
\end{proof}

\begin{proof}[Proof of Lemma \ref{distance}]. Recall that the claim reads 
\beq
\mathcal D_b+ \widehat{\mathcal D}_b  \geq \max\left( n-k, \frac{\mathsf L_{opt} n - k}{4} \right)\,.
\eeq
The validity of the first inequality, to wit
\beq \label{claim_facile}
\mathcal D_b + \widehat{\mathcal D}_b \geq n-k,
\eeq
relies on a self-evident fact, namely that the total distance of shared edges in the directed case is a lower bound for the undirected case. More precisely, since common edges  contribute to the number of steps performed while connecting $\boldsymbol 0$ to $\boldsymbol 1$, as soon as a backstep acts on a shared edge, the total distance between shared edges is bound to increase: the path has eventually to make up for the "lost ground". Another way to put it: the contribution $\mathcal D_b + \widehat{\mathcal D}_b$ is smallest when {\it all} shared edges are steps forward, in which case the total distance between these edges must be {\it at least} the minimal number of steps required to connect $\boldsymbol 0$ to $\boldsymbol 1$. Since this minimal number is clearly the dimension minus the number of shared (prescribed) edges, i.e. $n-k$, \eqref{claim_facile} is settled. \\

\noindent The second inequality
\beq \label{claim_difficile}
\mathcal D_b + \widehat{\mathcal D}_b \geq \frac{\mathsf L_{opt} n - k}{4},
\eeq
requires more work and depends on some key properties of  paths in $\mathcal J$. We begin with a couple of observations: 
\begin{itemize}
\item[i)] First we note that 
\[
d\left(v_{h(i)}^{\text{fi}},v_{h(i)}^{\text{la}}\right) = d\left(\hat v_{h(i)}^{\text{fi}}, \hat v_{h(i)}^{\text{la}}\right)\,,
\]
since inverting directionality clearly has no impact on the distance.
\item[ii)] Furthemore, in a (fully) stretched phase distance and length do, in fact, coincide:
\[ 
d\left(\hat v_{h(i)}^{\text{fi}}, \hat v_{h(i)}^{\text{la}}\right) = l_{\hat \pi}\left( \hat v_{h(i)}^{\text{fi}}, \hat v_{h(i)}^{\text{la}}\right)\,.
\]
\item[iii)] Finally, and by definition, 
\[
\sum_{i=1}^b k_{h(i)} = k\,.
\]
\end{itemize}
Plugging items i-iii) above in the $\mathcal D_b$-definition \eqref{defi_debe} yields
\beq \bea \label{tunnel}
\mathcal D_b+ \widehat{\mathcal D}_b & = 
\sum_{i=1}^{b} l_{\hat \pi} \left( \hat v_{h(i)}^{\text{fi}}, \hat v_{h(i)}^{\text{la}}\right)- k+ \sum_{i=0}^{b} d\left(\hat{v}_{h(i)}^{\text{la}},\hat{v}_{h(i+1)}^{\text{fi}}\right)\,.
\eea \eeq
We now claim that for $i \in \{0,1, \dots b\}$ , it holds: 
\beq \label{weeeird}
d\left(\hat{v}_{h(i)}^{\text{la}},\hat{v}_{h(i+1)}^{\text{fi}}\right) \geq \frac{1}{4} l_{\hat \pi}(\hat  v_{h(i)}^{\text{la}}, \hat v_{h(i+1)}^{\text{fi}}) \,.
\eeq
This is, in fact, our key technical claim, but since its proof requires some involved analysis, we assume its validity for the time being, and first show how it implies \eqref{claim_difficile}:  plugging \eqref{weeeird} in \eqref{tunnel} we obtain
\beq \bea \label{tunnel_two}
\mathcal D_b+ \widehat{\mathcal D}_b  & \geq \sum_{i=1}^{b} l_{\hat \pi} \left( \hat v_{h(i)}^{\text{fi}}, \hat v_{h(i)}^{\text{la}}\right)- k+ \frac{1}{4} \sum_{i=0}^{b}  l_{\hat \pi}(\hat  v_{h(i)}^{\text{la}}, \hat v_{h(i+1)}^{\text{fi}})\,. 
\eea \eeq
But by construction, 
\beq
l_{\hat \pi} \left( \hat v_{h(i)}^{\text{fi}}, \hat v_{h(i)}^{\text{la}}\right) \geq k_{h(i)},
\eeq
hence
\beq
\sum_{i=1}^{b} l_{\hat \pi} \left( \hat v_{h(i)}^{\text{fi}}, \hat v_{h(i)}^{\text{la}}\right)- k \geq \sum_{i=1}^{b} k_{h(i)} - k \geq 0,
\eeq
the last inequality by item iii) above. This positivity implies, in particular, that
\beq
\sum_{i=1}^{b} l_{\hat \pi} \left( \hat v_{h(i)}^{\text{fi}}, \hat v_{h(i)}^{\text{la}}\right)- k \geq \frac{1}{4} \left( \sum_{i=1}^{b} l_{\hat \pi} \left( \hat v_{h(i)}^{\text{fi}}, \hat v_{h(i)}^{\text{la}}\right)- k \right)\,,
\eeq
and using this lower bound in \eqref{tunnel_two} then yields 
\beq \bea
\mathcal D_b+ \widehat{\mathcal D}_b  & \geq 
\frac{1}{4} \left( \sum_{i=1}^{b} l_{\hat \pi} \left( \hat v_{h(i)}^{\text{fi}}, \hat v_{h(i)}^{\text{la}}\right)- k \right)+ \frac{1}{4} \sum_{i=0}^{b}  l_{\hat \pi}(\hat  v_{h(i)}^{\text{la}}, \hat v_{h(i+1)}^{\text{fi}})\\
& = \frac{1}{4} \left(  \underbrace{ \sum_{i=1}^{b} l_{\hat \pi} \left( \hat v_{h(i)}^{\text{fi}}, \hat v_{h(i)}^{\text{la}}\right) + \sum_{i=0}^{b}  l_{\hat \pi}(\hat  v_{h(i)}^{\text{la}}, \hat v_{h(i+1)}^{\text{fi}})}_{= \mathsf L_{opt}n }    -k   \right),
\eea \eeq
which settles our key claim \eqref{claim_difficile}. \\

It thus remains to prove \eqref{weeeird}. Recall that we are considering the situation where shared edges are separated by (at least) one $H$-plane\footnote{as otherwise the claim would be trivial anyhow: if the shared edges lie within two successive $H$-planes, the polymer is in a stretched phase in which case distance (d) and length (l) coincide, with the inequality \eqref{weeeird} thus trivially holding.}.  Since by definition an $H$-plane is also an $H'$-plane, prescribing the number of separating $H'$-planes allows to discriminate among different scenarios. Indeed, introducing, for $ i=0 \dots b$,
\beq \bea
c_{\hat \pi}(i)\defi & \quad \text{\sf number of $H'$-planes which lie between} \; \hat v_{h(i)}^{\text{la}} \; \text{\sf and} \; \hat v_{h(i+1)}^{\text{fi}}\,, 
\eea \eeq
a minute's thought suggests that there are three scenarios which are "structurally" manifestly different:
\begin{itemize}
\item $c_{\hat \pi}(i) > 2$: the common edges are separated by at least one $H$-plane, and multiple $H'$-planes. We will refer to this as the {\sf H'HH'-case}. 

\item $c_{\hat \pi}(i)=2$: in this case the common edges are separated by one $H$-plane, and one $H'$-plane (which is however not an $H$-plane). We will refer to this as the {\sf HH'-case}. 

\item $c_{\hat \pi}(i)=1$: the separating hyperplane must be an $H$-plane: we will refer to this as the {\sf H-case}. 
\end{itemize}

We will establish the validity of \eqref{weeeird} in all three possible scenarios. We anticipate that \eqref{weeeird} becomes more delicate the less hyperplanes are separating the common edges: this is due to the fact that the larger the number of separating hyperplanes the further apart (in terms of Hamming distance $d$) the common edges must lie, a feature which renders \eqref{weeeird} all the more likely. In line with this observation, the $c_{\hat \pi}(i) =1$ will turn out to be the most delicate. We emphasize that the index $i$ is given and fixed. To lighten notation we will thus omit it in the expressions, whenever no confusion can possibly arise. \\

A number of insights are common to the  treatment of all three scenarios. Given the nature of the inequality we are aiming to prove, it will not come as a surprise that we will need a good control - in the form of {\it lower bounds} - on the distance of two common edges, as well as a good control - this time around in the form of {\it upper bounds} - on the length of the substrands connecting the shared edges. \\

A reasonably tight, but what's more: valid for any of the three $c_{\hat \pi}$-scenarios, lower bound for the distance is provided by technical input {\bf (T1)} below. Let  $H_{\text{fi}}'$  be the first hyperplane on the right of $\hat v_{h(i)}^{\text{la}}$ and $H_{\text{la}}'$ be the last hyperplane on the left of  $\hat v_{h(i+1)}^{\text{fi}}$, and shorten $d_{\hat \pi}^{\text{fi}} \defi d(\hat v_{h(i)}^{\text{la}}, H_{\text{fi}}')$, and  $d_{\hat \pi}^{\text{la}} \defi d(H_{\text{la}}', \hat v_{h(i+1)}^{\text{fi}} )$.  A graphical depiction of this is given in Figure \ref{c(i)} below.
\begin{figure}[!h]
    \centering
  \includegraphics[scale=0.7]{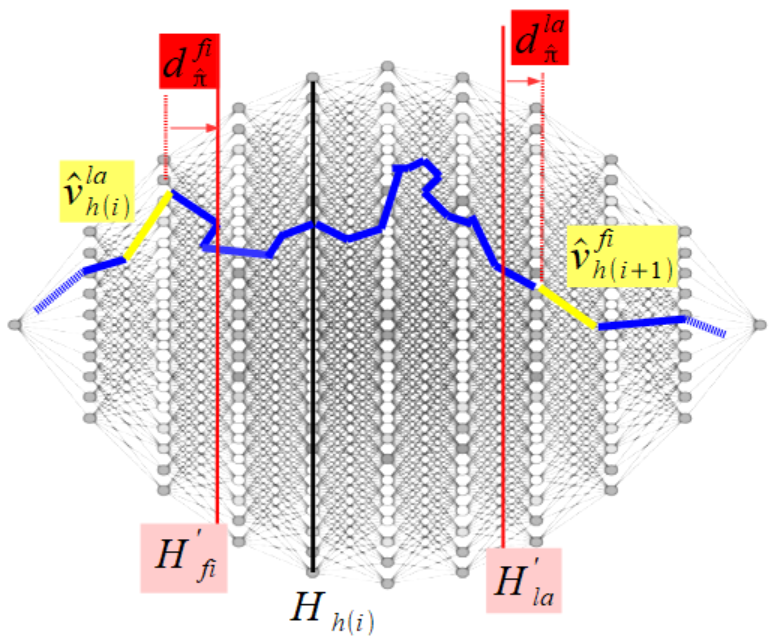}
\caption{$\hat v^{\text{la}}_{h(i)}$ and $\hat v^{\text{fi}}_{h(i+1)}$ separated by three hyperplanes.}
\label{c(i)}
\end{figure} 
The following estimate holds by definition/construction\footnote{it can also immediately evinced from Figure \ref{c(i)}.}
\[
d\left(\hat v_{h(i)}^{\text{la}}, \hat v_{h(i+1)}^{\text{fi}} \right) \geq d_{\hat \pi}^{\text{fi}}+\frac{c_{\hat \pi}(i)-1}{KK'}n+d_{\hat \pi}^{\text{la}} \hspace{2cm} \text{\bf (T1)}\,.
\]
(We note in passing that equality holds if and only if $\hat v_{h(i)}^{\text{la}}, \hat v_{h(i+1)}^{\text{fi}}$ are connected by a directed substrand; since a stretched substrand may have to perform backsteps while connecting these two vertices, {\bf (T1)} is in general only a lower bound). \\

As mentioned, the second technical input, {\bf (T2)} below, concerns {\it upperbounds} on the length of a substrand connecting $H'$-planes.  To see how these come about, let us denote by $\boldsymbol v \in H_{i,j}', \boldsymbol w \in H_{i,j+1}'$  {\it the} vertices  by which the $\hat \pi$-substrand connects the finer mesh. It is important to observe that in virtue of \eqref{defi_pkkprime}, there is no absolutely no ambiguity in the way we identify these vertices: in fact, 
\beq \bea \label{noamb}
& \text{\sf these vertices are unequivocally identified through the \emph{length} of} \\
& \qquad \text{\sf the substrand connecting the successive} \;  H'\;  \text{\sf -planes.}
\eea \eeq
We now claim that 
\[
l_{\hat \pi}(\boldsymbol v,\boldsymbol w) \leq \frac{1.46}{KK'}n \hspace{2cm} \text{\bf (T2)}
\]
The proof of {\bf (T2)} is rather immediate: first recall that in virtue of  \eqref{defi_pkkprime}, 
\beq\label{remf_0}
l_{\hat \pi}(\boldsymbol v,\boldsymbol w)=\left(\mathsf{ef}_i+\mathsf{eb}_i\right) \frac{n}{K'}=\left( \frac{1}{K}+2\mathsf{eb}_i\right) \frac{n}{K'},
\eeq 
the second equality by \eqref{b-f}.  But by \eqref{e_1} (and again  \eqref{b-f}), the number of effective backsteps between $H$-planes in the stretched phase satisfies
\beq \label{bs_2'}
\mathsf{eb}_i=\sinh(\overline{\ma}_{i-1} \mathsf E)\sinh(\ma_{i} \mathsf E)\sinh(\underline{\ma}_i \mathsf E),
\eeq
and by \eqref{t_sinh}, 
\beq \label{t_sinh'}
\sinh(\ma_i \mathsf E)\leq \frac{1}{K}+ \frac{1}{6 K^3}\,,
\eeq 
which combined with \eqref{bs_2'} yields
\beq\bea\label{b2'}
\mathsf{eb}_i &\leq \sinh(\overline{\ma}_i \mathsf E)\sinh(\underline{\ma}_i \mathsf E)  \left(\frac{1}{K}+ \frac{1}{6 K^3}\right) \\
&\leq \sinh\left( \frac{\mathsf E}{2} \right)^2  \left(\frac{1}{K}+ \frac{1}{6 K^3}\right). \\
\eea\eeq
the second inequality by \eqref{m_sinh}.  Since 
\beq
\sinh\left( \frac{\mathsf E}{2} \right)^2 \stackrel{\eqref{easy_est}} = \frac{\sqrt{2}-1}{2} \leq 0.22,
\eeq
and using that $K>10^7$, one plainly checks that
\beq\bea\label{b4_two}
\mathsf{eb}_i \leq 0.23 \times  \frac{1}{K}.\,
\eea\eeq
Plugging \eqref{b4_two} in \eqref{remf_0} settles {\bf (T2)}. \\

If it's true that there is no ambiguity in the way {\it vertices} on the $H'$-plane are identified (recall remark \eqref{noamb} above), it is nonetheless true there there is a certain amount of uncertainty in the way the polymer {\it connects} these planes. This is due to the fact that (contrary to the $H$-planes) the $H'$-planes are not repulsive, hence a polymer might cross them multiple times. Such excursions increase of course the length of the substrand, and introduce some "fuzziness" into the picture. Notwithstanding, we claim that 
\begin{center}
{\sf during one such excursion a polymer can overshoot,\\
in terms of Hamming distance, an $H'$-plane by at most  \\
$\frac{0.23}{KK'}n$ units.} \\
 \hspace{9cm} {\bf (T3)}
\end{center}
 Figure \ref{Confusionzone} below provides an elementary proof of this fact. 
\begin{figure}[!h]
    \centering
  \includegraphics[scale=0.4]{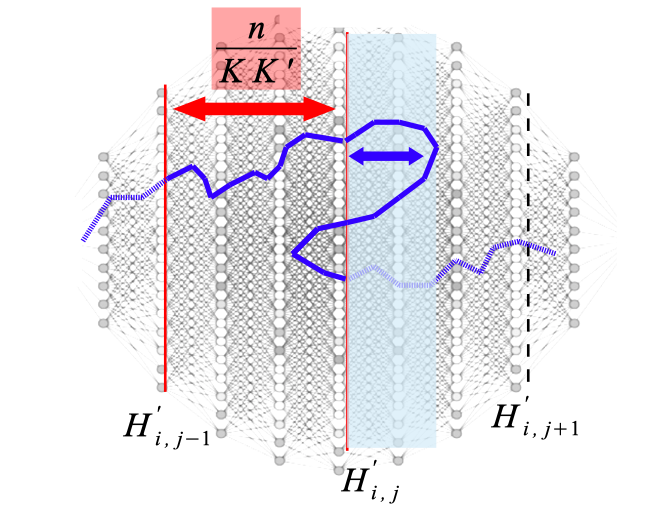}
\caption{The proof of {\bf (T3)} relies on two observations: {\it i)} By {\bf (T2)}, the length of the path connecting first and second $H'$-planes (the continuous blue strand)
is at most $\frac{1.46}{KK'}n$.  {\it ii)} By construction, the Hamming distance of these planes is $\frac{n}{KK'}$. Taking into account that the polymer must return to the second $H'$-plane, we see that the blue arrow is at most half the difference of these quantities, indeed $\frac{0.23}{KK'}n$, as claimed by {\bf (T3)}. (Remark that this case corresponds to a worst-case scenario: the polymer performs first all available forward steps, and only then all availble backsteps). 
}
\label{Confusionzone}
\end{figure} 

The above insight, captured by {\bf (T3)}, suggests to introduce the following set
\beq
{\mathfrak F}_{i,j} \defi \left\{\boldsymbol v \in V_n, \, d(v,H_{i,j}')\leq \frac{0.23}{KK'} n\right\}\,.
\eeq
We emphasize that whenever a common edge lies in this set, it can be crossed by a substrand which {\it either} connects $H_{i,j-1}'$ with $H_{i,j}'$ {\it or} $H_{i,j}'$ with $H_{i,j+1}'$: for this reason, we refer to ${\mathfrak F}_{i,j}$ (which is nothing but "twice" the blue-shaded region in Figure \ref{Confusionzone}) as the  {\sf fuzzy zone}.\\ 

We now record two useful consequences of {\bf (T2)} and {\bf {(T3)}} on the lengths of substrand which will play a role in the proof of \eqref{weeeird}. For reasons which will become clear, we will only need to consider the case where the first common edge lies in the {fuzzy zone} of the $H'-$plane which is on the left of $H_{fi}'$, and/or the other common edge lies on the right of $H_{la}'$. There are two cases: either shared edges lie outside the fuzzy zone, {\sf OuF} for short, or inside, {\sf InF}. \\

\begin{itemize}
\item[{\sf (InF)}] Remark that $\hat v_{h(i)}^{\text{la}}$ being in a fuzzy zone is equivalent to $d_{\hat \pi}^{\text{fi}}\geq \frac{0.77}{KK'}n$. Analogously,  $\hat v_{h(i+1)}^{\text{fi}}$ is in a {fuzzy zone} if and only if $d_{\hat \pi}^{\text{la}}\geq \frac{0.77}{KK'}n$. Furthermore, a path crossing $\hat v_{h(i)}^{\text{la}}$ (or $\hat v_{h(i+1)}^{\text{fi}}$) can cross multiple $H'$-planes besides that to which this vertex belongs: by  {\bf (T3)},
this phenomenon can contribute to the length of the substrand at most $\frac{0.46}{KK'}n$ units.\\
\item[{\sf (OuF)}] If neither $\hat v_{h(i)}^{\text{la}}$ nor $v_{h(i+1)}^{\text{fi}}$ are in a {fuzzy zone}, by {\bf (T2)}, the connecting substrands satisfy 
\[
l_{\hat \pi}(\hat v_{h(i)}^{\text{la}},\hat v_{h(i+1)}^{\text{fi}}) \leq \frac{\left(c_{\hat \pi(i)}+1\right)1.46}{KK'}n \,.
\]
\end{itemize}

We can finally move to the proof of \eqref{weeeird}: this will be done via case-by-case analysis of the three possible $c_{\hat \pi}$-scenarios. \\

\noindent \underline{\sf The H'HH'-case.} \\

\begin{figure}[!h]
    \centering
\includegraphics[scale=0.7]{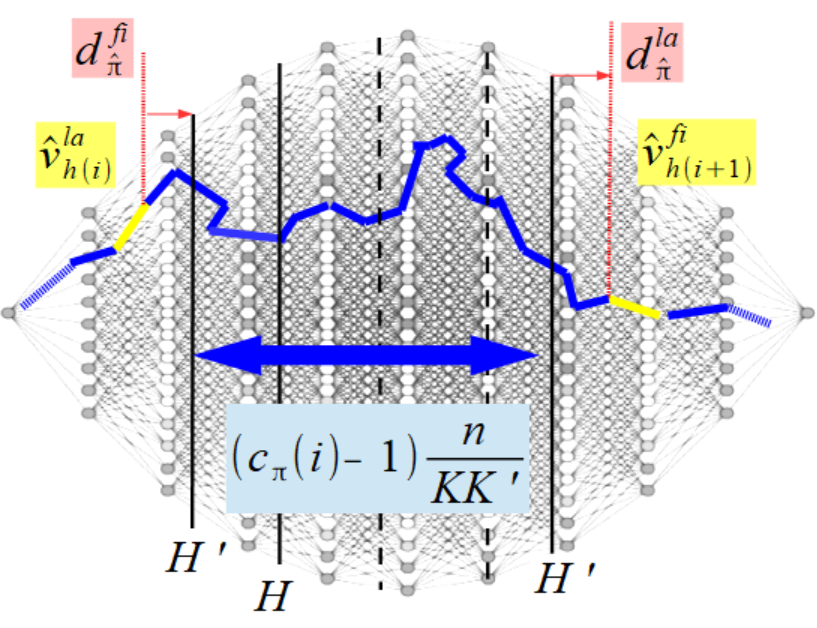}
    \caption{$c(i)\geq 3$: at least three hyperplanes, i.e. at least two $H'$ and one $H$, separating the common edges.}
\label{case_tre}
\end{figure}

\noindent This case is graphically summarized in Figure \ref{case_tre} below: combining {\sf (OuF)} and {\sf (InF)}, we immediately evince from this picture that 
\beq\bea\label{estiml}
l_{\hat \pi}(\hat v_{h(i)}^{\text{la}},\hat v_{h(i+1)}^{\text{fi}}) &\leq \frac{\left(c_{\hat \pi}(i)+1\right)1.46}{KK'}n+\frac{0.46}{KK'}n\left(1_{{d_{\hat \pi}^{\text{fi}}\geq \frac{0.77n}{KK'}}}+1_{d_{\hat \pi}^{\text{la}}\geq \frac{0.77n}{KK'}}\right)\,.
\eea \eeq
The {\sf H'HH'}-scenario at hand is characterized by $c_{\hat \pi}(i)>2$, in which case  the following inequality is immediate: 
\beq 
\frac{\left(c_{\hat \pi}(i)+1\right)1.46}{KK'} \leq \frac{4(c_{\hat \pi}(i)-1)}{KK'}.
\eeq
Using this in \eqref{estiml} we obtain
\beq \label{estiml2}
l_{\hat \pi}(\hat v_{h(i)}^{\text{la}},\hat v_{h(i+1)}^{\text{fi}}) \leq \frac{4(c_{\hat \pi}(i)-1)}{KK'} +\frac{0.46}{KK'}n\left(1_{{d_{\hat \pi}^{\text{fi}}\geq \frac{0.77n}{KK'}}}+1_{d_{\hat \pi}^{\text{la}}\geq \frac{0.77n}{KK'}}\right)\,.
\eeq
Concerning the last two terms on the r.h.s. above, we first observe that obviously
\beq \label{supermagic}
d \geq \frac{0.77}{KK'}n \Longrightarrow 4 d  \geq  \frac{0.46}{KK'}n\,,
\eeq 
hence
\beq
 \frac{0.46}{KK'}n 1_{d_{\hat \pi}^{\text{fi}}\geq \frac{0.77}{KK'}n} \leq 4d_{\hat \pi}^{\text{fi}}(\hat \pi), \qquad 
\frac{0.46}{KK'}n 1_{d_{\hat \pi}^{\text{la}}\geq \frac{0.77}{KK'}n} \leq 4d_{\hat \pi}^{\text{la}}(\hat \pi).
\eeq
Plugging this in \eqref{estiml2} yields 
\beq \bea 
l_{\hat \pi}(\hat v_{h(i)}^{\text{la}},\hat v_{h(i+1)}^{\text{fi}}) & \leq  \frac{4(c_{\hat \pi(i)}-1)}{KK'}n+4d_{\hat \pi}^{\text{fi}}(\hat \pi)+4d_{\hat \pi}^{\text{la}}(\hat \pi)\\
& \leq 4d\left(\hat v_{h(i)}^{\text{la}}, \hat v_{h(i+1)}^{\text{fi}} \right),
\eea \eeq
the last step by {\bf (T1)}. Claim \eqref{weeeird} is therefore settled for the {\sf H'HH'}-case.\\

 \noindent \underline{\sf The HH'-case.} \\

\begin{figure}[!h]
    \centering
\includegraphics[scale=0.7]{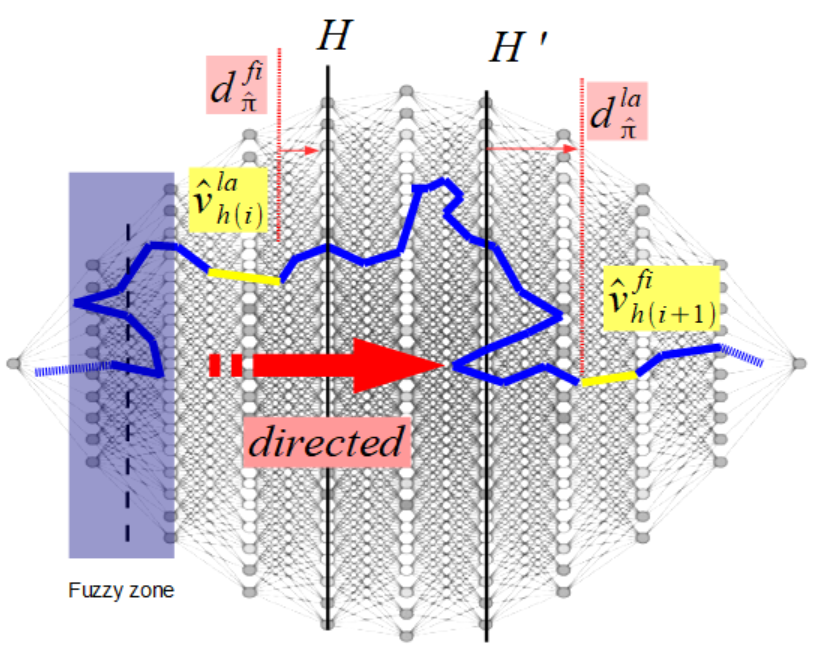}
    \caption{The common edges are separated by $c_{\hat \pi}(i)=2$.}
\label{case_due}
\end{figure}

In this case, see Figure \ref{case_due} below for a graphical rendition, a subpath connecting $\hat v_{h(i)}^{\text{la}}$ and $\hat v_{h(i+1)}^{\text{fi}}$, crosses $c_{\hat \pi(i)} = 2$ many $H'$-planes, one of which is also an $H$-plane. Without loss of generality, we assume that $H_{\text{fi}}'$ is the $H$-plane. We will here distinguish two subcases: $d_{\hat \pi}^{\text{fi}} \geq \frac{0.77}{KK'}n$, and its complement. It holds: \\
\begin{itemize}
\item If  $d_{\hat \pi}^{\text{fi}} \geq \frac{0.77}{KK'}n$,  i.e. the vertex $\hat v_{h(i)}^{\text{la}}$ is in the fuzzy zone, it follows from {\sf (OuF)} and {\sf (InF)} (cfr. also with Figure \ref{case_due}) that 
\beq\bea\label{l2}
l_{\hat \pi}(\hat v_{h(i)}^{\text{la}},\hat v_{h(i+1)}^{\text{fi}})&\leq \frac{3 \times 1.46}{KK'}n+\frac{0.46}{KK'}n\left(1_{d_{\hat \pi}^{\text{fi}}\geq \frac{0.77n}{KK'}}+1_{d_{\hat \pi}^{\text{la}}\geq \frac{0.77n}{KK'}}\right),\\
&=  \frac{4.38}{KK'}n+\frac{0.46}{KK'}n1+\frac{0.46}{KK'}n 1_{d_{\hat \pi}^{\text{la}}\geq \frac{0.77n}{KK'}}\\
& \stackrel{\eqref{supermagic}}{\leq} \frac{4.84}{KK'}n+4d_{\hat \pi}^{\text{la}}\\
& \leq \frac{4}{KK'}n+4\frac{0.77 }{KK'}n+4d_{\hat \pi}^{\text{la}}\\
&\stackrel{\bf{(T1)}}{\leq} 4d(\hat v_{h(i)}^{\text{la}}, \hat v_{h(i+1)}^{\text{fi}})\,.
\eea\eeq
\item If $d_{\hat \pi}^{\text{fi}}<\frac{0.77}{KK'}n$, the vertex $\hat v_{h(i)}^{\text{la}}$ is no longer in the fuzzy zone. However, and crucially, the "complement" of the fuzzy zone is necessarily the repulsive phase, cfr. Figure \ref{case_due} below. This in particular implies that the substrand will connect $\hat v_{h(i)}^{\text{la}}$ with the H-plane in a directed fashion, and therefore 
\beq\bea\label{l1}
l_{\hat \pi}(\hat v_{h(i)}^{\text{la}},\hat v_{h(i+1)}^{\text{fi}})&=l_{\hat \pi}(\hat v_{h(i)}^{\text{la}},H_{h(i)}\cap \hat \pi)+l_{\hat \pi}(H_{h(i)}\cap \hat \pi,\hat v_{h(i+1)}^{\text{fi}})\\
&= d_{\hat \pi}^{\text{fi}}+l_{\hat \pi}(H_{h(i)}\cap \hat \pi,\hat v_{h(i+1)}^{\text{fi}}),\\
\eea\eeq
As before, we estimate the last term on the r.h.s. above by {\sf OuF} and {\sf InF}. Here is the upshot:
\beq\bea
l_{\hat \pi}(\hat v_{h(i)}^{\text{la}},\hat v_{h(i+1)}^{\text{fi}})&\leq d_{\hat \pi}^{\text{fi}}+ \frac{2 \times 1.46}{KK'}n+\frac{0.46}{KK'}n1_{d_{\hat \pi}^{\text{la}}\geq\frac{0.77}{KK'}n}\\
&\stackrel{\eqref{supermagic}}{\leq}  d_{\hat \pi}^{\text{fi}}+\frac{2.92}{KK'}n+4d_{\hat \pi}^{\text{la}}\\
&\leq 4\frac{0.77}{KK'}n+\frac{4}{KK'}n+4d_{\pi}^{\text{la}}\\
&\stackrel{\bf{(T_1)}}{\leq} 4d(\hat v_{h(i)}^{\text{la}}, \hat v_{h(i+1)}^{\text{fi}}).
\eea\eeq
\end{itemize}
The claim \eqref{weeeird} is thus settled for the {\sf HH'}-case. \\

\noindent \underline{\sf The H-case.} 

\begin{figure}[!h]
    \centering
\includegraphics[scale=0.7]{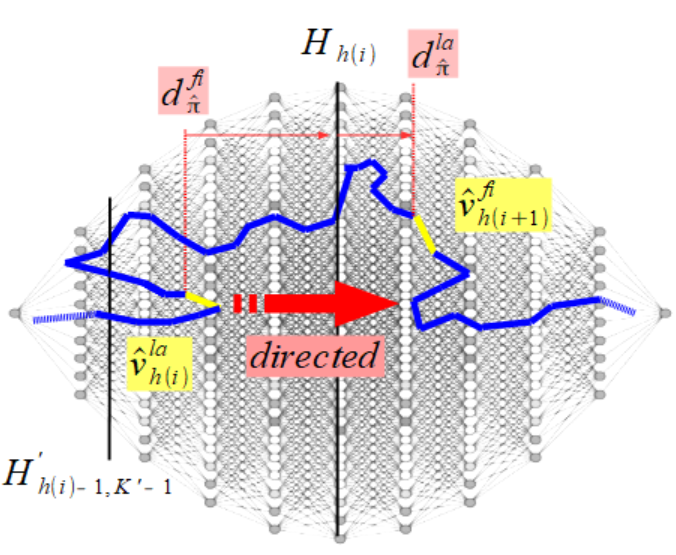}
    \caption{The common edges are separated by $c_{\hat \pi}(i)=1$.}
\label{case_uno}
\end{figure}

In this case, see Figure \ref{case_uno} above, a subpath connecting $\hat v_{h(i)}^{\text{la}}$ and $\hat v_{h(i+1)}^{\text{fi}}$, crosses $c_{\hat \pi(i)} = 1$ many $H'$-planes which is also an $H$-plane. Four subcases are possible:
\begin{itemize}
\item $d_{\hat \pi}^{\text{fi}}<\frac{0.77}{KK'}n$ and $d_{\hat \pi}^{\text{la}}<\frac{0.77}{KK'}n$, i.e. both vertices $\hat v_{h(i)}^{\text{la}}$ and $\hat v_{h(i+1)}^{\text{fi}}$ are in the (same) \textit{repulsive phase}: the substrand thus connects them in directed fashion, in  which case length and distance coincide, and
\beq\label{l3'}
l_{\hat \pi}(\hat v_{h(i)}^{\text{la}},\hat v_{h(i+1)}^{\text{fi}})=d(\hat v_{h(i)}^{\text{la}}, \hat v_{h(i+1)}^{\text{fi}})\leq 4d(\hat v_{h(i)}^{\text{la}}, \hat v_{h(i+1)}^{\text{fi}}).
\eeq
\item $d_{\hat \pi}^{\text{fi}} < \frac{0.77}{KK'}n$ and $d_{\hat \pi}^{\text{la}} \geq \frac{0.77}{KK'}n$. In this case:\\
\begin{itemize}
\item the vertex $\hat v_{h(i)}^{\text{la}}$ is in the repulsive phase (cfr. with the second subcase in the {\sf HH'}-regime above): in this first part of the journey, the substrand thus connects it with the H-plane in directed fashion, where again, and crucially, length and distance coincide.
\item as for the "rest of the journey", i.e. in order to deal with the length of the strand connecting H-plane and target vertex $\hat v_{h(i+1)}^{\text{fi}}$, we proceed exactly as in \eqref{l1}. \\
\end{itemize}
Splitting the substrand in first/second part of the journey, and then by these observations, we get
\beq\bea\label{l3}
l_{\hat \pi}(\hat v_{h(i)}^{\text{la}},\hat v_{h(i+1)}^{\text{fi}})&=l_{\hat \pi}(\hat v_{h(i)}^{\text{la}},H_{h(i)}\cap \hat \pi)+l_{\hat \pi}(H_{h(i)}\cap \hat \pi,\hat v_{h(i+1)}^{\text{fi}})\\
&\leq d_{\hat \pi}^{\text{fi}}+\frac{1.46}{KK'}n+\frac{0.46}{KK'}n\\
&\leq  4d_{\hat \pi}^{\text{fi}}+4 \frac{ 0.77}{KK'}n\\
&\stackrel{\bf{(T1)}}{\leq} 4d(\hat v_{h(i)}^{\text{la}}, \hat v_{h(i+1)}^{\text{fi}}).
\eea\eeq
\item $d_{\hat \pi}^{\text{fi}} \geq \frac{0.77}{KK'}n$ and $d_{\hat \pi}^{\text{la}} < \frac{0.77}{KK'}n$: this case is, by symmetry, equivalent to the previous.\\
\item $d_{\hat \pi}^{\text{fi}} \geq \frac{0.77}{KK'}n$ and $d_{\hat \pi}^{\text{la}}\geq \frac{0.77}{KK'}n$: both vertices being in the fuzzy zone,  we proceed exactly as in \eqref{estiml} to obtain 
\beq\bea\label{l4}
l_{\hat \pi}(\hat v_{h(i)}^{\text{la}},\hat v_{h(i+1)}^{\text{fi}})&\leq 2 \frac{1.46}{KK'}n+ 2\frac{0.46}{KK'}n\\
&\leq 4 \frac{0.77}{KK'}n+4 \frac{0.77}{KK'}n\\
&\stackrel{\bf{(T1)}}{\leq} 4d(\hat v_{h(i)}^{\text{la}}, \hat v_{h(i+1)}^{\text{fi}},
\eea\eeq
\end{itemize}
Claim \eqref{weeeird} thus holds true for all possible sub-scenarios of the third (and last) {\sf H}-case: this finishes the proof of Lemma \ref{distance}.
\end{proof}

\begin{proof} [Proof of Lemma \ref{r}]
We want now to estimate $f_{\pi}^{(s)}(n,k)$: Let $f^{(s)}_{l,\pi}(n,k)$ (respectively $f^{(s)}_{r,\pi}(n,k)$) the number of paths which are sharing $k$ edges with $\pi$ with at least one common edge betweeen $H_m$ and the middle of the hypercube ( respectively between the middle of the hypercube and $H_{K-m}$) but without considering first and last edge. It holds
\beq\bea\label{symtrick}
f^{(s)}_{\pi}(n,k)= f^{(s)}_{l,\pi}(n,k)+f^{(s)}_{r,\pi}(n,k)=2 f^{(s)}_{l,\pi}(n,k),
\eea\eeq
the last equality by symmetry (see \eqref{claimcomb'} ). Using \eqref{stimo}, \eqref{stima_due} and\eqref{stima_uno}, it clearly holds
\beq\bea\label{gainmil1}
f^{(s)}_{l,\pi}(n,k)&\lesssim  n^{\frac{2K+1}{2}} n^{K n^{\alpha}}\sum_{\boldsymbol k} \sum_{\boldsymbol l} \sum_{\boldsymbol \sigma} \sum_{\boldsymbol l'} \prod_{i=1}^{b}{\tanh\left(\mathsf E\frac{d(v_{h(i)}^{\text{fi}},v_{h(i)}^{\text{la}})-k_{h(i)}}{\mathsf L_{opt} n}\right)}^{d(v_{h(i)}^{\text{fi}},v_{h(i)}^{\text{la}})-k_{h(i)}}\\
&\hspace{2.5cm}\times{\cosh\left(\mathsf E\frac{ d(v_{h(i)}^{\text{fi}},v_{h(i)}^{\text{la}})-k_{h(i)} }{\mathsf L_{opt} n}\right)}^{n} {\left(\frac{\mathsf L_{opt} n}{e \mathsf E}\right)}^{d(v_{h(i)}^{\text{fi}},v_{h(i)}^{\text{la}})-k_{h(i)}}\\
&\hspace{1.5cm} \times \prod_{i=0}^{b} {\tanh\left(\frac{\hat{l_i} \mathsf E}{\mathsf L_{opt}n}\right)}^{d\left(\hat{v}_{h(i)}^{\text{la}},\hat{v}_{h(i+1)}^{\text{fi}}\right)}{\cosh\left(\frac{\hat{l_i} \mathsf E}{\mathsf L_{opt}n}\right)}^{n}{\left(\frac{\mathsf L_{opt}n}{\mathsf E e}\right)}^{\hat{l}_i}.\\
\eea\eeq
Using the monotonicity of the $\cosh$-function \eqref{e9}, and the fact that all paths in $\mathcal J$ have the same length\footnote{\label{foot} Recall from \eqref{a)''} that
$\sum_{i=0}^b \hat{l}_i +\sum_{i=1}^b d(v_{h(i)}^{\text{fi}},v_{h(i)}^{\text{la}})-k_{h(i)}  = \sf L_{opt} n -k,
$.}, in \eqref{gainmil1} yields 
\beq\bea\label{gainmil2}
f^{(s)}_{l,\pi}(n,k)&\lesssim  n^{\frac{2K+1}{2}} n^{K n^{\alpha}}\sum_{\boldsymbol k} \sum_{\boldsymbol l} \sum_{\boldsymbol \sigma} \sum_{\boldsymbol l'}{\cosh\left(\mathsf E\frac{\mathsf L_{opt} n-k }{\mathsf L_{opt}n}\right)}^{n}{\left(\frac{\mathsf L_{opt}n}{\mathsf E e}\right)}^{\mathsf L_{opt} n-k} \\
&\prod_{i=0}^{b} {\tanh\left(\frac{\hat{l_i} \mathsf E}{\mathsf L_{opt}n}\right)}^{d\left(\hat{v}_{h(i)}^{\text{la}},\hat{v}_{h(i+1)}^{\text{fi}}\right)}\prod_{i=1}^{b}{\tanh\left(\mathsf E\frac{d(v_{h(i)}^{\text{fi}},v_{h(i)}^{\text{la}})-k_{h(i)}}{\mathsf L_{opt} n}\right)}^{d(v_{h(i)}^{\text{fi}},v_{h(i)}^{\text{la}})-k_{h(i)}}.\\
\eea\eeq
Let $q\defi \min\{h(i)>m, k_{h(i)}>0\}$, splitting the product of the $\tanh$-terms according to $q$, we obtain 
\beq\bea\label{gainmil2'}
&\prod_{i=0}^{b} {\tanh\left(\frac{\mathsf E\hat{l_i}}{\mathsf L_{opt}n}\right)}^{d\left(\hat{v}_{h(i)}^{\text{la}},\hat{v}_{h(i+1)}^{\text{fi}}\right)}\prod_{i=1}^{b}{\tanh\left(\mathsf E\frac{d(v_{h(i)}^{\text{fi}},v_{h(i)}^{\text{la}})-k_{h(i)}}{\mathsf L_{opt} n}\right)}^{d(v_{h(i)}^{\text{fi}},v_{h(i)}^{\text{la}})-k_{h(i)}}\\
&= \prod_{i=0}^{q-1}{\tanh\left(\frac{ \mathsf E\hat{l_i}}{\mathsf L_{opt}n}\right)}^{d\left(\hat{v}_{h(i)}^{\text{la}},\hat{v}_{h(i+1)}^{\text{fi}}\right)}\prod_{i=1}^{q} {\tanh\left(\mathsf E\frac{d(v_{h(i)}^{\text{fi}},v_{h(i)}^{\text{la}})-k_{h(i)}}{\mathsf L_{opt} n}\right)}^{d(v_{h(i)}^{\text{fi}},v_{h(i)}^{\text{la}})-k_{h(i)}}\\
& \times \prod_{i=q}^{b} {\tanh\left(\frac{\mathsf E\hat{l_i}}{\mathsf L_{opt}n}\right)}^{d\left(\hat{v}_{h(i)}^{\text{la}},\hat{v}_{h(i+1)}^{\text{fi}}\right)} \prod_{i=q+1}^{b} {\tanh\left(\mathsf E\frac{d(v_{h(i)}^{\text{fi}},v_{h(i)}^{\text{la}})-k_{h(i)}}{\mathsf L_{opt} n}\right)}^{d(v_{h(i)}^{\text{fi}},v_{h(i)}^{\text{la}})-k_{h(i)}}\\
& \leq {\tanh\left(\mathsf E\frac{\widehat{L}_{q-1}+\mathcal D_q}{\mathsf L_{opt} n}\right)}^{\widehat{\mathcal D}_{q-1}+\mathcal D_q}\times{\tanh\left(\mathsf E\frac{\widehat{L}_{b}-\widehat{L}_{q-1}+\mathcal D_b-\mathcal D_q}{\mathsf L_{opt} n}\right)}^{\widehat{\mathcal D}_{b}-\widehat{\mathcal D}_{q-1}+\mathcal D_b-\mathcal D_q},
\eea\eeq
the last r.h.s using the monotonicity of the $\tanh$-terms \eqref{e8'''} two times: one time for the first line and a second time for the second line of the second equality. Putting \eqref{gainmil2'} into \eqref{gainmil2} yields
\beq\bea\label{gainmil3}
f^{(s)}_{l,\pi}(n,k)&\lesssim  n^{\frac{2K+1}{2}} n^{K n^{\alpha}}\sum_{\boldsymbol k} \sum_{\boldsymbol l} \sum_{\boldsymbol \sigma} \sum_{\boldsymbol l'} {\cosh\left(\mathsf E\frac{\mathsf L_{opt} n-k }{\mathsf L_{opt}n}\right)}^{n}{\left(\frac{\mathsf L_{opt}n}{\mathsf E e}\right)}^{\mathsf L_{opt} n-k}\\
&{\tanh\left(\mathsf E\frac{\widehat{L}_{q-1}+\mathcal D_q}{\mathsf L_{opt} n}\right)}^{\widehat{\mathcal D}_{q-1}+\mathcal D_q}\times{\tanh\left(\mathsf E\frac{\widehat{L}_{b}-\widehat{L}_{q-1}+\mathcal D_b-\mathcal D_q}{\mathsf L_{opt} n}\right)}^{\widehat{\mathcal D}_{b}-\widehat{\mathcal D}_{q-1}+\mathcal D_b-\mathcal D_q}.\\
\eea\eeq

We now claim that for $0<x\leq y\leq \mathsf E$,
\beq\bea\label{tricksmilieu}
\tanh(x)\leq \frac{3}{4} \tanh(x+y).
\eea\eeq
Indeed, using the addition formula for the $\tanh$ function, it holds
\beq\bea\label{tricksmilieu'}
\frac{\tanh(x)}{\tanh(x+y)}=\frac{\tanh(x)\left(1+\tanh(x)\tanh(y)\right)}{\tanh(x)+\tanh(y)}=\frac{1+\tanh(x)\tanh(y)}{1+\frac{\tanh(y)}{\tanh(x)}}\leq \frac{1+\tanh(\mathsf E)^2}{2}=\frac{3}{4},
\eea\eeq
the last inequality because the function $\tanh$ is increasing and the claim \eqref{tricksmilieu} is settled. 

Again using that $\tanh$ is increasing we also have that
\beq\bea\label{tricksmilieu1}
\tanh(y)\leq \tanh(x+y).
\eea\eeq
Using in  \eqref{gainmil3} the estimates \eqref{tricksmilieu} and \eqref{tricksmilieu1}  with
\beq
x\defi \min\{\widehat{L}_{q-1}+\mathcal D_q, \widehat{L}_{b}-\widehat{L}_{q-1}+\mathcal D_b-\mathcal D_q\},
\eeq
and
\beq
y\defi \max\{\widehat{L}_{q-1}+\mathcal D_q, \widehat{L}_{b}-\widehat{L}_{q-1}+\mathcal D_b-\mathcal D_q\},
\eeq
we obtain
\beq\bea\label{gainmil4}
f^{(s)}_{l,\pi}(n,k)&\lesssim  n^{\frac{2K+1}{2}} n^{K n^{\alpha}}\sum_{\boldsymbol k} \sum_{\boldsymbol l} \sum_{\boldsymbol \sigma} \sum_{\boldsymbol l'} \left(\frac{3}{4}\right)^{\min\{\widehat{\mathcal D}_{q-1}+\mathcal D_q,\widehat{\mathcal D}_{b}-\widehat{\mathcal D}_{q-1}+\mathcal D_b-\mathcal D_q\}}\\
& \hspace{2cm}{\tanh\left(\mathsf E\frac{\mathcal D_b+\widehat{L}_{b}}{\mathsf L_{opt} n}\right)}^{\mathcal D_b+\widehat{\mathcal D}_{b}}{\cosh\left(\mathsf E\frac{\mathsf L_{opt} n-k }{\mathsf L_{opt}n}\right)}^{n}{\left(\frac{\mathsf L_{opt}n}{\mathsf E e}\right)}^{\mathsf L_{opt} n-k}.\\
\eea\eeq
With the same line of reasoning as in \eqref{claim_facile}, we clearly have that
\beq\label{x1}
\widehat{\mathcal D}_{q-1}+\mathcal D_q \geq m \hat n_{K}-k,
\eeq
and 
\beq\label{x2}
\widehat{\mathcal D}_{b}-\widehat{\mathcal D}_{q-1}+\mathcal D_b-\mathcal D_q \geq \frac{n}{2}-k.
\eeq
Thus, it follows from \eqref{x1} and \eqref{x2} that
\beq\label{x}
\min\{\widehat{\mathcal D}_{q-1}+\mathcal D_q,\widehat{\mathcal D}_{b}-\widehat{\mathcal D}_{q-1}+\mathcal D_b-\mathcal D_q\}\geq m \hat n_{K}-k.
\eeq
Plugging \eqref{x} into \eqref{gainmil4} and recalling that paths in $ \mathcal J$ have the same, prescribed length (recall once more \eqref{a)''} or, which is the same, footnote \ref{foot}), it holds 
\beq\bea\label{gainmil5}
f^{(s)}_{l,\pi}(n,k)&\lesssim  n^{\frac{2K+1}{2}} n^{K n^{\alpha}}\sum_{\boldsymbol k} \sum_{\boldsymbol l} \sum_{\boldsymbol \sigma} \sum_{\boldsymbol l'} \; \Bigg[ \left(\frac{3}{4}\right)^{m \hat n_{K}-k} {\tanh\left(\mathsf E\frac{\mathsf L_{opt} n-k}{\mathsf L_{opt} n}\right)}^{\mathcal D_b+\widehat{\mathcal D}_{b}} \\
& \hspace{6cm} \times  {\cosh\left(\mathsf E\frac{\mathsf L_{opt} n-k }{\mathsf L_{opt}n}\right)}^{n}{\left(\frac{\mathsf L_{opt}n}{\mathsf E e}\right)}^{\mathsf L_{opt} n-k} \Bigg].\\
\eea\eeq
We follow {\it exactly} the same steps which from \eqref{stimo_due} lead to \eqref{estimate2'}, this time of course with the factor $\left(\frac{3}{4}\right)^{m \hat n_{K}-k}$. Omitting the details, we obtain
\beq\bea\label{gainmil6}
f^{(s)}_{l,\pi}(n,k)&\leq P_n n^{K n^{\alpha}}\left(\frac{3}{4}\right)^{m \hat n_{K}-k} {\tanh\left(\mathsf E\frac{\mathsf L_{opt} n-k}{\mathsf L_{opt} n}\right)}^{\max\left(n-k,\frac{\mathsf L_{opt} n-k}{4}\right)} \\
& \hspace{5cm}  \times {\cosh\left(\mathsf E\frac{ \mathsf L_{opt} n-k}{\mathsf L_{opt} n}\right)}^{n}{\left(\frac{\mathsf L_{opt} n}{\mathsf Ee}\right)}^{\mathsf L_{opt} n-k},
\eea\eeq
where $P_n$ is a finite degree polynomial. Combining  \eqref{symtrick} and \eqref{gainmil6} and the fact that for $k \leq 200 \hat n_{K}$, $\left(\frac{3}{4}\right)^{m \hat n_{K}-k}\leq \left(\frac{3}{4}\right)^{(m-200) \hat n_{K}}$ finishes the proof of Lemma \ref{r}.
\end{proof}

\section{Concentration of the optimal length: proof of Theorem \ref{stronger_statement}}\label{section_concentration}
Recall that claim \eqref{concentration} reads
\beq
\lim_{n\to \infty} \PP\left(  \#\left\{\pi \in \Pi_{n}:\;  X_{\pi}\leq E+\epsilon^2,  \frac{| l_\pi - \mathsf L n|}{n} \geq a \epsilon \right\} > 0 \right) = 0\,,
\eeq
for $a>0$ large enough. The proof, which is (vaguely) inspired by the {\it saddle point method} \cite{flajolet}, exploits the  strong concentration of the expansion of the $\sinh$-function on specific Taylor-terms. To see how this goes, in virtue of the by now "classical" route (union bounds and Markov's inequality / independence / tail estimates) it holds 
\beq\bea \label{van}
\PP\left(  \#\left\{\pi \in \Pi_{n}:\;  X_{\pi}\leq E+\epsilon^2,  \frac{| l_\pi - \mathsf L n|}{n} \geq a \epsilon \right\}  \right) & \lsim \sum_{\frac{|l - \mathsf L n|}{n}\geq a\epsilon} M_{n,l}\frac{(E+\e^2)^l}{l!}\,.
\eea\eeq 
Splitting the above sum 
\beq\bea\label{2sum}
\sum_{\frac{|l_\pi-\mathsf L n|}{n}\geq a\epsilon} M_{n,l}\frac{{\left(E+\epsilon^2\right)}^l}{l!}=\sum_{l=0}^{( \mathsf L-a\epsilon) n}M_{n,l}\frac{{\left(E+\epsilon^2\right)}^l}{l!}+\sum_{l=(\mathsf L+a\epsilon) n}^{\infty} M_{n,l}\frac{{\left(E+\epsilon^2\right)}^l}{l!}\,,
\eea\eeq 
we claim that both contributions vanish in the large-$n$ limit. 

Concerning the first sum, by Stanley's M-bound \eqref{sf}, and for {\it any} $x>0$, we have that
\beq\bea\label{4sum}
\sum_{l=0}^{(\mathsf L-a\epsilon) n}M_{n,l}\frac{{\left(E+\epsilon^2\right)}^l}{l!} \leq {\sinh(x)}^n\sum_{l=0}^{(\mathsf L-a\epsilon) n}{\left(\frac{E+\epsilon^2}{x}\right)}^l, 
\eea\eeq
We choose $x \defi E+\epsilon^2-\epsilon$, in which case the largest term in the above sum is given by $l= \mathsf L-a\epsilon$, and therefore
\beq\bea
(\ref{4sum})&\lsim {\sinh \left(E+\epsilon^2-\epsilon\right)}^n{\left(\frac{E+\epsilon^2}{E+\epsilon^2-\epsilon}\right)}^{(\mathsf L-a\epsilon) n} \times n \\
&=n \exp\left\{n \log \sinh(E+\epsilon^2-\epsilon)-(\mathsf L- a\epsilon)\log \left(1-\frac{\epsilon}{E+\epsilon^2}\right)\right\}\,.
\eea \eeq
To get a handle on the above exponent we proceed by Taylor expansions around $E$:
\beq \bea
\sinh(E+\epsilon^2-\epsilon)&=\sinh(E)+(\epsilon^2-\epsilon)\cosh(E)+(\epsilon^2-\epsilon)^2\frac{\sinh(E)}{2}+o(\e^2)\\
&=1+(\epsilon^2-\epsilon)\sqrt{2}+\frac{\epsilon^2}{2}+o(\e^2) \qquad (\e \downarrow 0).
\eea\eeq
Further using that $\log(1-x) = 1-x-\frac{x^2}{2}+o(x^2)$ for $x\downarrow 0$, we thus get
\beq \bea \label{Taylor}
& \log \sinh(E+\epsilon^2-\epsilon)-(\mathsf L- a\epsilon)\log \left(1-\frac{\epsilon}{E+\epsilon^2}\right)\\
&\qquad = (\e^2-\e)\sqrt{2}+\frac{\e^2}{2}-(\mathsf L -a\epsilon)\left(-\frac{\e}{E}-\frac{\e^2}{2E^2}\right)+o(\e^2) \\
&\qquad = \epsilon^2 \left(\frac{1}{2}+\sqrt2+\frac{1}{\sqrt2E}-\frac{a}{E} \right)+o(\e^2)\,, \\
\eea\eeq
for $\e \downarrow 0$. But the r.h.s. \eqref{Taylor} is clearly negative as soon as $a>\frac{E}{2}+\sqrt2E+\frac{1}{\sqrt2}$, implying that the first sum in \eqref{2sum} yields no contribution in the large-$n$ limit, as claimed. 

We proceed in full analogy for the second sum, but this time around via Stanley's M-bound with $x\defi E+\epsilon^2+\epsilon$: an elementary estimate of the ensuing geometric series yields
\beq\bea\label{sum0}
\sum_{l=(\mathsf L+a\epsilon) n}^{\infty}M_{n,l}\frac{{\left(E+\epsilon^2\right)}^l}{l!} &\lesssim {\sinh\left(E+\epsilon^2+\epsilon \right)}^n{\left(\frac{E+\epsilon^2}{E+\epsilon^2+\epsilon}\right)}^{(\mathsf L+a\epsilon) n} \frac{E+\epsilon^2+\epsilon}{\epsilon}\\
&\lsim \exp{ n\left\{\log \sinh(E+\epsilon^2+\epsilon) -(\mathsf L+a\epsilon)\log \left(1+\frac{\epsilon}{E+\epsilon^2}\right)\right\}}\,,
\eea \eeq
recalling in the last step the definition of $l_{\e, n} = \mathsf L +a \e$. Once again Taylor-expanding the exponent (around $E$) we get 
\beq \bea
& \log \sinh(\mathsf E+\epsilon^2+\epsilon) -(\mathsf L+a\epsilon)\log \left(1+\frac{\epsilon}{E+\epsilon^2}\right) = \epsilon^2\left(\frac{1}{2}+\sqrt2+\frac{1}{\sqrt2E}-\frac{a}{E}\right)+o(\epsilon^2),
\eea\eeq
for $\e\downarrow 0$: as this is also negative for $a>\frac{E}{2}+\sqrt2E+\frac{1}{\sqrt2}$, the second claim is also settled, and the proof of the Theorem \ref{stronger_statement} follows. \\

\hfill $\square$

\section{Appendix}
We give for completeness the short proof of Stanley's formula \eqref{e5}, which states that
\beq
{\sinh(x)}^{d}{\cosh(x)}^{n-d} =\sum_{l=0}^{\infty}M_{n,l,d}\frac{x^l}{l!}\,.
\eeq
Indeed, by the Binomial Theorem, it holds
\beq\bea
{\sinh(x)}^{d}{\cosh(x)}^{n-d}&=\frac{1}{2^n}{\left(e^{x}-e^{-x}\right)}^{d}{\left(e^{x}+e^{-x}\right)}^{n-d}\\
&=\frac{1}{2^n}\left(\sum_{j=0}^{d}\binom{d}{j}(-1)^j e^{(d-2j)x}\right)\left(\sum_{i=0}^{n-d}\binom{n-d}{i}e^{(n-d-2i)x}\right)\\
& = \frac{1}{2^n} \sum_{j=0}^{d}  \sum_{i=0}^{n-d} \binom{n-d}{i} \binom{d}{j}(-1)^j \exp\left( n-2(i+j) x \right)\,.
\eea \eeq
Taylor expanding the exponential function, we get that the r.h.s. above equals
\beq \bea
& \sum_{l=0}^{\infty}\frac{1}{2^n}\sum_{i=0}^{n-d}\sum_{j=0}^{d}\binom{d}{j}\binom{n-d}{i}{(-1)}^{j}{(n-2(i+j))}^{l}\frac{x^l}{l!}\\
& \qquad\qquad =  \sum_{l=0}^{\infty} \left\{ \frac{1}{2^n} \sum_{i'=j}^{n-d+j}\sum_{j=0}^{d}\binom{d}{j}\binom{n-d}{i'-j}{(-1)}^{j}{(n-2i')}^{l} \mathbbm{1}_{j\leq i'} \right\} \frac{x^l}{l!}\,,
\eea \eeq
the last step by the substitution $i' \hookrightarrow i+j$. By definition of the $M's$, Stanley's formula thus follows .

\hfill $\square$



\begin{thebibliography}{11}
\bibitem{Beresticky} Berestycki, Julien, Eric Brunet, and Zhan Shi. {\it The number of accessible paths in the hypercube.} Bernoulli 22.2 (2016): 653-680.
\bibitem{Beresticky2} Berestycki, Julien, Eric Brunet, and Zhan Shi. {Accessibility percolation with backsteps}, ALEA, Lat. Am. J. Probab. Math. Stat. 14 (2017): 45–62


\bibitem{Durrett} Durrett, Rick. {\it Lecture notes on particle systems and percolation.} Wadsworth, Belmont CA (1988)
\bibitem{Fill_Pemantle} Fill, James Allen, and Robin Pemantle. {\it Percolation, first-passage percolation and covering times for richardson's model on the $n$-Cube.} The Annals of Applied Probability (1993): 593-629.
\bibitem{flajolet} Flajolet, Philippe, and Robert Sedgewick. {\it Analytic Combinatorics}. Cambridge University Press (2009).
\bibitem{Hegarty} Hegarty, Peter, and Anders Martinsson. {\it On the existence of accessible paths in various models of fitness landscapes.} The Annals of Applied Probability 24.4 (2014): 1375-1395.
\bibitem{kistler} Kistler, Nicola. {\it Derrida's random energy models. From spin glasses to the extremes of correlated radom fields.} In: V. Gayrard and N. Kistler (Eds.) Correlated Random Systems: five different methods,  Springer Lecture Notes in Mathematics (2015): Vol. 2143.
\bibitem{kss} Kistler, Nicola, Adrien Schertzer and Marius A. Schmidt. {\em First passage percolation in the mean field limit}. Brazilian Journal of Probability and Statistics 34.2 (2020): 414-425.
\bibitem{kss1} Kistler, Nicola, Adrien Schertzer and Marius A. Schmidt. {\em First passage percolation in the mean field limit, 2. The extremal process}. The Annals of Applied Probability 30.2 (2020): 788-811.
\bibitem{Krug} Hwang, Sungmin, Benjamin Schmiegelt, Luca Ferretti, and Joachim Krug. {\it Universality classes of interaction structures for NK fitness landscapes.} Journal of Statistical Physics 172, no. 1 (2018): 226-278.
\bibitem{Krug2} Krug, Joachim. {\it Accessibility percolation in random fitness landscapes.} To appear
in {\sf Probabilistic Structures in Evolution}, ed. by E. Baake and A. Wakolbinger
\bibitem{Martinsson1} Martinsson, Anders. {\it Unoriented first-passage percolation on the $n$-cube.} The Annals of Applied Probability 26.5 (2016): 2597-2625.
\bibitem{Martinsson2} Martinsson, Anders. {\it First-passage percolation on Cartesian power graphs.} The Annals of Probability, 46.2 (2018): 1004-1041.
\bibitem{Martinsson3} Martinsson, Anders. {\it Accessibility percolation and first-passage site percolation on the unoriented binary hypercube}. Preprint arXiv:1501.02206 (2015)


\bibitem{stanley} Stanley, Richard P. {\em Algebraic Combinatorics}, Springer 20 (2013): 22.

\end{thebibliography}
\end{document}